\documentclass{amsart}
\usepackage[T1]{fontenc}
\usepackage{amsmath,a4wide}
\usepackage{amssymb}
\usepackage{amscd}
\usepackage{epsfig}
\usepackage[T1]{fontenc}
\usepackage{epsfig}
\usepackage{amsmath}
\usepackage{graphicx}
\usepackage{subfigure}
\usepackage{mathrsfs}
\usepackage{epstopdf}
\usepackage{color}
\usepackage{booktabs}
\usepackage{ifpdf}
\usepackage{cite}
\usepackage{amsthm}
\usepackage{amsmath}
\definecolor{myblue}{rgb}{0.2,0,0.9}
\definecolor{blue-violet}{rgb}{0.54, 0.17, 0.89}
\usepackage{enumitem}
\usepackage{algorithm,algorithmic}

\usepackage{pgfplots}
\usepgfplotslibrary{groupplots}
\pgfplotsset{compat=1.12}

\usepackage{tikz}
\usetikzlibrary{arrows,intersections}
\usepackage{boldline}
\usepackage{multirow}
\usepackage[margin=1.32in]{geometry}

\usepackage{varioref}
\usepackage{xr-hyper}
\usepackage{hyperref}
\hypersetup{
	colorlinks,linkcolor=blue,citecolor=blue,filecolor=red,urlcolor=blue}

\def\hatotimes{\mathbin{\hat \otimes}}

\definecolor{myblue}{rgb}{0.2,0,0.9}
\definecolor{blue_violet}{rgb}{0.54, 0.17, 0.89}
\definecolor{darkgreen}{rgb}{0,0.35,0}

\allowdisplaybreaks

\makeatletter
\DeclareRobustCommand*\cal{\@fontswitch\relax\mathcal}
\makeatother

\newtheorem{theorem}{Theorem}[section]
\newtheorem{proposition}[theorem]{Proposition}

\newtheorem{cor}[theorem]{Corollary}
\newtheorem{lemma}[theorem]{Lemma}
\numberwithin{equation}{section}

\theoremstyle{definition}
\newtheorem{remark}[theorem]{Remark}
\newtheorem{definition}[theorem]{Definition}
\newtheorem{assumption}[theorem]{Assumption}
\newtheorem{framework}[theorem]{Framework}

\usepackage{xparse}
\makeatletter
\RenewDocumentCommand{\title}{om}{%
	\IfNoValueTF{#1}
	{\gdef\shorttitle{}}
	{\gdef\shorttitle{#1}}%
	\gdef\@title{#2}%
}
\makeatother

\makeatletter
\def\@tocline#1#2#3#4#5#6#7{\relax
	\ifnum #1>\c@tocdepth %
	\else
	\par \addpenalty\@secpenalty\addvspace{#2}%
	\begingroup \hyphenpenalty\@M
	\@ifempty{#4}{%
		\@tempdima\csname r@tocindent\number#1\endcsname\relax
	}{%
		\@tempdima#4\relax
	}%
	\parindent\z@ \leftskip#3\relax \advance\leftskip\@tempdima\relax
	\rightskip\@pnumwidth plus4em \parfillskip-\@pnumwidth
	#5\leavevmode\hskip-\@tempdima
	\ifcase #1
	\or\or \hskip 2em \or \hskip 2em \else \hskip 3em \fi%
	#6\nobreak\relax
	\hfill\hbox to\@pnumwidth{\@tocpagenum{#7}}\par%
	\nobreak
	\endgroup
	\fi}
\makeatother

\title{Robust $Q$-learning for mean-field control under Wasserstein uncertainty in common noise}
\author{Mathieu Lauri{\`e}re}
\address{Shanghai Center for Data Science; NYU-ECNU Institute of Mathematical Sciences, NYU Shanghai}
\email{mathieu.lauriere@nyu.edu}

\author{Ariel Neufeld}
\address{Division of Mathematical Sciences, Nanyang Technological University}
\email{ariel.neufeld@ntu.edu.sg}

\author{Kyunghyun Park}
\address{Division of Mathematical Sciences, Nanyang Technological University}
\email{kyunghyun.park@ntu.edu.sg}

\thanks{\textit{Funding:} M.\;Lauri{\`e}re acknowledges the support of NYU Shanghai HPC for the numerical experiments. A.\;Neufeld gratefully acknowledges support by the MOE AcRF Tier 1 Grant   RG109/25. K.\;Park
acknowledges the support of the National Research Foundation of Korea (grant DOI: RS-2025-02633175).}

\date{\today.}

\begin{document}

\begin{abstract}
	In this article, we present a robust $Q$-learning algorithm for discrete-time mean-field control problems under Wasserstein uncertainty in the common noise law.
	The~algorithm combines a quantization-and-projection scheme with a Wasserstein dual reformulation on the common-noise space.
	We establish its convergence together with finite-time iteration bounds for both synchronous and asynchronous learning schemes.
	Numerical experiments on systemic risk and epidemic models compare the asynchronous implementation with an idealized Bellman iteration, illustrate the robustness--performance tradeoff under common-noise misspecification, and report the observed convergence behavior of the asynchronous $Q$-learning algorithm.
\end{abstract}

\maketitle

\section{Introduction}\label{sec:intro}
Mean-field control (MFC) \cite{bensoussan2013mean,carmona2018probabilistic,pham2017dynamic}, also referred to as McKean-Vlasov control, has become an exciting source of progress in the study of dynamic stochastic systems with large-population of cooperative agents. These agents take the same reward and transition functions, while being influenced by the states and actions of other agents through their empirical distributions. Under the notion of {mean-field interaction} \cite{lasry2007mean,Huang2006nashCertainty}, MFC problems admit a tractable reformulation: rather than analyzing a high-dimensional system of explicitly coupled agents, one solves a representative agent's optimization problem in which the coupling is reduced to the evolving population distribution. This paradigm has grown into an active interdisciplinary field attracting theoretically inclined mathematicians as well as engineers and social scientists.

Realistic engineering and economic applications often require a source of randomness that is common to all agents. Indeed, demand fluctuations in traffic systems, environmental disturbances in robotics, and liquidity shocks in financial markets provide  convincing evidence for the importance of incorporating common randomness on top of the idiosyncratic randomness affecting only individual agents (e.g., \cite{fouque2015mean,elamvazhuthi2019mean,gu2021mean,bauso2016opinion}).
The introduction of {\it common noise} substantially increases both the modeling and analytical complexity of MFC problems, since the resulting mean-field interaction itself becomes stochastic, yet its incorporation remains essential for capturing aggregate randomness and systemic effects arising in realistic large-population systems. This has led to a recent surge of works on MFC problems with common noise \cite{carmona2023model,motte2022mean,djete2022extended,djete2022mckean,pham2017dynamic,motte2023quantitative}.

Notably, two major challenges arise when implementing MFC problems with common noise in practice.
The first is that analytical solutions of MFC problems are rarely available, while the corresponding numerical methods typically suffer from high dimensionality. In particular, MFC problems with common noise can be formulated as stochastic control problems driven by stochastic flows of probability measures representing mean-field interaction dynamics,
and the corresponding optimality conditions take the form of Bellman equations or forward-backward equations defined on spaces of probability measures
(see, e.g., \cite{motte2022mean,carmona2023model,djete2022mckean,djete2022mckeanAOP}). Such equations are therefore intrinsically infinite-dimensional, making them analytically intractable and numerically challenging to solve.

The second is that the true law of the common noise is unknown. Model misspecification and estimation errors are thus unavoidable, particularly when the common-noise law must be inferred from limited or non-stationary observations of population measure flows. Crucially, such uncertainty is apparent in MFC problems with common noise, since the dynamics of the probability measure flow are themselves adapted to the common noise. Consequently, misspecification of the common-noise law may lead to inaccurate stochastic control formulations and unreliable solutions.

This article aims to address the aforementioned challenges by developing a robust $Q$-learning algorithm for MFC problems under common-noise uncertainty. $Q$-learning and its variants \cite{watkins1992q}, which apply stochastic approximation \cite{robbins1951stochastic} based on the Bellman optimality condition, are among the most widely used reinforcement learning paradigms. Robust $Q$-learning \cite{blanchet2023double,panaganti2022robust,zhang2020robust,roy2017reinforcement,liu2022distributionally,
	lu2025distributionally,
	neufeld2024robust,sester2024q}, a recent advancement of this paradigm, incorporates distributional robustness into the learning objective to address uncertainty and distributional shifts arising from limited real data. Building on our companion article \cite{lauriere2025robust} in which a theoretical framework covering the dynamic programming principle for discrete-time robust MFC problems under common-noise uncertainty is established,
we aim to develop in this paper the corresponding robust $Q$-learning algorithm and its convergence analysis.

In order to introduce the learning target for our robust $Q$-learning algorithm, we begin by briefly describing the robust MFC problem under Wasserstein uncertainty in common noise, considered in this work; see Section~\ref{sec:MFC_framework} for the precise formulation. Let ${\cal Q}$ denote a set of probability measures over both idiosyncratic and common noise, inducing Wasserstein uncertainty in the common noise.
Given an initial state $\xi^i\in S$, an action processes $(a^i_t)_{t\geq 0}\subseteq A$, and a probability measure $\mathbb{P}\in {\cal Q}$, the state of agent $i\in \mathbb{N}$ evolves according to
\begin{align}\label{eq:Intro_condMckVla}
	s_{t+1}^i:=\operatorname{F}(s_{t}^i,a^i_{t},\Lambda_t^i,\varepsilon_{t+1}^i,\varepsilon_{t+1}^0)\quad\mbox{for $t\geq 0$},\quad \mbox{with $s_{0}^i:=\xi^i$},
\end{align}
where $\operatorname{F}$ denotes the transition~function,  $\varepsilon_{t+1}^i\in E$ and $\varepsilon_{t+1}^0\in E^0$ denote the idiosyncratic and common noise, respectively, and $\Lambda^i_t$ denotes the conditional law of $(s^i_{t},a^i_{t})$ given $(\varepsilon_{1}^0,\dots,\varepsilon_{t}^0)$.

Here, the idiosyncratic noise process $(\varepsilon_t^i)_{t\geq 1}$ has the fixed law $p_\varepsilon\in {\cal P}(E)$, while the common noise process $(\varepsilon^0_{t})_{t\geq 1}$ is uncertain in the following sense: Fix $m\geq 0$, $q\in \mathbb{N}$. For any $t\geq 1$, the conditional law of $\varepsilon_{t}^0$ given $(\varepsilon_{1}^0,\dots,\varepsilon_{t-1}^0)$ under $\mathbb{P}$ belongs to
\begin{align}\label{eq:uncertainty}
	{\cal B}^0_{m,q}(\widehat{p}_{\varepsilon^0}):=\{p\in {\cal P}({E^0}): {W}_q(p,\widehat{p}_{\varepsilon^0})\leq m  \},
\end{align}
where
$\widehat{p}_{\varepsilon^0}$ is a fixed reference measure and ${W}_q$ denotes the Wasserstein $q$-distance (see \eqref{eq:wasserstein}).

In this setting, the robust MFC problem is then defined~by
\begin{align}\label{eq:Intro_robustMFC}
	\sup_{(a^i_t)_{t\geq 0}} \inf_{\mathbb{P}\in {\cal Q}} \mathbb{E}^{\mathbb{P}}\bigg[\sum_{t=0}^\infty \beta^tr(s^i_{t},a^i_{t},\Lambda^i_t)\bigg],
\end{align}
where $r$ denotes the one-step reward function and $\beta \in [0,1)$ is the discount factor.

Following MFC terminology, the conditional law of $s_t^i$ and that of $(s^i_t, a^i_t)$ given the common noise history are regarded as elements of a \emph{lifted space}---the space of probability measures of the state space and the state--action space, respectively. The value function of the problem \eqref{eq:Intro_robustMFC} thus lives on the lifted state space, and, via the dynamic programming principle established in \cite{lauriere2025robust}, satisfies a Hamilton--Jacobi--Bellman--Isaacs equation on this lifted space.

The learning target of our robust $Q$-learning algorithm is the optimal $Q$-function $Q^*$ defined on ${\cal P}(S)\times \Pi$, where ${\cal P}(S)$ denotes the lifted state space and $\Pi $ denotes the set of Markov kernels inducing randomized actions. More precisely, the optimal $Q$-function $Q^*$ is the unique fixed point of the following lifted state-action operator, defined by solving for every $(\mu,\pi)\in {\cal P}(S)\times \Pi$
\begin{align}\label{intro:Hoperator}
	Q^*(\mu,\pi)= \overline r(\mu \hatotimes \pi)+\beta\inf_{p\in {\cal B}^0_{m,q}(\widehat{p}_{\varepsilon^0})} \int_{E^0} \sup_{\pi' \in \Pi} Q^*\big(\overline{\operatorname{F}}(\mu\hatotimes\pi,e^0),\pi'\big) p (de^0),
\end{align}
where $\overline r$ and $\overline {\operatorname{F}}$  denote the lifted reward and lifted transition functions, induced by the one-step reward function $r$ in \eqref{eq:Intro_condMckVla} and transition function $\operatorname{F}$ in \eqref{eq:Intro_robustMFC}, respectively; see Definition \ref{dfn:lifted_compo} in Section \ref{sec:Qlearning} for the precise definitions of $\Pi$,  $\overline r$ and $\overline {\operatorname{F}}$.

We note that this robust formulation in \eqref{eq:Intro_condMckVla} and \eqref{eq:Intro_robustMFC} reduces to the non-robust MFC framework with common noise studied in \cite{carmona2023model,motte2022mean} when the radius $m$ in \eqref{eq:uncertainty} is set to be zero. In this non-robust regime, the optimal $Q$-function is the fixed point of a standard Bellman operator,  in which the nonlinear expectation over the set ${\cal B}^0_{m,q}(\widehat{p}_{\varepsilon^0})$ in \eqref{intro:Hoperator} is replaced by a linear expectation under the reference measure $\widehat{p}_{\varepsilon^0}$. We remark that dynamic programming principles on spaces of probability measures  have already appeared in existing non-robust discrete-time MFC frameworks  \cite{gu2023dynamic,gu2021mean,lauriere2016dynamic,pham2017dynamic,pham2016discrete}.  Notably, the corresponding reinforcement learning algorithms have been developed in \cite{gu2023dynamic,angiuli2026analysis,gu2021mean,carmona2023model}. The present article contributes to this growing literature by establishing, to the best of our knowledge, the first tabular robust $Q$-learning framework for MFC under Wasserstein uncertainty in common noise.

We now turn to the main contributions of this article. First, we develop a \textit{tabular} robust $Q$-learning algorithm for learning the optimal $Q$-function defined in \eqref{intro:Hoperator} (see also \eqref{eq:optimalQ}). Two nontrivial challenges arise in this endeavor. The first is a dimensionality issue: even though the original state and action spaces $(S,A)$ are assumed to be finite, the lifted state space ${\cal P}(S)$ and the associated kernel sets have infinite cardinality, precluding a direct tabular treatment. The second is a robust-evaluation issue: the law of the common noise is not assumed to be known exactly, whereas the Bellman target requires a worst-case expectation over the Wasserstein ball ${\cal B}^0_{m,q}(\widehat{p}_{\varepsilon^0})$. Together, these features make the incorporation of robustness into a sample-based learning procedure nontrivial.

To address these challenges, we introduce two key ingredients. The first is a quantization-and-projection scheme for the domain $({\cal P}(S),\Pi)$ of $Q^*$, yielding a finite-dimensional representation amenable to tabular learning. This approach is inspired by the non-robust MFC setting of \cite{carmona2023model}. The second ingredient is a Wasserstein dual reformulation on the common-noise space itself: applying the duality theory of the Wasserstein metric \cite{bartl2020computational,blanchet2019quantifying,gao2023distributionally,mohajerin2018data} reduces the worst-case expectation over ${\cal B}^0_{m,q}(\widehat{p}_{\varepsilon^0})$ to a one-dimensional optimization problem involving a single expectation under the reference law $\widehat{p}_{\varepsilon^0}$. Furthermore, the dual reformulation enables the use of independent common-noise samples within the asynchronous algorithm.
Together, these two ingredients give rise to a robust $Q$-learning algorithm that achieves distributional robustness without requiring direct access to the worst-case common-noise law.
We refer the reader to Section~\ref{sec:Qlearningalg} for the detailed description of our $Q$-learning algorithm design.

Second, we establish a non-asymptotic convergence analysis of our
robust $Q$-learning algorithm. Following the batch/offline $Q$-learning paradigm
of \cite{even2003learning}, the algorithm admits synchronous and asynchronous
variants according to the specification of the learning rate. The proof of the
convergence result is composed of two principal layers: the first concerns the
discretized optimal $Q$-function induced by the quantization-and-projection
scheme and the resulting discretization error; the second addresses the
stochastic approximation error. Theorem~\ref{thm:asynchro} establishes that
both variants converge to within the discretization error accuracy, where the
asynchronous case further requires standard conditions on sampled data. While
\cite[Theorem~1]{6796861} serves as a backbone, the estimates for the dual
optimizer and the iterative $Q$-function are newly established and, together with
a tailored application of the Wasserstein duality theorem, yield the desired
result.

Next, Theorem~\ref{thm:asynchro_rate} provides the finite-time iteration bound in terms of the algorithmic parameters. Namely, it establishes the number of iterations sufficient to achieve a target accuracy at a prescribed confidence level. As in the convergence proof for Theorem \ref{thm:asynchro}, the discretization error analysis is required; crucially, a concentration inequality for martingale sequences \cite[Exercise~9.2.4]{azuma1967weighted} is employed to obtain the explicit stochastic approximation error bound for the discretized $Q$-function. The verification of the requisite martingale structure relies on our Wasserstein duality result and the algorithmic design of the robust $Q$-learning. As a consequence of Theorem~\ref{thm:asynchro_rate}, we further derive the convergence rate with respect to the iteration number in Corollary~\ref{cor:Qlearningrate}.

Finally, Section~\ref{sec:numerics} complements the theoretical developments with numerical experiments on systemic risk and epidemic control. These examples compare the asynchronous implementation of Algorithm~\ref{alg:Qlearning} with an idealized finite-grid Bellman iteration, illustrate how moderate robustness can improve performance under common-noise misspecification, and report the observed convergence of the asynchronous $Q$-function iterates.

\subsection{Related literature}
Recent works have adapted reinforcement learning methods to discrete-time mean-field frameworks. For discrete-time mean-field control (MFC) and related mean-field Markov decision/control formulations, we refer to \cite{gast2011mean,gast2012mean,gu2021mean,motte2022mean,carmona2023model,gu2023dynamic,angiuli2026analysis}. We stress that MFC problems are different from mean-field games (Nash equilibria among infinitely many infinitesimal players) and from mean field type games (Nash equilibria between players solving MFC problems), which have both been the subject of reinforcement learning methods, see, e.g.,  \cite{subramanian2019reinforcement,guo2019learning,elie2020convergence,cui2021approximately,lauriere2022scalable,anahtarci2023q} and~\cite{carmona2020policy,sanjari2020optimal,carmona2021linear,shao2024reinforcement,jeloka2025learning}. In contrast, MFC do not require computing Nash equilibria. Moreover, in the continuous-time setting, reinforcement learning and deep learning approaches for mean-field control and games have also been investigated e.g. in \cite{firoozi2022exploratory,wei2024unified,frikha2025actor,wei2025continuous,ren2026continuousI,ren2026continuousII}. The recent works \cite{ren2026continuousI,ren2026continuousII} are close in their treatment of common noise, but they study continuous-time non-robust mean-field control; by contrast, the present paper develops a discrete-time tabular algorithm under Wasserstein uncertainty in the common-noise law. We further refer to the survey articles \cite{lauriere2021numerical,hu2023recent,Lauriere2024meanfield} for comprehensive overviews of existing numerical methods and machine learning approaches for mean-field control and game problems.

Robustness in mean-field game and control problems has been studied through risk-sensitive and min--max formulations, including \cite{tembine2013risk,bauso2016robust,moon2016robust,huang2017robust,langner2024markov,lauriere2025robust,liang2022robust,zaman2024robust,liang2026mean}. In continuous-time settings, robust mean-field control problems have also been studied, particularly in linear--quadratic frameworks with drift, disturbance, or volatility uncertainty \cite{huang2021social,wang2017social,wang2020social}. In parallel, we refer to \cite{bauerle2022distributionally,
	wiesemann2013robust,
	xu2012distributionally,li2026policy} for robust Markov decision processes (MDPs) and to \cite{blanchet2023double,panaganti2022robust,zhang2020robust,roy2017reinforcement,liu2022distributionally,
	lu2025distributionally,
	neufeld2024robust,sester2024q} for robust $Q$-learning algorithms. These robust MDPs and $Q$-learning works do not address the lifted state-action structure arising from mean-field interactions, while the robust mean-field works above do not provide the tabular Wasserstein robust $Q$-learning algorithm and finite-time iteration bound analysis developed here. Consequently, reinforcement-learning-based numerical methods for robust mean-field game and control problems remain comparatively underdeveloped.

Moving away from the above robust mean-field framework toward the perspective of offline $Q$-learning convergence and finite-time analysis, our convergence result in Theorem~\ref{thm:asynchro} is closely aligned with the classical convergence theory of $Q$-learning algorithms \cite{6796861,watkins1992q}, which originates from stochastic approximation theory \cite{dvoretsky1955stochastic,robbins1951stochastic}. Furthermore, the finite-time iteration bound analysis established in Theorem~\ref{thm:asynchro_rate} is in the spirit of \cite{even2003learning}, and is also related to recent non-asymptotic sample-complexity analyses for $Q$-learning algorithms developed in \cite{qu2020finite,li2020sample,li2024settling,kearns1998finite,li2024breaking}. Nevertheless, our setting involves several additional layers of difficulty beyond the classical framework. In particular, the lifted mean-field state-action structure requires the construction of suitable finite discretizations together with a corresponding approximation error analysis. Moreover, the incorporation of robustness through Wasserstein uncertainty leads to a significantly more delicate convergence analysis and finite-time error estimate, due to the additional min--max structure and distributional dependence of the problem.

\subsection{Outline of the article}
The paper is organized as follows. Section~\ref{sec:Qlearning} introduces the optimal $Q$-function together with its underlying spaces and key components, presents the robust $Q$-learning algorithm built on the discretization scheme and the common-noise Wasserstein duality theorem, and establishes the convergence and finite-time iteration bound analysis as well as the convergence rate with respect to the iteration number (Theorems~\ref{thm:asynchro} and~\ref{thm:asynchro_rate} and Corollary~\ref{cor:Qlearningrate}). Section~\ref{sec:numerics} presents the numerical setup, robustness profiles, and asynchronous $Q$-function convergence diagnostics for systemic-risk and epidemic-control examples. Section~\ref{sec:MFC} derives the optimal $Q$-function from the robust MFC problem under common noise uncertainty and introduces the discretized $Q$-function. Section~\ref{sec:proof} collects the key technical lemmas and the proofs of Theorems~\ref{thm:asynchro} and~\ref{thm:asynchro_rate} and Corollary~\ref{cor:Qlearningrate}.

\section{$Q$-learning algorithm for robust mean-field control}\label{sec:Qlearning}
\subsection{Optimal $Q$-function}
In this section, we introduce the optimal $Q$-function that serves as the learning target for solving the robust mean-field control problem under Wasserstein uncertainty in the common noise.

We start by introducing the underlying spaces and the corresponding probability spaces. Let $S$ and $A$ denote the state and action spaces, and let $E$ and $E^0$ denote the idiosyncratic and common noise spaces, respectively. Throughout this article, we assume that these spaces are nonempty finite subsets of (possibly different) Euclidean spaces, endowed with the Euclidean norm $|\cdot|$.  For any such space, or any finite product of such spaces $X$, we write $\overline X:={\cal P}(X)$ for the set of all Borel probability measures on $X$.

We endow $\overline X$ with the topology induced by the 1-Wasserstein distance, which we recall to be the following: For any $\mu,\nu\in \overline X$, we write $\operatorname{Cpl}(\mu,\nu)\subset \overline {X\times X}$ for the subset of probability measures on $X\times X$ with marginals $\mu$ and $\nu$, called couplings. For~$q\in\mathbb{N}$, the $q$-Wasserstein distance between $\mu$ and $\nu$ is defined by
\begin{align}\label{eq:wasserstein}
	W_q(\mu,\nu):=\bigg(\inf_{\gamma \in \operatorname{Cpl}(\mu,\nu)}\int_{X\times X}|x-y|^q\gamma(dx,dy)\bigg)^{1/q}.
\end{align}
Since $X$ is finite, for any $q\in \mathbb{N}$  the topology induced by $W_q$ coincides with that of weak convergence; see, e.g., \cite[Chapter 6]{villani2008optimal}.

Let us introduce key components used to define the optimal $Q$-function. To this end, we recall from \eqref{eq:Intro_condMckVla} and \eqref{eq:Intro_robustMFC} in Section \ref{sec:intro} the transition function $\operatorname{F}:S\times A\times \overline{S\times A}\times E\times E^0 \to S$, the one-step reward function $r:S\times A\times \overline{S\times A}\to \mathbb{R}$, the discount factor $\beta \in [0,1)$, and the fixed law of the idiosyncratic noise $p_\varepsilon \in \overline E$.
\begin{definition}\label{dfn:lifted_compo}
	\begin{itemize}
		\item [(i)] Define $\overline {\operatorname{F}}: \overline{S\times A}\times E^0\ni(\Lambda,e^0)\mapsto \overline {\operatorname{F}}(\Lambda,e^0)\in \overline S$ %
		      by
		      \[
			      \overline {\operatorname{F}}(\Lambda,e^0)(ds'):=\big((\Lambda\otimes p_\varepsilon)\circ \operatorname{F}(\cdot,\cdot,\Lambda,\cdot,e^0)^{-1}\big)(ds'),
		      \]
		      i.e., the push-forward of the product measure $\Lambda\otimes p_\varepsilon:=\Lambda(ds,da) p_\varepsilon (de) \in \overline{S\times A\times E}$ by the mapping $\operatorname{F}(\cdot,\cdot,\Lambda,\cdot,e^0):S\times A\times E\to S$.
		\item [(ii)] Define $\overline r: \overline{S\times A}\ni \Lambda\mapsto \overline r(\Lambda)\in \mathbb{R}$ by
		      \[
			      \overline{r}(\Lambda):=\int_{S\times A}r(s,a,\Lambda)\Lambda(ds,da).
		      \]
		\item [(iii)] Let \(
		      \Pi:=\{\pi:S\ni s\mapsto \pi(\cdot|s)\in \overline A\}\)
		      denote the policy set  (i.e., the set of Markov kernels), endowed with the metric
		      \begin{equation*}
			      d_{\Pi}(\pi,\pi'):= \max_{s\in S} W_1(\pi(da|s),\pi'(da|s))\quad\mbox{for $\pi,\pi'\in \Pi$}.
		      \end{equation*}
	\end{itemize}
\end{definition}
In mean-field control frameworks \cite{gu2023dynamic,motte2022mean,carmona2023model,lauriere2025robust}, the space $\overline S$ and $\overline{S \times A}$ are referred to as the lifted state and state-action spaces, respectively. The mappings $\overline{\operatorname{F}}$ and $\overline{r}$ define the corresponding lifted transition and reward functions on the lifted space. Moreover, the set ${\Pi}$ denotes the admissible policy class and serves as the domain of maximization in the Bellman--Isaacs operator characterized by $\overline{\operatorname{F}}$ and $\overline{r}$, which we introduce in \eqref{eq:bellmanisaacs}.

The optimal $Q$-function is derived from a fixed-point theorem of the Bellman--Isaacs operator, which is established under the following conditions on the transition function $\operatorname{F}$, the reward function $r$, and the discount factor $\beta$.
\begin{assumption}\label{as:MFC} There are some constants $C_{\operatorname{F}}>0$ and $C_{{r}}>0$ such that
	\begin{itemize}
		\item [(i)] for every $(s,a,\Lambda,e^0)$ and $(\tilde s,\tilde a,\tilde\Lambda,\tilde e^0)\in S \times A \times \overline{S\times A}  \times E^0$
		      \begin{align*}
			      \quad \quad & \int_{E}|\operatorname{F}(s,a,\Lambda,e,e^0)-\operatorname{F}(\tilde s,\tilde a,\tilde\Lambda,e,\tilde e^0)|p_\varepsilon(de) \\
			                  & \quad \leq C_{{\operatorname{F}}} \big(|s-\tilde s|+|a-\tilde a|+ W_{1}(\Lambda,\tilde \Lambda)+|e^0-\tilde e^0|\big),
		      \end{align*}
		\item [(ii)] for every $(s,a,\Lambda)$ and $(\tilde s,\tilde a,\tilde \Lambda)\in S\times A \times \overline{S\times A}$
		      \begin{align*}
			      |r (s,a,\Lambda)-r (\tilde s,\tilde a,\tilde \Lambda)|\leq C_{{r}} \big(|s-\tilde s|+|a-\tilde a|+W_1(\Lambda,\tilde \Lambda)\big).
		      \end{align*}
		\item [(iii)] $\beta $ is in $[0,1\wedge (2 C_{{\operatorname{F}}})^{-1})$.
	\end{itemize}
\end{assumption}
\begin{remark}\label{rem:regular_lift}
	Since $S$ and $A$ are finite and $\overline {S\times A}$ is compact, Assumption \ref{as:MFC}\;(ii) implies that%
	\[
		C_{r,\infty}:=\sup_{(s,a,\Lambda)\in S\times A\times \overline{S\times A}}|r(s,a,\Lambda)|<\infty.
	\]
	Moreover, we have by \cite[Lemma 5.2]{lauriere2025robust} that under Assumption \ref{as:MFC}\;(i),\;(ii), the following properties on the lifted components $\overline{\operatorname{F}}$ and $\overline{r}$ in Definition \ref{dfn:lifted_compo} hold:
	\begin{itemize}
		\item [(i)]  for every $(\Lambda ,e^0),(\tilde \Lambda,\tilde e^0)\in \overline{S\times A}\times E^0$, \[W_1(\overline{\operatorname{F}}(\Lambda, e^0),\overline{\operatorname{F}}(\tilde \Lambda,\tilde e^0))\leq C_{\operatorname{F}}(2W_1(\Lambda,\tilde \Lambda)+|e^0-\tilde e^0|),\]
		\item [(ii)] for every $\Lambda, \tilde \Lambda \in \overline{S\times A}$,
		      \[
			      |\overline r(\Lambda)|\leq  C_{r,\infty}\quad \mbox{and}\quad |\overline r (\Lambda)- \overline r(\tilde \Lambda) |\leq  2 C_r.
		      \]
	\end{itemize}
\end{remark}

To introduce the Bellman--Isaacs operator  ${\cal T}$, let $C_b(\overline{S})$ denote the set of bounded, continuous functions $v:\overline{S}\to \mathbb{R}$, endowed with the supremum norm
\(
\|v\|_{\overline{S}} := \sup_{\mu \in \overline{S}} |v(\mu)|.
\)
Moreover, for any $L \geq 0$, let $\operatorname{Lip}_{b,L}(\overline{S}) \subset C_b(\overline{S})$ denote the set of bounded, $L$-Lipschitz continuous functions.

Using the components introduced in Definition~\ref{dfn:lifted_compo} together with the uncertainty set ${\cal B}^0_{m,q}(\widehat{p}_{\varepsilon^0})$ defined in~\eqref{eq:uncertainty}, we define the Bellman--Isaacs operator ${\cal T}$ by setting, for any $v \in C_b(\overline{S})$,
\begin{equation}\label{eq:bellmanisaacs}
	{\cal T}v(\mu):= \sup_{\pi\in \Pi}\bigg\{\overline{r}(\mu\hatotimes\pi)+\beta\inf_{p\in {\cal B}^0_{m,q}(\widehat{p}_{\varepsilon^0})} \int_{E^0} v(\overline{\operatorname{F}}(\mu\hatotimes\pi,e^0)) p(de^0)
	\bigg\},\quad \mbox{$\mu\in \overline{S}$},
\end{equation}
where $\mu \hatotimes\pi(ds,da):= \pi(da|s)\mu(ds)\in \overline{S\times A}$.%

We now state the fixed-point theorem associated with the operator ${\cal T}$ in~\eqref{eq:bellmanisaacs}, which follows from \cite[Propositions~2.15, 2.16]{lauriere2025robust}.
\begin{proposition}
	\label{pro:LNP2025} Suppose that Assumption~\ref{as:MFC} is satisfied, and let $L^*:= 2C_r/{(1-2\beta C_{\operatorname{F}})}$. Then there exists a unique fixed point $v^*\in \operatorname{Lip}_{b,L^*}(\overline{S})$ of~${\cal T}$ such that $v^*={\cal T}v^*$. Moreover, for any $v\in \operatorname{Lip}_{b,L^*}(\overline{S})$, $\lim_{n\to \infty}{\cal T}^nv=v^*$.
\end{proposition}

Using $v^*$ from Proposition \ref{pro:LNP2025}, we define the optimal $Q$-function $Q^*:\overline S\times \Pi \to \mathbb{R}$ by%
\begin{align}\label{eq:optimalQ}
	{Q}^*(\mu, \pi):= \overline r(\mu \hatotimes \pi)+\beta\inf_{p\in {\cal B}^0_{m,q}(\widehat{p}_{\varepsilon^0})} \int_{E^0} v^*(\overline{\operatorname{F}}(\mu\hatotimes\pi,e^0)) p (de^0),\quad (\mu,\pi)\in\overline{S}\times \Pi.
\end{align}
It is {optimal} in the sense that
\(
\sup_{\pi \in \Pi} Q^*(\mu,\pi)= v^*(\mu)\) for any $\mu\in \overline{S}$.

The optimal $Q$-function in \eqref{eq:optimalQ} can be viewed as the state--action value function of a {\it lifted} robust Markov decision process (MDP) under model uncertainty, induced by the robust mean-field control problem under Wasserstein uncertainty in the common noise. This connection will be presented in detail in Section~\ref{sec:MFC}.

\subsection{$Q$-learning algorithm}\label{sec:Qlearningalg}
Our goal is to design a {\it tabular} $Q$-learning algorithm to learn the optimal $Q$-function $Q^*$ defined in~\eqref{eq:optimalQ}. We emphasize that although the state and action spaces $(S,A)$ are finite, the corresponding lifted spaces of measures $\overline S$ and kernels $\Pi$ are not.
Therefore, from a numerical view point,  two key ingredients are required: (i) quantization and projection of the domain $(\overline S,\Pi)$ of $Q^*$ and (ii) computation of worst-case expectations over the Wasserstein uncertainty set ${\cal B}_{m,q}^0(\widehat p_{\varepsilon^0})$ in $Q^*$. These two enable a finite-dimensional approximation of $Q^*$: the discretization of $(\overline S,\Pi)$ enables a tractable tabular representation of $Q^*$, while the evaluation of worst-case expectations ensures that the robustness with respect to the set ${\cal B}^0_{m,q}(\widehat p_{\varepsilon^0})$ is properly incorporated into the learning procedure.

We begin by introducing the quantization and projection for $(\overline S,\Pi)$, following the approach in Section~5.3 of \cite{carmona2023model}.
\begin{definition}\label{dfn:projectedmaps} Let $\check S\subset \overline{S}$ %
	be a {\it finite} subset such that for any $\mu\in \overline {S}$,
	\[
		\min_{\check\mu\in \check S}{W}_1(\mu,\check\mu)\leq \varepsilon_{\check S}\quad\mbox{for some $\varepsilon_{\check S}>0$},
	\]
	representing an $\varepsilon_{\check S}$-net of $\overline{S}$ under $W_1$. Analogously, let $\check A\subset \overline{A}$ be a finite subset such that for any $\nu\in \overline {A}$,
	\(
	\min_{\check\nu\in \check A}{W}_1(\nu,\check\nu)\leq \varepsilon_{\check A}\) for some~$\varepsilon_{\check A}>0$.
	\begin{itemize}
		\item [(i)] Let ${\operatorname{pj}}_{\check S}:\overline{S}\ni\mu \mapsto \operatorname{pj}_{\check S}(\mu)\in \check S$ be a mapping satisfying
		      \[
			      {W}_1(\mu,\operatorname{pj}_{\check S}(\mu))=\min_{\check\mu \in \check S}{W}_1(\mu,\check\mu).
		      \]
		\item [(ii)]
		      Define
		      \(
		      \check \Pi:=\{\check\pi:S\ni s\mapsto \check\pi(da|s)\in \check A\}
		      \)
		      so that it is a finite subset of $\Pi$; see Definition~\ref{dfn:lifted_compo}\;(iii). %
		      Then let $\operatorname{pj_{\check \Pi}}:\Pi\ni \pi \mapsto \operatorname{pj_{\check \Pi}}(\pi)\in \check \Pi$ be a mapping~satisfying
		      \[
			      d_\Pi (\pi, \operatorname{pj}_{\check{{\Pi}}}(\pi))= \min_{\check \pi \in \check \Pi} {d}_\Pi(\pi,\check\pi).
		      \]
	\end{itemize}
\end{definition}
\begin{remark}
	\begin{itemize}
		\item [(i)]
		      We present a construction of $\check S$ below; an analogous construction applies to $\check A$.
		      Denote $S:=\{s_1,\dots,s_{|S|}\}$ and let $\Delta_S:=\max_{s\neq s' \in S}|s-s'|<\infty$. For $k\in \mathbb{N}$, define
		      \[
			      \check S:= \bigg\{\mu \in \overline{S}:\mu(\{s_i\})=\frac{b_i}{k},\;b_i\in \mathbb{N}\cup \{0\}\;\mbox{$\forall i=1,\dots,|S|$},\;\sum_{i=1}^{|S|} b_i =k\bigg\},
		      \]
		      which is finite with $|\check S|=\binom{k+|S|-1}{|S|-1}$.

		      For any $\mu \in \overline{S}$, set $\ell:=k- \sum_{i=1}^{|S|} \lfloor k \mu(\{s_i\}) \rfloor $ and define $\check \mu' \in \check S$  by %
		      \[
			      \check\mu'(\{s_i\}):=\frac{\lfloor k \mu(\{s_i\}) \rfloor  +1}{k}\quad \mbox{if}\;i=1,\dots,\ell;\quad \check\mu'(\{s_i\}):=\frac{\lfloor k \mu(\{s_i\}) \rfloor}{k}\quad \mbox{else}.
		      \]
		      Then $|\mu(\{s_i\})-\check\mu'(\{s_i\})|\leq \frac{1}{k}$ for all $i=1,\dots,|S|$. Hence by \cite[Theorem 6.15]{villani2008optimal},
		      \[
			      W_1(\mu,\check \mu')\leq \Delta_S \| \mu -\check \mu'\|_{\operatorname{TV}}= \frac{\Delta_S}{2}\sum_{i=1}^{|S|}|\mu(\{s_i\})-\check\mu'(\{s_i\})|\leq \frac{\Delta_S|S|}{2k},
		      \]
		      where $\| \cdot\|_{\operatorname{TV}}$ denotes the total variation and the equality holds because $S$ is finite. Thus, choosing $k\geq \lceil \frac{\Delta_S|S|}{2\varepsilon_{\check S}} \rceil$ ensures $\min_{\check\mu\in \check S}{W}_1(\mu,\check\mu)\leq \varepsilon_{\check S}$.

		      Finally, once the construction of $\check A$ is obtained as above, the set $\check \Pi$ is naturally induced as the set of mappings from $S$ to $\check A$, and hence satisfies $|\check \Pi| = |\check A|^{|S|}$.
		\item [(ii)] %
		      A construction of the measurable mapping $\operatorname{pj}_{\check S}$ in Definition \ref{dfn:projectedmaps}\;(i)  is given as follows.
		      Fix an arbitrary ordering $\check S=\{\check\mu_1,\dots,\check\mu_{N}\}$ for some $N\in \mathbb{N}$. For any $\mu \in \overline{S}$, define%
		      \begin{align*}
			      \quad \quad \mathrm{pj}_{\check S}(\mu)
			      := \check\mu_{i^{*,\mu}},\quad \mbox{where } i^{*,\mu}
			      := \min\Big\{i \in\{1,\dots,N\} :
			      W_1(\mu,\check\mu_i)
			      = \min_{1\le \ell\le N} W_1(\mu,\check\mu_\ell) \Big\},
		      \end{align*}
		      i.e., the nearest element of
		      $\check S$, with tie-breaking according to
		      the fixed ordering.

		      Since $\check\Pi$ in Definition \ref{dfn:projectedmaps}\;(ii) is finite, we can construct the mapping
		      $\operatorname{pj}_{\check\Pi}$ analogously by replacing
		      $(\overline S,\check S;W_1)$ with $(\Pi,\check \Pi;d_\Pi)$.
	\end{itemize}
\end{remark}

Next, to enable the tractable computation of worst-case expectations over the set ${\cal B}^0_{m,q}(\widehat p_{\varepsilon^0})$, we exploit a convex duality representation that expresses worst-case expectations in terms of linear expectations with respect to the reference measure $\widehat p_{\varepsilon^0}$,  as established, e.g., in \cite{bartl2020computational,blanchet2019quantifying,gao2023distributionally,mohajerin2018data}. To this end, we recall the notion of $\lambda c $-transform on the common noise space $E^0$; see Section~2 in \cite{bartl2020computational}, Section~5 in \cite{villani2008optimal}.

\begin{definition}\label{dfn:lctransf}
	For $f:E^0\to \mathbb{R}$ and %
	$\lambda \geq 0$, let $(f)^{\lambda }:E^0\ni e^0\mapsto (f)^{\lambda }(e^0)\in\mathbb{R}$ denote the  $\lambda c$-transform of~$f$ with the cost function $
		c:E^0\times E^0\ni(e^0,\tilde e^0)\mapsto c(e^0,\tilde e^0):=|e^0-\tilde e^0|^q$ defined by
	\begin{align}\label{eq:lctransform}
		(f)^{\lambda }(e^0):=\max_{\tilde e^0 \in E^0} \{f(\tilde e^0) - \lambda |e^0-\tilde e^0|^q\}.
	\end{align}
\end{definition}

We next recall the convex duality result established, e.g., in Theorem~2.4 of \cite{bartl2020computational}.
\begin{lemma}\label{lem:lctransform}
	For any mapping $f:E^0\to \mathbb{R}$, the following convex duality holds:
	\[
		\inf_{p\in {\cal B}^0_{m,q}(\widehat p_{\varepsilon^0})}\int_{E^0} f (e^0)p(de^0)= \sup_{\lambda \geq 0}\bigg\{\int_{E^0}\big(-(-f)^{\lambda }(e^0)\big)\widehat p_{\varepsilon^0}(de^0) -m^q \lambda\bigg\}.
	\]
\end{lemma}

Using the quantization and projection in Definition~\ref{dfn:projectedmaps}, together with the convex duality result in Lemma~\ref{lem:lctransform}, we present a tabular $Q$-learning algorithm under both the synchronous and asynchronous frameworks of \cite{even2003learning}.
\begin{framework}\label{frame:Qlearning} %
	Let $(\Omega^0,{\mathcal F}^0,\widehat{\mathbb{P}}^0)$ be a probability space supporting a family
	of  independent and identically distributed random variables \((\varepsilon^0_{t,(\check\mu,\check\pi)})_{
			t\ge 1,
			(\check\mu,\check\pi)\in\check S\times\check\Pi
		}\subseteq E^0,
	\)
	with law $\widehat{p}_{\varepsilon^0}$. Both synchronous and asynchronous frameworks share the $Q$-learning update rule in (iii) and differ only in the choice of learning rates specified in (i) and (ii). Fix $w \in (\frac{1}{2},1)$.
	\begin{itemize}
		\item [(i)] (Synchronous learning rate) Define the synchronous learning rates by $\alpha_{t ,(\check \mu, \check \pi)}:=(t +1)^{-w}$ for every $t \geq 0$ and $(\check \mu,\check \pi)\in \check S\times \check \Pi$.
		\item [(ii)] (Asynchronous learning rate) Let $(\check \mu_t,\check \pi_t)_{t\geq 0} \subseteq \check S \times \check \Pi$ be a pre-sampled, projected dataset. %
		      Then we define for every \((\check \mu,\check \pi)\in \check S\times \check \Pi\)
		      \begin{align}\label{eq:visiting_time}
			      T_{(\check \mu, \check \pi)}:= \{t\geq 0: (\check \mu,\check \pi)=(\check \mu_t, \check \pi_t)\},
		      \end{align}
		      representing the set of times $t\geq 0$ at which $(\check \mu,\check \pi)$ is visited along $(\check \mu_t, \check \pi_t)_{t\geq 0}$. Finally, we define the asynchronous learning rates as follows: for every $t\geq 0$ and $(\check \mu,\check \pi)\in \check S\times \check \Pi$,
		      \begin{align}\label{eq:learningrate}
			      \alpha_{t,(\check \mu, \check \pi)}:=({N}_{t,(\check \mu,\check \pi)}+1)^{-w}\quad \mbox{if $t\in T_{(\check \mu, \check \pi)}$};\quad \alpha_{t,(\check \mu, \check \pi)}(\check \mu, \check \pi):=0\quad \mbox{else},
		      \end{align}
		      where \(
		      {N}_{t,(\check \mu,\check \pi)}%
		      \)
		      denotes the number of times $0\leq \ell<t$ at which $(\check \mu,\check \pi)=(\check \mu_\ell, \check \pi_\ell)$, with the convention that ${N}_{0,(\check \mu,\check \pi)}=0$.
		\item [(iii)] %
		      Under either choice of learning rates in (i) or (ii),  the $Q$-learning algorithm is defined by the following stochastic iteration: %
		      For every $t\geq 0$ and every pair $(\check \mu,\check\pi)\in \check S\times \check \Pi$, %
		      \begin{align}\label{eq:asynchro_Q0}
			      \check{Q}_0(\check \mu, \check \pi):=     & \check Q_0,\quad \mbox{for some $\check Q_0\in \mathbb{R}$},                                                                                                                                                                                                                                       \\
			      \check{Q}_{t+1}(\check \mu, \check \pi):= & (1-\alpha_{t,(\check \mu, \check \pi)})\check{Q}_{t}(\check \mu, \check \pi)
			      \label{eq:asynchro_Q1}                                                                                                                                                                                                                                                                                                                         \\
			                                                & \nonumber +\alpha_{t,(\check \mu, \check \pi)}\Big(\overline r(\check \mu\hatotimes\check \pi) +\beta \Big[-\big(-J_{t,(\check \mu, \check \pi)} \big)^{\lambda_{t,(\check \mu, \check \pi)}^*}(\varepsilon^0_{t+1,(\check\mu,\check\pi)})-m^q \lambda_{t,(\check \mu, \check \pi)}^* \Big] \Big),
		      \end{align}
		      where $J_{t,(\check \mu, \check \pi)}:E^0\ni e^0\mapsto J_{t,(\check \mu, \check \pi)}(e^0)\in\mathbb{R}$ is defined by
		      \begin{align}\label{eq:defJ}
			      J_{t,(\check \mu, \check \pi)}(e^0):= \max_{\check\pi'\in \check \Pi}\check {Q}_t\big(\operatorname{pj}_{\check S}\big(\overline{\operatorname{F}}(\check \mu\hatotimes\check \pi,e^0)\big),\check\pi'\big),
		      \end{align}
		      where $\lambda_{t,(\check \mu, \check \pi)}^*$ appearing in \eqref{eq:asynchro_Q1} is an optimizer~of the dual problem $\Phi_{t,(\check\mu,\check \pi)}$ defined by %
		      \begin{align}\label{eq:defconjugateJ}
			      \begin{aligned}
				       &
				      \Phi_{t,(\check\mu,\check \pi)}:= \sup_{\lambda \geq 0}\phi_{t,(\check\mu,\check \pi)}(\lambda),                                                                                       \\
				       & \phi_{t,(\check\mu,\check \pi)}(\lambda):=\int_{E^0}\Big(-(-J_{t,(\check \mu, \check \pi)})^{\lambda }(e^0)\Big)\widehat p_{\varepsilon^0}(de^0) -m^q \lambda,\quad \lambda \geq 0. %
			      \end{aligned}
		      \end{align}
	\end{itemize}
\end{framework}

\begin{remark}\label{rem:dual_Phi} %
	For the synchronous $Q$-learning algorithm under Framework \ref{frame:Qlearning}\;(i)~and~(iii), the update $\check Q_{t+1}$ in~\eqref{eq:asynchro_Q1} occurs at all pairs in $\check S\times\check\Pi$. In contrast, under the asynchronous $Q$-learning algorithm under Framework~\ref{frame:Qlearning}\;(ii)~and~(iii), $\check Q_{t+1}$ is updated only for the visited pair $(\check\mu_t,\check\pi_t)$ at time $t$, since $\alpha_{t,(\cdot, \cdot)}$ is nonzero only in this case; see \eqref{eq:learningrate}. These update rules are explicitly illustrated in Algorithm \ref{alg:Qlearning}, which presents the overall iteration in chronological order.

	The following statements hold for either choice of learning rates. For $t\geq 0$ and $(\check \mu,\check\pi)\in \check S\times \check \Pi$,
	\begin{itemize}
		\item [(i)] $\lambda_{t,(\check \mu,\check \pi)}^*$ in \eqref{eq:asynchro_Q1} is well-defined under certain conditions on $\beta$ and $\check Q_0$ in \eqref{eq:asynchro_Q0}; see Assumption \ref{as:Qlearning} and Remark~\ref{rem:Qlearning}. %
		\item [(ii)] %
		      We have by Lemma~\ref{lem:lctransform} that $\Phi_{t,(\check\mu,\check \pi)}$ in \eqref{eq:defconjugateJ} satisfies %
		      \begin{align}\label{eq:dual_Phi}
			      \Phi_{t,(\check\mu,\check \pi)}=\inf_{p \in {\cal B}^0_{m,q}(\widehat p_{\varepsilon^0})} \int_{E^0} \max_{\check\pi'\in \check \Pi}\check {Q}_t\big(\operatorname{pj}_{\check S}(\overline{\operatorname{F}}( \mu\hatotimes\pi, e^0)),\check\pi'\big) p(de^0).
		      \end{align}
		      Moreover, when $m=0$ (i.e., the non-robust case), the update rule \eqref{eq:asynchro_Q1} reduces to
		      \begin{align*}
			      \qquad \quad \check{Q}_{t+1}(\check \mu, \check \pi)=
			      (1-\alpha_{t,(\check \mu, \check \pi)})\check{Q}_{t}(\check \mu, \check \pi)+\alpha_{t,(\check \mu, \check \pi)}\Big[\overline r(\check \mu\hatotimes\check \pi)+\beta \max_{\check\pi'\in \check \Pi}\check {Q}_t(\check \mu_{t+1,(\check\mu,\check\pi)}',\check\pi')\Big],
		      \end{align*}
		      where $\check \mu_{t+1,(\check\mu,\check\pi)}'
			      :=
			      \operatorname{pj}_{\check S}
			      (
			      \overline{\operatorname{F}}
			      (\check\mu\hatotimes\check\pi,
			      \varepsilon_{t+1,(\check\mu,\check\pi)}^0)
			      )\in \check S$ represents the lifted-state transition from the pair $(\check \mu, \check \pi)$ at time~$t$. Therefore, this case corresponds to the standard tabular $Q$-learning algorithm on the finite space $\check S\times \check \Pi$ (see, e.g., \cite{bertsekas2025neuro,watkins1992q,littman1996generalized,6796861}).
	\end{itemize}
\end{remark}

\begin{remark}\label{rem:presample_data}
	The asynchronous $Q$-learning algorithm under Framework~\ref{frame:Qlearning} (ii) and (iii) adopts an offline-learning, independent-sampling setting. The pre-sampled, projected data $(\check{\mu}_t, \check{\pi}_t)_{t \geq 0}$ in Framework~\ref{frame:Qlearning} (ii) is obtained by applying the projection mappings $\operatorname{pj}_{\check S}$ and $\operatorname{pj}_{\check \Pi}$ given in Definition~\ref{dfn:projectedmaps} to a historically collected lifted trajectory $(\mu_t, \pi_t)_{t \geq 0}\subseteq \overline S\times \Pi$, and is therefore  regarded as fixed or, more generally, as defined on an auxiliary probability space independent of $(\Omega^0, \mathcal{F}^0, \widehat{\mathbb{P}}^0)$. Consequently, the visitation pattern of state and action pairs is prescribed by this projected data. The stochasticity in update~\eqref{eq:asynchro_Q1} enters through the random variables  $(\varepsilon^0_{t,(\check{\mu},\check{\pi})})_{t \geq 1,\, (\check{\mu},\check{\pi}) \in \check{S} \times \check{\Pi}}$. These are directly sampled from a simulator with law $\widehat{p}_{\varepsilon^0}$, constructed from the original (unprojected) historical lifted trajectory $(\mu_t, \pi_t)_{t \geq 0}\subseteq \overline S\times \Pi$, for instance, via filtering or direct empirical estimation of the common noise processes. The probability space $(\Omega^0, \mathcal{F}^0, \widehat{\mathbb{P}}^0)$ is then understood as the canonical space supporting this simulator, with $\widehat{\mathbb{P}}^0$ the law it induces on $\Omega^0$. This is consistent with the batch/offline $Q$-learning paradigm, in which learning is carried out from static, previously collected data;  see, e.g., \cite{lange2012batch, fujimoto2019off, levine2020offline, rashidinejad2022bridging, li2024settling}.

	Numerically, one could alternatively generate data along the way using an $\epsilon$-greedy policy. This approach generally leads to better numerical results but it can be difficult to guarantee that the assumptions used in the proof of convergence (in particular Assumption~\ref{as:Qlearning} (ii) related to the covering time) are satisfied.
\end{remark}

\begin{algorithm}[t]
	\caption{$Q$-learning algorithm for robust MFC problem}
	\label{alg:Qlearning}
	\begin{algorithmic}[1]
		\REQUIRE  Number of iteration $T\in \mathbb{N}$. pre-sampled, projected dataset $(\check \mu_t, \check \pi_t)_{t=0}^{T-1}\subseteq \check S\times \check \Pi$ used in the asynchronous $Q$-learning algorithm under Framework~\ref{frame:Qlearning}\;(ii),(iii). %
		\STATE Initialize $\check Q_0(\check \mu,\check \pi):=\check Q_0$ for all $(\check \mu,\check \pi)\in\check S\times \check \Pi$ for some $\check Q_0\in \mathbb{R}$.
		\IF{Framework \ref{frame:Qlearning}\;(i),(iii) (i.e., the synchronous $Q$-learning)}
		\STATE Set $\alpha_{t,(\mu, \check \pi)}:=(t+1)^{-w}$ for all $t=0,\dots,T-1$ and $(\check \mu,\check \pi)\in \check S\times \check \Pi$.
		\FOR{$t = 0, \ldots, T-1$}
		\FOR{every $(\check \mu,\check \pi)\in  \check S\times \check \Pi$}
		\STATE Execute action $\check \pi$ in state $\check \mu$, and observe $\overline r(\check \mu\hatotimes\check \pi)$ and $\varepsilon_{t+1,(\check \mu,\check \pi)}^0$.
		\STATE Compute the optimizer $\lambda^*_{t,(\check \mu,\check \pi)}$ of $\Phi_{t,(\check \mu,\check \pi)}$; see \eqref{eq:defconjugateJ}.
		\STATE Compute the $\lambda c$-transform value $(-J_{t,(\check \mu, \check \pi)} )^{\lambda_{t,(\check \mu, \check \pi)}^*}(\varepsilon^0_{t+1,(\check\mu,\check\pi)})$; see \eqref{eq:lctransform} and \eqref{eq:defJ}.
		\STATE Update $\check{Q}_{t+1}(\check \mu, \check \pi)$ according to \eqref{eq:asynchro_Q1}.
		\ENDFOR
		\ENDFOR
		\ELSE
		\STATE Use $(\check \mu_t, \check \pi_t)_{t=0}^{T-1}$ to compute $(\alpha_{t,(\check \mu_t, \check \pi_t)})_{t=0}^{T-1}$ according to \eqref{eq:learningrate} (i.e., the asynchronous $Q$-learning algorithm under Framework~\ref{frame:Qlearning}\;(ii),(iii)).
		\FOR{$t = 0, \ldots, T-1$}
		\STATE Set $\check Q_{t+1}(\check \mu,\check \pi):=\check Q_{t}(\check \mu,\check \pi)$ for all $(\check \mu,\check\pi)\in \check S\times \check \Pi$.
		\STATE Apply Steps 6--9 with $(\check \mu_t,\check \pi_t)$ in place of $(\check \mu,\check \pi)$ to update only $\check Q_{t+1}(\check \mu_t,\check \pi_t)$.
		\ENDFOR
		\ENDIF
		\RETURN $\check Q_{T^*}$
	\end{algorithmic}
\end{algorithm}

\subsection{Main theorems}
We proceed with our main results on the convergence of the $Q$-learning algorithm prescribed in Framework \ref{frame:Qlearning}, together with its finite-time iteration bound analysis.

To this end, we impose the following conditions.
\begin{assumption}\label{as:Qlearning} Recall the constants $C_{\operatorname{F}}>0$ and $C_r>0$ in Assumption \ref{as:MFC}.
	\begin{itemize}
		\item [(i)] $\beta$ is in $[0,\frac{1}{3}\wedge (2C_{\operatorname{F}})^{-1})$. Moreover, $\check Q_0$ in \eqref{eq:asynchro_Q0} satisfies $|\check Q_0|\leq \check Q_{\operatorname{max}}:=\frac{C_{{r,\infty}}}{1-3\beta}$.
		\item [(ii)] There exists a finite covering time $T_{\operatorname{cov}}> 1$ such that for every $s\geq 0$ \[\check S\times \check \Pi \subseteq (\check \mu_t,\check \pi_t)_{t=s,\dots,s+T_{\operatorname{cov}}-1},\]
		      where $(\check \mu_t,\check \pi_t)_{t\geq 0}$ is the pre-sampled, projected dataset used for Framework~\ref{frame:Qlearning}\;(ii).
	\end{itemize}
\end{assumption}
\begin{remark}\label{rem:Qlearning}
	The condition on $\beta$ in Assumption \ref{as:Qlearning}\;(i) is stronger than that in Assumption~\ref{as:MFC}~(iii). We impose this stronger condition, together with the assumption on $\check Q_0$, in order to establish the existence of $\lambda_{t,(\cdot, \cdot)}^*$ for all $t\geq 0$ in \eqref{eq:defconjugateJ}, as well as to obtain uniform-in-time bounds for $(\lambda_{t,(\cdot, \cdot)}^*)_{t\geq 0}$ and  $(\check Q_t)_{t\geq 0}$; see Lemma~\ref{lem:Q_est}. We also note that other conditions on the discount factor appear in robust $Q$-learning frameworks such as \cite{tamar2014scaling,roy2017reinforcement}, where they are primarily used to establish the contraction property and convergence of the $Q$-learning iteration.
\end{remark}
\begin{remark}
	The covering-time condition in Assumption~\ref{as:Qlearning}\;(ii) ensures that under the asynchronous learning rates in Framework \ref{frame:Qlearning}\;(ii), every pair in $\check S\times\check \Pi$ is visited at least once within any time window of length greater than the covering time~$T_{\operatorname{cov}}$. Such a condition is standard in the asynchronous $Q$-learning literature to guarantee sufficient exploration and convergence of the respective $Q$-learning algorithm; see, e.g., \cite{beck2012error,li2020sample,qu2020finite,szepesvari1997asymptotic}.
\end{remark}

We now present the convergence result of our $Q$-learning algorithm.
To this end, we define
\begin{align}\label{eq:sep_dis_diam}
	\underline{\Delta}_S:=\min_{s\neq s'\in S}|s-s'|>0\quad \mbox{and}\quad \Delta_A:=\max_{a\neq a'\in A}|a-a'|<\infty,
\end{align}
which represent the minimal separation distance in~$S$ and the diameter of $A$, respectively.

Last, we recall the constant $L^*$ from Proposition \ref{pro:LNP2025}. %
\begin{theorem}\label{thm:asynchro}
	Under Assumptions \ref{as:MFC}\;(i),(ii) and \ref{as:Qlearning}\;(i), the following statements hold:
	\begin{itemize}
		\item [(i)] the synchronous version of the $Q$-learning algorithm in Framework \ref{frame:Qlearning}\;(i),(iii) satisfies, for every
		      $(\mu,\pi)\in \overline S\times \Pi$,
		      \begin{align}\label{eq:Qconverge}
			      \lim_{t\to \infty}|\check{Q}_{t}(\operatorname{pj}_{\check S}(\mu),\operatorname{pj}_{\check \Pi}(\pi))-Q^*(\mu, \pi)|\leq C_1\varepsilon_{\check S}+{C_2}\varepsilon_{\check A},\quad \mbox{$\widehat{\mathbb{P}}^0$-a.s.},
		      \end{align}
		      where $C_1:=2(1+\Delta_A\underline{{\Delta}}_S^{-1})(C_r+\beta L^* C_{\operatorname{F}})+\frac{\beta}{1-\beta}L^*>0$ and $C_2:=\frac{2(C_r+\beta L^* C_{\operatorname{F}})}{1-\beta}>0$.
		\item [(ii)] Moreover, if Assumption \ref{as:Qlearning}\;(ii) also holds, then the asynchronous version of the $Q$-learning algorithm in Framework \ref{frame:Qlearning}\;(ii),(iii) satisfies  \eqref{eq:Qconverge} for every $(\mu,\pi)\in \overline S\times \Pi$.
	\end{itemize}
\end{theorem}

\begin{remark}\label{rem:proj_error}
	The error term ``$C_1\varepsilon_{\check S}+{C_2}\varepsilon_{\check A}$'' appearing in~\eqref{eq:Qconverge} does not arise from the stochastic iteration of the $Q$-learning algorithm itself, but rather from the quantization and projection procedure in Definition~\ref{dfn:projectedmaps} used to construct the discretized approximation of the optimal $Q$-function. The corresponding projection and discretization error analysis is provided in Section~\ref{sec:DiscretizedQ}, particularly in Lemma~\ref{lem:proj_errrors}.
\end{remark}

We next present the finite-time iteration bound analysis for our $Q$-learning algorithm. %
\begin{theorem}\label{thm:asynchro_rate}
	Suppose that Assumptions \ref{as:MFC}\;(i),(ii) and \ref{as:Qlearning}\;(i) are satisfied. %
	Then the following statements hold for any $\widehat \varepsilon>0$ and $\widehat\delta \in(0,1)$.
	\begin{itemize}
		\item [(i)] There exists an iteration number $T^*\in \mathbb{N}$ of order\footnote{We denote by $O(\cdot)$ the Landau symbol.}
		      \begin{align*}
			      \quad \qquad O \bigg(\bigg[\Big(\frac{\check Q_{\operatorname{max}}}{\widehat \varepsilon}\Big)^2\max\bigg\{\log\Big(\frac{\check Q_{\operatorname{max}}}{\widehat \varepsilon}\Big),\log\bigg(\frac{|\check S||\check A|^{|\check S|}}{\widehat \delta}\log\Big(\frac{\check Q_{\operatorname{max}}}{\widehat \varepsilon}\Big)\bigg)\bigg\}\bigg]^{\frac{1}{w}}+ \bigg[\log\Big(\frac{\check Q_{\operatorname{max}}}{\widehat \varepsilon}\Big)\bigg]^{\frac{1}{1-w}} \bigg),
		      \end{align*}
		      such that %
		      the synchronous version of the $Q$-learning algorithm in Framework~\ref{frame:Qlearning}\;(i),(iii) satisfies, for every $T\geq T^*$,
		      \begin{align}\label{eq:Qconverge2}
			      \widehat{\mathbb{P}}^0\bigg(\bigg\{\sup_{(\mu,\pi)\in \overline S\times \Pi}\Big|\check Q_{T}(\operatorname{pj}_{\check S}(\mu),\operatorname{pj}_{\check \Pi}(\pi))- Q^*(\mu,\pi)\Big|\leq C_1\varepsilon_{\check S}+{C_2}\varepsilon_{\check A}+\widehat\varepsilon\bigg\}\bigg)\geq 1-\widehat \delta,
		      \end{align}
		      with the constants $C_1$ and $C_2$ given in Theorem \ref{thm:asynchro}.
		\item [(ii)] Moreover, if Assumption \ref{as:Qlearning}\;(ii) is satisfied and we let
		      \begin{align}\label{eq:order_parameter}
			      \check N_1:={T_{\operatorname{cov}}^{1+5w}\check Q_{\operatorname{max}}^2},\quad \check N_2:= {T_{\operatorname{cov}}\check Q_{\operatorname{max}}},\quad  \check N_3:=T_{\operatorname{cov}}|\check S||\check A|^{|\check S|}, %
		      \end{align}
		      then there exists an iteration number $T^*\in \mathbb{N}$ of order
		      \begin{align*}
			      \qquad \quad  O\bigg(\bigg[\frac{\check N_1}{\widehat \varepsilon^2}\max\bigg\{\log\Big(\frac{\check N_2}{\widehat \varepsilon}\Big),\log\bigg(\frac{\check N_3}{\widehat \delta}\log\Big(\frac{\check Q_{\operatorname{max}}}{\widehat \varepsilon }\Big)\bigg)\bigg\}\bigg]^{\frac{1}{w}}+ \bigg[T_{\operatorname{cov}}\log\Big(\frac{ \check Q_{\operatorname{max}}}{\widehat \varepsilon}\Big) \bigg]^{\frac{1}{1-w}} \bigg),
		      \end{align*}
		      such that %
		      the asynchronous version in Framework~\ref{frame:Qlearning}\;(ii),(iii) also satisfies \eqref{eq:Qconverge2} for every $T\geq T^*$.
	\end{itemize}
\end{theorem}

\begin{remark}
	Theorem~\ref{thm:asynchro_rate} presents a finite-time high-probability convergence estimate for the $Q$-learning algorithm in Framework~\ref{frame:Qlearning}. More precisely, for any prescribed error and confidence levels, one can construct an iteration number \(T^*\) of the stated order such that, for every \(T\geq T^*\), the corresponding \(Q\)-learning iterate \(\check Q_T\) has error bounded by the prescribed error plus the discretization error \(C_1\varepsilon_{\check S}+C_2\varepsilon_{\check A}\) (see also Remark~\ref{rem:proj_error}) with at least the prescribed confidence~level.

	For the asynchronous $Q$-learning case, Theorem~\ref{thm:asynchro_rate} can also be read as a finite-sample complexity bound, since each iteration updates one state and action pair along the pre-sampled, projected trajectory, so that the iteration bound is equivalently a bound on the number of sampled updates needed to reach the stated accuracy. Related non-asymptotic sample-complexity analyses for asynchronous $Q$-learning algorithms are presented in~\cite{qu2020finite,li2020sample,li2024settling,kearns1998finite,li2024breaking}. %
\end{remark}

The following corollary reformulates Theorem~\ref{thm:asynchro_rate} in terms of the convergence rate with respect to the iteration number \(T\). %
\begin{cor}\label{cor:Qlearningrate} Suppose that Assumptions \ref{as:MFC}\;(i),(ii) and \ref{as:Qlearning}\;(i) are satisfied. Then the synchronous version of the
	$Q$-learning algorithm in Framework~\ref{frame:Qlearning}\;(i),(iii) satisfies,\footnote{We denote by $O_{\widehat{\mathbb{P}}^0}(\cdot)$ the Landau symbol in $\widehat{\mathbb{P}}^0$-probability; that is, for a sequence of random variables $(X_n)_{n \geq 1}$ and a deterministic sequence $(a_n)_{n \geq 1}$, we write $X_n = O_{\widehat{\mathbb{P}}^0}(a_n)$ as $n\to\infty$, if for every $\widehat \delta > 0$, there exists a finite $M > 0$ and $N\in \mathbb{N}$ such that \(
		\widehat{\mathbb{P}}^0(|\frac{X_n}{a_n}| > M) \leq \widehat \delta\)
		for every $n\geq N$.} as $T\to \infty$,
	\begin{align}\label{eq:Qconverge3}
		\sup_{(\mu,\pi)\in\overline S\times\Pi}
		\left|
		\check Q_T(\operatorname{pj}_{\check S}(\mu),
		\operatorname{pj}_{\check\Pi}(\pi))
		-
		Q^*(\mu,\pi)
		\right|
		\leq
		C_1\varepsilon_{\check S}
		+
		C_2\varepsilon_{\check A}+O_{\widehat{\mathbb P}^0}\bigg(
		\sqrt{\frac{\log T}{T^w}}
		\bigg),
	\end{align}
	with the constants $C_1$ and $C_2$ given in Theorem \ref{thm:asynchro}.
	Moreover, if Assumption \ref{as:Qlearning}\;(ii) is satisfied, then the asynchronous version in Framework~\ref{frame:Qlearning}\;(ii),(iii) also satisfies \eqref{eq:Qconverge3}.
\end{cor}

\section{Numerics}\label{sec:numerics}
In this section we illustrate the robust $Q$-learning algorithm on three
finite-grid mean-field control examples. 
\footnote{The code is available at \url{https://github.com/mlauriere/RobustMeanFieldControlQLearning}.}

\subsection{Numerical Setup}\label{subsec:numerics-setup}

The computations use the projected finite state and policy spaces from
Definition~\ref{dfn:projectedmaps}.  The runs of Algorithm~\ref{alg:Qlearning}
use the asynchronous branch, namely the \textsc{else} part of the algorithm,
corresponding to Framework~\ref{frame:Qlearning}\;(ii),(iii), with learning
rates as in \eqref{eq:learningrate}.  The robust targets are computed using the
common-noise dual in Lemma~\ref{lem:lctransform} and
\eqref{eq:defconjugateJ}; in particular the transportation cost is on
$E^0$.  In all experiments
reported below, we set $q=1$ and $|E^0|=3$, and the reference common-noise law is
\[
	\widehat p_{\varepsilon^0}=(0.1,0.8,0.1).
\]
The pre-sampled lifted state--policy pairs
$(\check\mu_t,\check\pi_t)_{t\ge 0}$ used by the asynchronous implementation
are generated by repeated random permutations of $\check S\times\check\Pi$.
Hence the visitation pattern is independent of the current $\check Q$ table,
and one independent common-noise realization is drawn at each update.  The
dashed reference curves in Figures \ref{fig:numerics-sysrisk-profile} and \ref{fig:numerics-sis-profile} are not produced by a model-free algorithm.  They
are obtained by deterministic Bellman iteration for the finite projected
mean-field control problem on $\check S\times\check\Pi$. %
Below, for radius $m$, the resulting finite-grid fixed point is denoted by $\check Q^{*,m}$, while the iterated $Q$-function at iteration $T$ is denoted by $\check Q_T^{m}$. The Bellman-iteration procedure to obtain $\check{Q}^{*,m}$ requires the knowledge of the transition and reward functions and hence does not form a reinforcement algorithm. However, we use this methodology as benchmark to compare it with the outcome of our robust $Q$-learning algorithm.

In the examples below, $S_{\rm sys}$,
$S_{\rm SIS}$, and $S_{\rm SEIR}$ denote concrete individual state spaces
$S$ in the notation of Section~\ref{sec:Qlearning}; these should be
distinguished from the finite projected lifted-state grid $\check S$.

Robustness profiles are evaluated under perturbed common-noise laws
\[
	p_{\rm eval}^{\zeta}=(1-\zeta)\widehat p_{\varepsilon^0}
	+\zeta\,p_{\rm adv},
	\qquad \zeta\in\{0,0.125,\ldots,1\},
\]
where $\zeta=0$ is the reference law and $\zeta=1$ is an example-specific
adverse endpoint used only for evaluation.  The law $p_{\rm adv}$ is not a
new training distribution; it is a plotting device that concentrates mass on
an unfavorable common-noise state so that increasing $\zeta$ produces a
controlled misspecification of the reference law.  We interpret $ p_{\rm eval}^{\zeta}$ as the true law for the common noise, which is, however, unknown to the agent, forcing him to optimize robustly. If $\zeta = 0$, then the reference law $\widehat p_{\varepsilon^0}$ coincides with the true law $ p_{\rm eval}^{\zeta}$, whereas if $\zeta > 0$, then there is some model misspecification.

The plotted quantities in Figures \ref{fig:numerics-sysrisk-profile} and \ref{fig:numerics-sis-profile} are
discounted rewards of the greedy
policy induced by the computed $\check Q$ table.  For each fixed value of $\zeta$, this is an ordinary evaluation under the fixed common-noise law $p_{\rm eval}^{\zeta}$, not an additional robust infimum over laws.  Since the primitive dynamics are given by the map $\overline{\operatorname{F}}$, the law $p_{\rm eval}^{\zeta}$ enters only through the finite average over common-noise values in the display below.  For a fixed evaluation law
$p_{\rm eval}^{\zeta}$ and greedy finite-grid policy
$\check\pi_{\check Q}(\check\mu)\in\operatorname*{arg\,max}_{\check\pi\in\check\Pi}
	\check Q(\check\mu,\check\pi)$, this reward is
computed by solving the finite linear system
\[
	V_{\check Q}(\check\mu)= \overline r(\check\mu\hatotimes\check\pi_{\check Q}(\check\mu)) +
	\beta\sum_{e^0\in E^0} p_{\rm eval}^{\zeta}(e^0)
	V_{\check Q}\big(\operatorname{pj}_{\check S}(
	\overline{\operatorname{F}}(\check\mu\hatotimes
	\check\pi_{\check Q}(\check\mu),e^0))\big),
	\quad \check\mu\in\check S,
\]
and averaging $V_{\check Q}$ with respect to the uniform measure on the finite
projected state grid $\check S$.
Larger plotted values in Figures \ref{fig:numerics-sysrisk-profile} and \ref{fig:numerics-sis-profile} therefore correspond to better performance. In Figures \ref{fig:numerics-sysrisk-profile} and~\ref{fig:numerics-sis-profile}, solid curves show means over $20$ independent
seeds of the asynchronous implementation of Algorithm~\ref{alg:Qlearning},
and dashed curves show the deterministic finite-grid Bellman references.

Table~\ref{tab:numerics-config} summarizes the finite-grid configurations.
We use
\[
	\begin{aligned}
		{\cal R}_{\rm sys}
		 & =\{0,0.05,0.1,0.2,0.3,0.4,0.5,0.6,0.75,1\}, \\
		{\cal R}_{\rm epi}
		 & =\{0,0.005,0.01,0.02,0.03,0.04,0.05,0.06,
		0.08,0.1,0.15,0.2,0.3,0.5,0.75,1\}.
	\end{aligned}
\]
Here ${\cal R}_{\rm sys}$ and ${\cal R}_{\rm epi}$ denote the choices for Wasserstein robustness radii $m$ used in the robustness-profile experiments.
The construction of $\check S$ and $\check \Pi$ is explained in each example.
\begin{table}[htbp]
	\centering
	\begin{tabular}{lcccccc}
		\toprule
		\addlinespace[0.35em]
		Example       & $\beta$         & $|\check S|$ & $|\check\Pi|$
		              & Updates per $m$ & $w$          & Robustness radii $m$ \\
		\midrule
		Systemic Risk & $0.5$           & $21$         & $216$
		              & $5{,}000{,}000$ & $0.75$       & ${\cal R}_{\rm sys}$ \\
		SIS           & $0.5$           & $13$         & $11$
		              & $1{,}000{,}000$ & $0.7$        & ${\cal R}_{\rm epi}$ \\
		SEIR          & $0.9$           & $165$        & $11$
		              & $3{,}000{,}000$ & $0.7$        & ${\cal R}_{\rm epi}$ \\
		\bottomrule
	\end{tabular}
	\caption{Finite-grid configurations for the asynchronous robustness
		profiles.}
	\label{tab:numerics-config}
\end{table}
\begin{remark}
	The finite examples satisfy the finiteness and boundedness requirements by
	construction.  The lifted rewards and lifted transition maps used below are
	Lipschitz on the relevant finite-dimensional simplices, with clipping preserving
	Lipschitz continuity.  The asynchronous data schedule is generated by repeated
	random permutations of $\check S\times\check\Pi$, hence the covering-time
	condition in Assumption~\ref{as:Qlearning}\;(ii) holds with
	$T_{\operatorname{cov}}=|\check S||\check\Pi|$.  The discount-factor condition
	in Assumption~\ref{as:Qlearning}\;(i), however, is a sufficient condition for
	the theorem and is not imposed in all numerical runs, in particular in the main
	SEIR experiment where $\beta=0.9$.  Thus the main numerical results should be
	read as empirical finite-grid evidence for Algorithm~\ref{alg:Qlearning},
	while Figure~\ref{fig:numerics-q-convergence} illustrates the rate form in
	Corollary~\ref{cor:Qlearningrate}.
\end{remark}

\noindent {\bf $Q$-function convergence.}
Figure~\ref{fig:numerics-q-convergence} reports the convergence speed of
the asynchronous $\check Q$ iterates on the same finite grids used in the robustness
profiles.  For each displayed robustness radius $m$, let $\check Q^*_m$ denote the
corresponding idealized finite-grid fixed point.  At iteration time $T$, the
plotted error is
\[
	E_T(m)=\|\check Q_T^{m}-\check Q^{*,m}\|_{\check S\times \check \Pi}
	=
	\max_{(\check\mu,\check\pi)\in\check S\times\check\Pi}
	|\check Q_T^m(\check\mu,\check\pi)-\check Q^{*,m}(\check\mu,\check\pi)|.
\]
The line is the median over $10$ seeds and the band is the $10$--$90$
percentile range.  The black dashed guide is proportional to
$\sqrt{\log(T)/T^{w}}$, with the learning-rate exponent $w$ from
Table~\ref{tab:numerics-config}.  This matches the rate form in
Corollary~\ref{cor:Qlearningrate} up to constants and discretization error.

\begin{figure}[htbp]
	\centering
	\includegraphics[width=\textwidth]{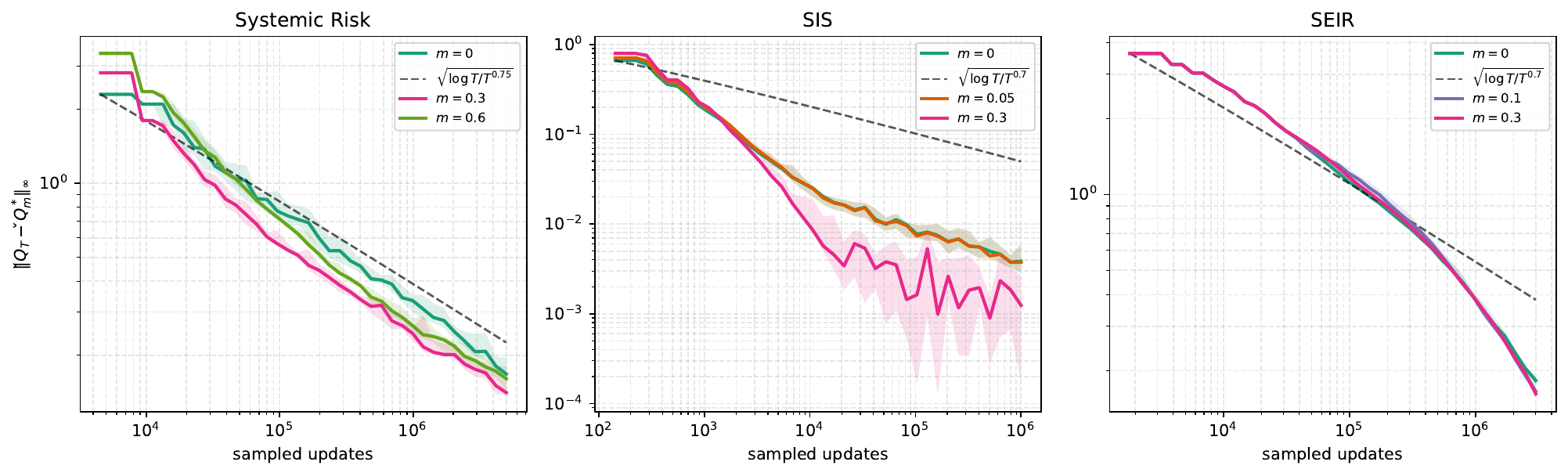}
	\caption{\textbf{Asynchronous $Q$-function convergence.} Error $E_T(m)$ against the idealized finite-grid fixed point for
		selected robustness radii.}
	\label{fig:numerics-q-convergence}
\end{figure}

\subsection{Systemic Risk}\label{subsec:numerics-systemic-risk}

This example is a stylized finite-state model of a population of financial
institutions whose capital levels are affected by individual controls and
aggregate shocks.
The individual state, action, and common-noise spaces are
\[
	S_{\rm sys}=\{0,1,2\},\qquad
	A=\{-1,0,1\},\qquad
	E^0=\{-1,0,1\}.
\]
The states $0,1,2$ represent distressed, normal, and well-capitalized
agents.  Given individual state $x$, action $a$, and common shock $e^0$,
the next state is
\[
	x'=\min\{2,\max\{0,x+a+e^0\}\}.
\]
There is no idiosyncratic noise in this example.  If the current population
distribution is $\mu$ and the finite-grid population policy is $\pi$, the
mean-field transition is the law of $x'$ when $x\sim\mu$ and
$a\sim\pi(\cdot|x)$.  The lifted one-step reward is
\[
	\overline r(\mu\hatotimes\pi)
	=-\|\mu-(0,1,0)\|_2^2-2\mu_0.
\]
where $\mu_0$ is the population mass in the distressed state.  Thus the
planner is rewarded for avoiding both deviation from the normal population
state and mass in the distressed state.  The adverse law is
$p_{\rm adv}=(1,0,0)$, which puts all mass on the negative common shock.

The projected state grid is the uniform probability-simplex grid
\[
	\check S_{\rm sys}
	=
	\left\{
	\left(\frac{k_0}{5},\frac{k_1}{5},\frac{k_2}{5}\right)
	:
	k_0,k_1,k_2\in\mathbb N\cup \{0\},\
	k_0+k_1+k_2=5
	\right\}.
\]
Thus $|\check S_{\rm sys}|=\binom{7}{2}=21$.  The individual action set has
three actions, $A=\{-1,0,1\}$.  We discretize local randomized actions by
\[
	\check A_{\rm sys}
	=
	\left\{
	\left(\frac{\ell_{-1}}{2},\frac{\ell_0}{2},\frac{\ell_1}{2}\right)
	:
	\ell_{-1},\ell_0,\ell_1\in\mathbb N\cup \{0\},\
	\ell_{-1}+\ell_0+\ell_1=2
	\right\},
\]
so $|\check A_{\rm sys}|=\binom{4}{2}=6$.  A finite-grid population policy
assigns one element of $\check A_{\rm sys}$ to each individual state
$x\in S_{\rm sys}$, hence
\[
	\check\Pi_{\rm sys}
	=
	\{\check\pi:S_{\rm sys}\to \check A_{\rm sys}\},
	\qquad
	|\check\Pi_{\rm sys}|=6^3=216.
\]

With the configuration in Table~\ref{tab:numerics-config},
Figure~\ref{fig:numerics-sysrisk-profile} shows a clear hedging effect:
under the adverse law $\zeta=1$, the mean reward from Algorithm~\ref{alg:Qlearning} increases from $-5.119$ at
$m=0$ to $-3.230$ at $m=0.6$, while excessive robustness remains slightly
below the best moderate radius.  Under the reference law, by contrast,
robustness is only mildly useful at very small radius and then lowers reward
as $m$ increases.  The asynchronous and idealized curves are closest for moderate
and large robustness radii; the larger non-robust gap visible at $m=0$
reflects slower finite-time learning in this example.

\begin{figure}[htbp]
	\centering
	\includegraphics[width=0.84\textwidth]{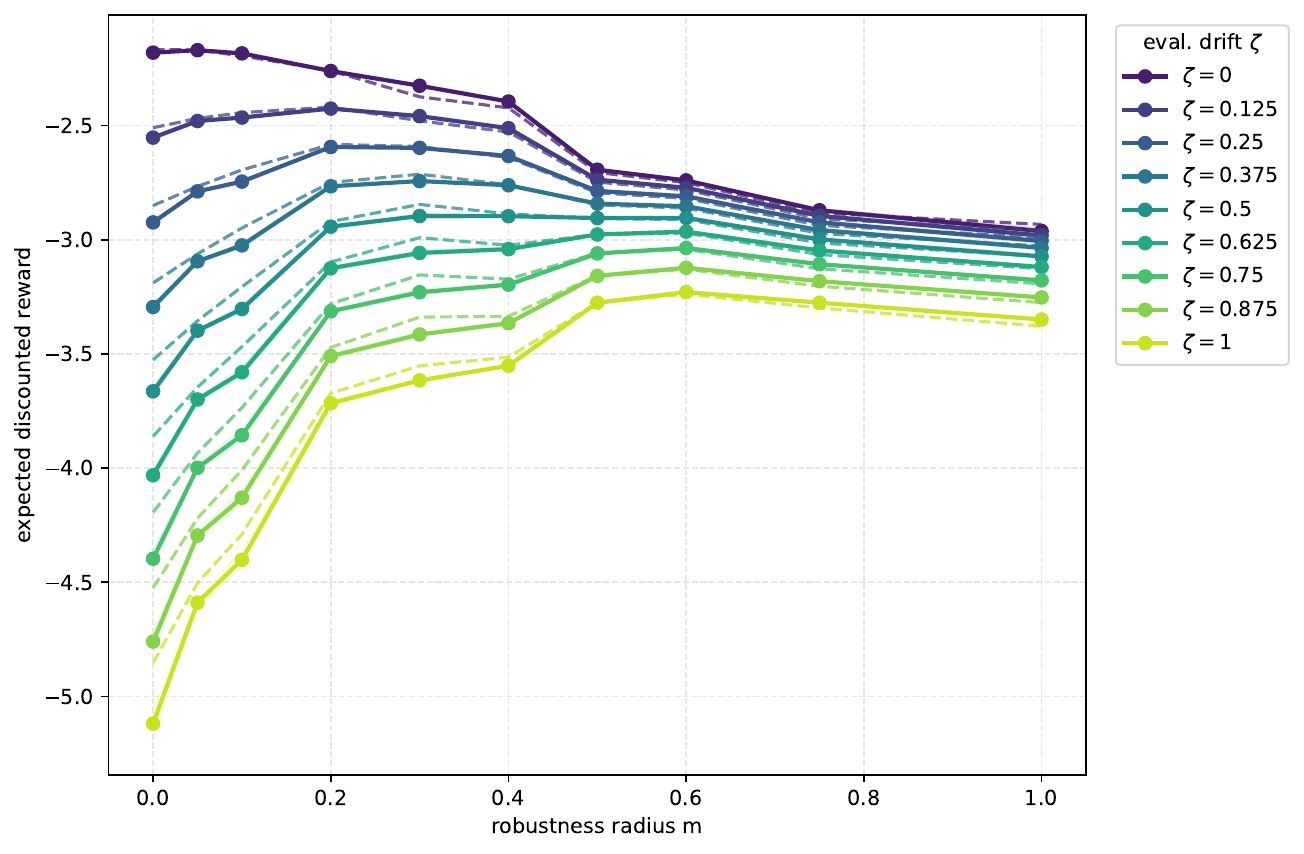}
	\caption{\textbf{Systemic Risk robustness profile.}  Solid curves show
		reward means from the asynchronous implementation of Algorithm~\ref{alg:Qlearning};
		dashed curves show same-grid idealized references.  Moderate
		robustness substantially improves performance under adverse common
		shocks.}
	\label{fig:numerics-sysrisk-profile}
\end{figure}

\subsection{Epidemics}\label{subsec:numerics-epidemics}

The epidemic examples model a planner who uses distancing to reduce disease
transmission in a population subject to aggregate transmission-rate shocks.
Both use common noise
\[
	E^0=\{0.5,0.81,1.8\},
\]
interpreted as the transmission-rate state.  The last point, $1.8$, is the
high-transmission state.  The adverse law is $p_{\rm adv}=(0,0,1)$, so the
interpolation defining $p_{\rm eval}^{\zeta}$ shifts evaluation mass toward
this high-transmission state as $\zeta$ increases.  In both epidemic models
we use the reduced two-action set $A=\{U,D\}$, where $U$ denotes unrestricted
behavior and $D$ denotes distancing.  Thus $\pi(D|s)$ denotes the probability
that policy $\pi$ assigns to the distancing action in individual state $s\in S$.
Since distancing non-susceptible agents does not affect the transition in
these finite models and only lowers reward, the reduced policy grid sets
$\pi(D|s)=0$ for non-susceptible states and varies $\pi(D|s)$ only at susceptible states.

\subsubsection{SIS}\label{subsubsec:numerics-sis}

The individual state space is
$S_{\rm SIS}=\{\mathrm S,\mathrm I\}$ corresponding to susceptible and infected states.  With
$\mu=(\mu_{\mathrm S},\mu_{\mathrm I})$, common noise $e^0=b$, and
$u=\pi(U|\mathrm S)$, the finite-grid mean-field transition used in the
computation is
\[
	\mu_{\mathrm I}'=\big[0.7\mu_{\mathrm I}
		+b\,\mu_{\mathrm I}\mu_{\mathrm S} u\big]_0^1,
	\qquad \mu_{\mathrm S}'=1-\mu_{\mathrm I}',
\]
where $[z]_0^1=\min\{1,\max\{0,z\}\}$.  This corresponds to recovery
probability $0.3$ for infected agents and infection pressure proportional
to the current infected mass, the transmission state $b$, and the
unrestricted susceptible fraction.  The lifted one-step reward penalizes the current infected mass and the
population-average probability of choosing the distancing action, and is defined as
\[
	\overline r(\mu\hatotimes\pi)=-\mu_{\mathrm I}
	-0.5\sum_{s\in S_{\rm SIS}}\mu_s\pi(D|s).
\]

In this example, the projected state grid is
\[
	\check S_{\rm SIS}
	=
	\left\{
	\left(\frac{k_{\mathrm S}}{12},\frac{k_{\mathrm I}}{12}\right)
	:
	k_{\mathrm S},k_{\mathrm I}\in\mathbb N_0,\
	k_{\mathrm S}+k_{\mathrm I}=12
	\right\},
\]
so $|\check S_{\rm SIS}|=13$.  The individual action set is
$A=\{U,D\}$, where $U$ denotes unrestricted behavior and $D$ denotes
distancing.  The susceptible-state action distribution is discretized as
\[
	\check A_{\rm epi}
	=
	\left\{
	\left(\frac{\ell_U}{10},\frac{\ell_D}{10}\right)
	:
	\ell_U,\ell_D\in\mathbb N_0,\
	\ell_U+\ell_D=10
	\right\},
\]
with coordinates corresponding to $(U,D)$.  Since distancing infected agents
does not affect the SIS transition and only lowers reward, the reduced policy
class fixes infected agents to choose unrestricted behavior:
\[
	\check\Pi_{\rm SIS}
	=
	\left\{
	\check\pi:
	\check\pi(\cdot|\mathrm S)\in\check A_{\rm epi},
	\quad
	\check\pi(\cdot|\mathrm I)=\delta_U
	\right\}.
\]
Therefore $|\check\Pi_{\rm SIS}|=|\check A_{\rm epi}|=11$.

With the configuration in Table~\ref{tab:numerics-config},
Figure~\ref{fig:numerics-sis-profile} shows close agreement between the
asynchronous and idealized profiles.  The moderate-robustness effect is visible
at intermediate drift: at $\zeta=0.5$, the mean reward from Algorithm~\ref{alg:Qlearning} increases from
$-1.089$ at $m=0$ to $-1.080$ at $m=0.03$, then decreases to $-1.098$ at
$m=1$.  At the severe law $\zeta=1$, high robustness remains beneficial,
with the mean reward from Algorithm~\ref{alg:Qlearning} increasing from $-1.175$ at $m=0$ to $-1.138$ at
$m=0.75$.

\begin{figure}[htbp]
	\centering
	\includegraphics[width=0.84\textwidth]{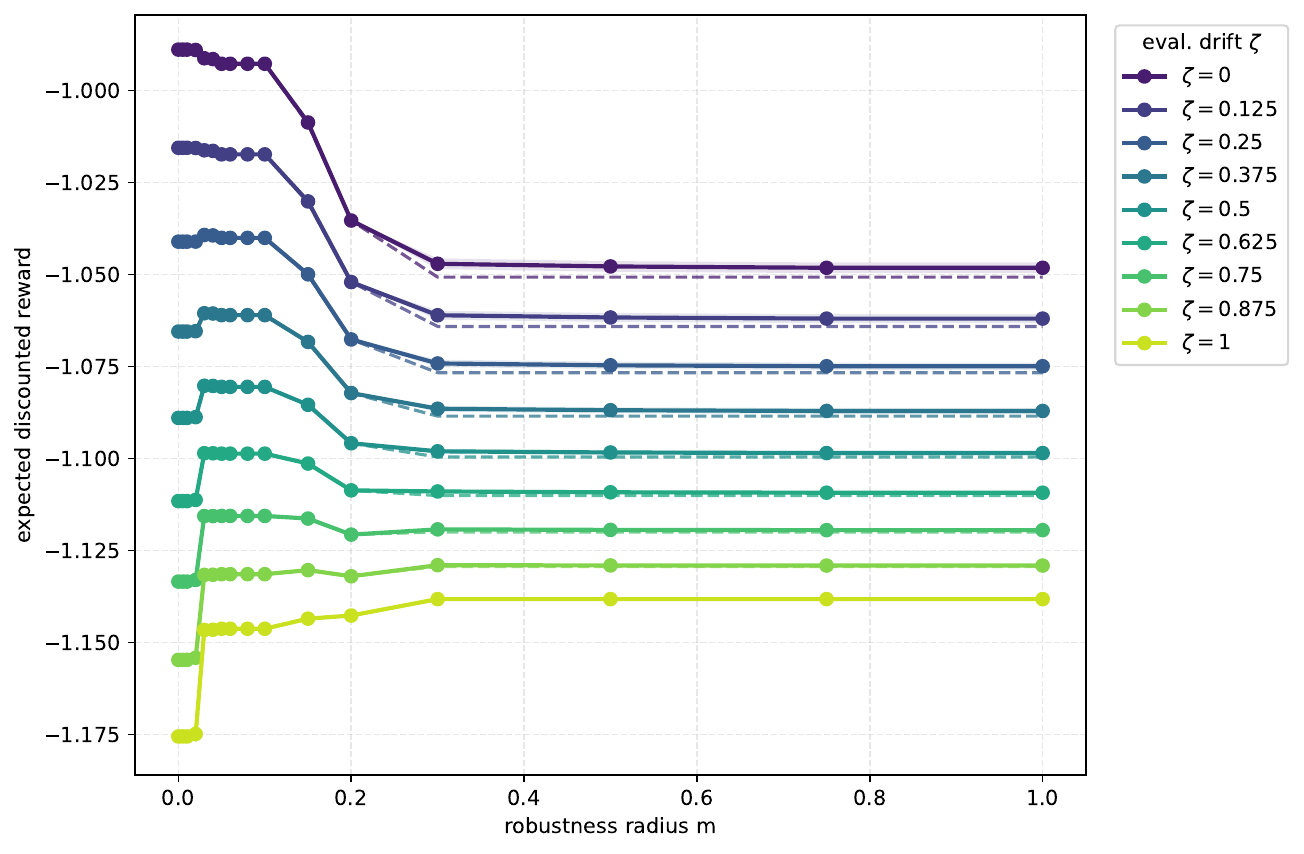}
	\caption{\textbf{SIS robustness profile.}  Solid curves show reward means
		from the asynchronous implementation of Algorithm~\ref{alg:Qlearning};
		dashed curves show idealized references.  The algorithm closely tracks
		the finite-grid idealized benchmark.}
	\label{fig:numerics-sis-profile}
\end{figure}

\subsubsection{SEIR}\label{subsubsec:numerics-seir}

The individual state space is
$S_{\rm SEIR}=\{\mathrm S,\mathrm E,\mathrm I,\mathrm R\}$, where the two new states are interpreted as Exposed and Recovered. In the Exposed state, the agent is not yet infectious. The common noise
$e^0=b$ is again the transmission rate.  Let
\[
	\theta_{\mathrm S}(\pi)=\pi(U|\mathrm S)+(1-\eta)\pi(D|\mathrm S).
\]
Here $\eta$ is the efficacy of distancing for susceptible agents, $\sigma$
is the incubation probability, and $\rho_{\rm rec}$ is the recovery
probability.  For $\mu=(\mu_{\mathrm S},\mu_{\mathrm E},\mu_{\mathrm I},\mu_{\mathrm R})$, the lifted transition map is described by the population mass transfers
\[
	\begin{aligned}
		d_{\mathrm S\mathrm E}
		                       & =\min\{\mu_{\mathrm S},
		b\mu_{\mathrm I}\mu_{\mathrm S}\theta_{\mathrm S}(\pi)\}, \\
		d_{\mathrm E\mathrm I} & =\sigma\mu_{\mathrm E},\qquad
		d_{\mathrm I\mathrm R}=\rho_{\rm rec}\mu_{\mathrm I}.
	\end{aligned}
\]

Equivalently,
\[
	\begin{aligned}
		\mu'_{\mathrm S} & =\mu_{\mathrm S}-d_{\mathrm S\mathrm E},                        \\
		\mu'_{\mathrm E} & =\mu_{\mathrm E}+d_{\mathrm S\mathrm E}-d_{\mathrm E\mathrm I}, \\
		\mu'_{\mathrm I} & =\mu_{\mathrm I}+d_{\mathrm E\mathrm I}-d_{\mathrm I\mathrm R}, \\
		\mu'_{\mathrm R} & =\mu_{\mathrm R}+d_{\mathrm I\mathrm R},
	\end{aligned}
\]
followed by clipping and renormalization on the finite population grid.  The
numerical experiments use $\eta=1$, $\sigma=0.5$, and
$\rho_{\rm rec}=0.2$.  The one-step reward is
\[
	\overline r(\mu\hatotimes\pi)=-\mu_{\mathrm I}-0.85 \sum_{s\in S_{\rm SEIR}}\mu_s\pi(D|s).
\]
Under the reduced policy class, this distancing term is
$\mu_{\mathrm S}\pi(D|\mathrm S)$.

In this example, the projected state grid is
\[
	\check S_{\rm SEIR}
	=
	\left\{
	\left(
	\frac{k_{\mathrm S}}{8},
	\frac{k_{\mathrm E}}{8},
	\frac{k_{\mathrm I}}{8},
	\frac{k_{\mathrm R}}{8}
	\right)
	:
	k_{\mathrm S},k_{\mathrm E},k_{\mathrm I},k_{\mathrm R}\in\mathbb N_0,\
	k_{\mathrm S}+k_{\mathrm E}+k_{\mathrm I}+k_{\mathrm R}=8
	\right\}.
\]
Thus $|\check S_{\rm SEIR}|=\binom{11}{3}=165$.  As in the SIS example,
$A=\{U,D\}$ and the susceptible-state action distribution is discretized on
\[
	\check A_{\rm epi}
	=
	\left\{
	\left(\frac{\ell_U}{10},\frac{\ell_D}{10}\right)
	:
	\ell_U,\ell_D\in\mathbb N_0,\
	\ell_U+\ell_D=10
	\right\}.
\]
Because distancing exposed, infected, or recovered agents does not affect the
SEIR transition in this finite model and only lowers reward, the reduced policy
class fixes these states to unrestricted behavior:
\[
	\check\Pi_{\rm SEIR}
	=
	\left\{
	\check\pi:
	\check\pi(\cdot|\mathrm S)\in\check A_{\rm epi},
	\quad
	\check\pi(\cdot|s)=\delta_U
	\ \text{for }s\in\{\mathrm E,\mathrm I,\mathrm R\}
	\right\}.
\]
Therefore $|\check\Pi_{\rm SEIR}|=|\check A_{\rm epi}|=11$.

With the configuration in Table~\ref{tab:numerics-config}, SEIR is the most
demanding of the three examples because the state space is larger and the
discount factor $\beta$ is higher.  Nevertheless,
Figure~\ref{fig:numerics-seir-profile} shows that the asynchronous curves remain
close to the idealized reference.  At $\zeta=0.5$, the mean reward from Algorithm~\ref{alg:Qlearning}
increases from $-2.680$ at $m=0$ to $-2.641$ at $m=0.15$, then decreases
slightly to $-2.646$ at $m=1$.  At $\zeta=1$, robustness increases the
mean reward from Algorithm~\ref{alg:Qlearning} from $-2.775$ at $m=0$ to $-2.650$ at $m=0.5$.

\begin{figure}[htbp]
	\centering
	\includegraphics[width=0.84\textwidth]{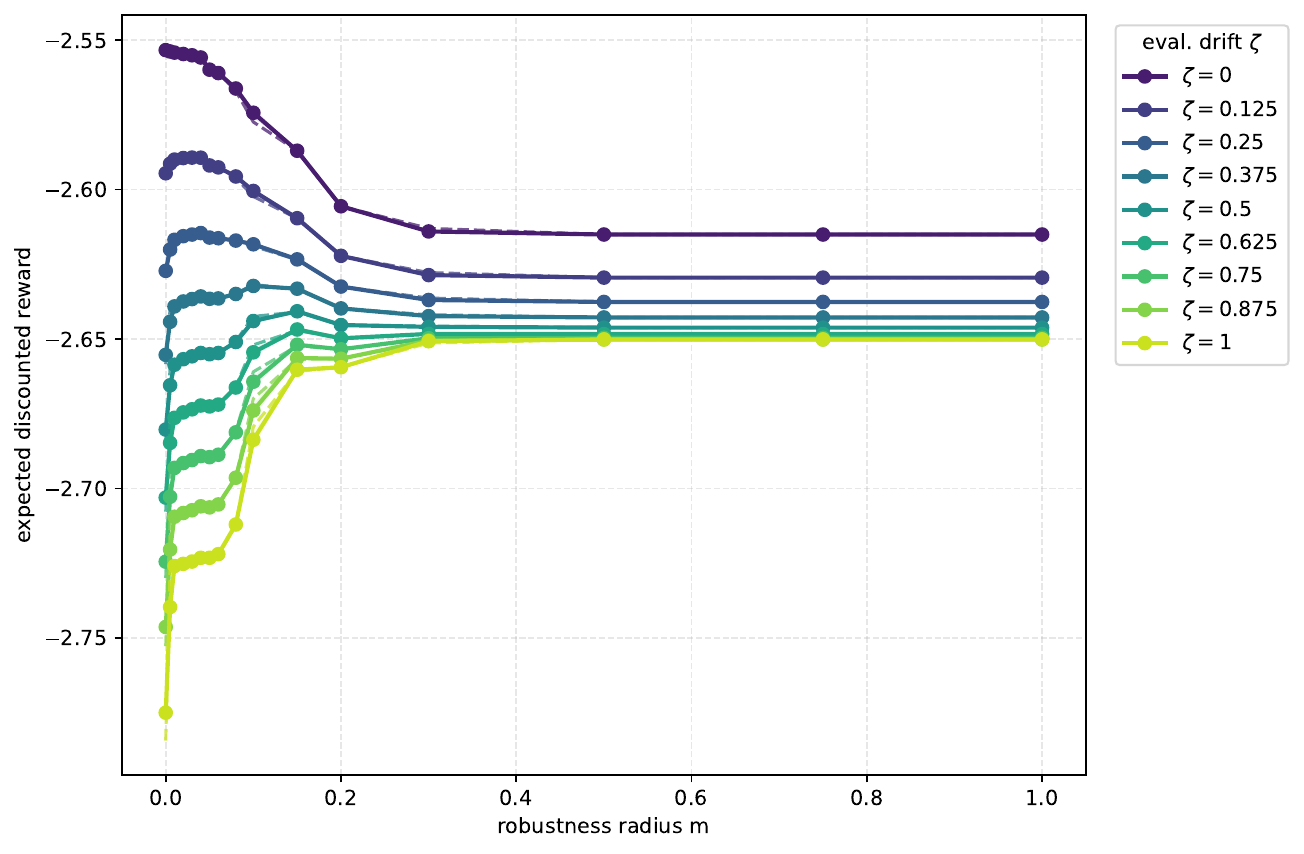}
	\caption{\textbf{SEIR robustness profile.}  Solid curves show reward means
		from the asynchronous implementation of Algorithm~\ref{alg:Qlearning};
		dashed curves show idealized references.  Even at the high discount
		$\beta=0.9$, the robust $Q$-learning algorithm matches the finite-grid idealized benchmark.}
	\label{fig:numerics-seir-profile}
\end{figure}

\newpage 
\section{From Mean-field control to $Q$-learning}\label{sec:MFC}
In this section, we introduce the robust mean-field control (MFC) problem under common noise uncertainty, which serves as the baseline for our $Q$-learning algorithm developed in Section~\ref{sec:Qlearning}. We first show how this MFC problem is linked to the fixed-point $v^*$ established in Proposition~\ref{pro:LNP2025}. In particular, the optimal $Q$-function $Q^*$ defined in~\eqref{eq:optimalQ}, constructed from $v^*$, serves as the learning target for our $Q$-learning algorithm.

Building on this connection, we then introduce a discretized approximation of $Q^*$, which enables the tabular representation of our $Q$-learning algorithm as in Framework \ref{frame:Qlearning}. In particular, the quantization and projection introduced in Definition \ref{dfn:projectedmaps} constitute the key ingredients for constructing the stochastic iterative scheme associated with the discretized $Q$-function. %
\subsection{Robust MFC problem under common noise uncertainty}\label{sec:MFC_framework}
We introduce %
a social planner's robust MFC problem under common noise uncertainty studied in \cite{lauriere2025robust}, which serves as the basis for our $Q$-learning algorithm and its convergence results in Section \ref{sec:Qlearning}.

We begin by constructing a measurable space $(\Omega,{\cal F})$.
Recalling from Section \ref{sec:Qlearning} the idiosyncratic noise space $E$ and common noise space $E^0$, define $G:=[0,1]$ and $\Theta:=[0,1]$, which represent the initial information space and randomization source space, respectively.

We let the space $\Omega$ be defined by
\begin{align*}
	\Omega:= \left\{\omega:=\big((g^i)_{i\in \mathbb{N}},(\theta_t^i)_{t\geq 0, i\in \mathbb{N}},(e_t^i)_{t\geq 1, i\in \mathbb{N}},(e^0_t)_{t\geq 1}\big) \; :
	\begin{aligned}
		 & \; (g^i,\theta^i_t)\in G \times \Theta,\;\;\mbox{for $t\geq 0$},\; i\in \mathbb{N};
		\\
		 & \; (e^i_t, e^0_t)\in E\times E^0,\;\;\mbox{for}\;t\geq 1,\; i\in \mathbb{N}
	\end{aligned}
	\right\}.
\end{align*}
Then define, for every $\omega \in \Omega$,
\begin{align}\label{eq:coordinate}
	\begin{aligned}
		\big(\gamma^i( \omega),\vartheta_0^i(\omega)\big)        & :=(g^i,\theta^i_0)\in G\times \Theta \quad   &  & i \in \mathbb{N},              \\
		\big(\vartheta_t^i(\omega),\varepsilon^i_t( \omega)\big) & :=(\theta^i_t,e_t^i)\in \Theta\times E \quad &  & t\geq 1,\;\; i \in \mathbb{N}, \\
		\varepsilon_t^0(\omega)                                  & :=e_t^0\in E^0\quad                          &  & t\geq 1,
	\end{aligned}
\end{align}
so that $\gamma^i$  and $(\vartheta_t^i)_{t\geq 0}$ represent the initial state information of agent $i$ and her %
randomization source process, respectively. Moreover, $(\varepsilon_t^i)_{t\geq1}$ represents her  idiosyncratic noise process, whereas $(\varepsilon_t^0)_{t\geq1}$ represents the common noise process for all agents.

Last, we let the $\sigma$-algebra $\mathcal F$ be defined by %
${\cal F}:=\sigma((\gamma^i)_{i\in \mathbb{N}},(\vartheta_t^i)_{t\geq 0, i\in \mathbb{N}},(\varepsilon_t^i)_{t\geq 1, i\in \mathbb{N}},(\varepsilon^0_t)_{t\geq 1})$.

Under this setting, %
we introduce the following filtrations: for each $i\in \mathbb{N}$
\begin{itemize}
	\item [$\cdot$] $\mathbb{F}^0:=({\cal F}_t^0)_{t\geq 0}$ is given by ${\cal F}_0^0:=\{\emptyset,\Omega\}$ and ${\cal F}^0_t:=\sigma(\varepsilon_{1:t}^0)$ for any $t\geq 1$.
	\item [$\cdot$] $\mathbb{F}^{i}:=({\cal F}^{i}_t)_{t\geq 0}$ is given by ${\cal F}_0^{i}:= \sigma(\gamma^i)$ and ${\cal F}^{i}_t:=\sigma(\gamma^i,\vartheta_{0:t-1}^i,\varepsilon_{1:t}^i,\varepsilon_{1:t}^0)$ for any $t\geq 1$.%
	\item [$\cdot$] $\mathbb{G}^{i}:=({\cal G}^{i}_t)_{t\geq 0}$ is given by
	      ${\cal G}_t^{i}:={\cal F}_t^{i}\vee \sigma(\vartheta_t^i)$ for all $t\geq 0$ so that $\mathbb{F}^{i}\subseteq\mathbb{G}^{i}$.
\end{itemize}
While ${\cal F}_t^0$ represents the common noise information shared by {all} agents at time $t$, ${\cal F}_t^{i}$ and ${\cal G}_t^{i}$ represent the information of agent $i$ at time $t$ for her state and action processes, respectively (see Definition~\ref{dfn:MFC}). In particular, ${\cal G}_t^{i}$ contains the current randomization source~$\vartheta_t^i$ whereas ${\cal F}_t^{i}$ does not. Consequently, these filtrations satisfy $\mathbb{F}^0\subset \mathbb{F}^i\subset \mathbb{G}^i$.

In what follows, we construct a set of probability measures that support all the random sources defined in \eqref{eq:coordinate} while inducing common noise uncertainty.

To this end, we recall from Section~\ref{sec:intro} the set ${\cal B}^0_{m,q}(\widehat{p}_{\varepsilon^0})$ in~\eqref{eq:uncertainty}, together with the law~$p_\varepsilon\in \overline E$.
\begin{definition}\label{dfn:measures}
	Let $p_\gamma:=p_\vartheta:={\cal L}_{[0,1]}$, where ${\cal L}_{[0,1]}$ denotes the Lebesgue measure on $[0,1]$.
	\begin{itemize}[leftmargin=3.em]
		\item [(i)] Let $\mathcal{K}^0$ be the set of $(p_t)_{t\geq 1}$ consisting of a measure and sequence of kernels such that
		      \begin{align*}
			      \hspace{3.em}  p_1 \in {\cal B}^0_{m,q}(\widehat{p}_{\varepsilon^0});\qquad p_t:(E^0)^{t-1}\ni e_{1:t-1}^0 \mapsto  p_t(de_t^0|e_{1:t-1}^0) \in {\cal B}^0_{m,q}(\widehat{p}_{\varepsilon^0})\quad \mbox{for all $t\geq 2$},
		      \end{align*}
		      inducing distributional uncertainty in the law of the common~noise process $(\varepsilon_t^0)_{t\geq 1}$.
		\item [(ii)] %
		      Let ${\cal Q}$ be defined as the subset of all Borel probability measures~$\mathbb{P}$ on $\Omega$ {induced by} %
		      some $(p_t)_{t\geq 1}\in \mathcal{K}^0$ in the sense that\footnote{For any set $X$ and $t\ge 1$, we use the notation $X^{t}:=X\times \cdots \times X$ for %
			      the $t$-times Cartesian product of $X$. Moreover, we write $X^{\mathbb{N}}$ for the countably infinite Cartesian product of $X$.
		      } for every $B_0\in \bigvee_{i \in \mathbb{N}} \mathcal{G}_0^{i}$ and $B_1 \in  \bigvee_{i \in \mathbb{N}}{\cal G}_1^{i}$
		      \begin{align*}%
			      \hspace{4.em}
			      \begin{aligned}
				       & \mathbb{P}\big\{(\gamma^i,\vartheta_0^i)_{i\in \mathbb{N}}\in B_0\big\}= \hat{Q}_0 %
				      (B_0),\quad \mathbb{P}\big\{((\gamma^i,\vartheta_{0:1}^i, \varepsilon_1^i)_{i\in \mathbb{N}},\varepsilon_1^0)\in B_1\big\}= (\hat{Q}_0 \otimes  \hat Q^{p_1})(B_1),
			      \end{aligned}
		      \end{align*}
		      where  %
		      \begin{align*}
			      \hspace{4.em}
			      \begin{aligned}
				      \hat Q_0\big((dg^i,d\theta_0^i)_{i\in \mathbb{N}}\big)                & :=\mathop{\otimes}\limits_{i \in \mathbb{N}}\big\{p_\gamma(dg^i)p_\vartheta(d\theta^i_0)\big\}\in {\cal P}\big((G\times\Theta)^{\mathbb{N}} \big) \\
				      \hat Q^{p_1}\big((d\theta^i_1, de^i_1)_{i\in \mathbb{N}}, de^0_1\big) & :=
				      \mathop{\otimes}\limits_{i \in \mathbb{N}}\big\{p_\vartheta(d\theta_1^i) p_\varepsilon(de_1^i)
				      \big\}p_1(de^0_1)\in {\cal P}\big((\Theta\times E)^{\mathbb{N}}\times E^0\big), %
			      \end{aligned}
		      \end{align*}
		      whereas for every $t\geq 2$ and $B_t\in  \bigvee_{i \in \mathbb{N}}{\cal G}_t^{i}$
		      \begin{align*}%
			      \hspace{3.em} \mathbb{P}\big\{\big((\gamma^i,\vartheta^i_{0:t}, \varepsilon_{1:t}^i)_{i\in \mathbb{N}},\varepsilon_{1:t}^0\big)\in B_t\big\}= (\hat Q_0 \otimes \hat Q^{p_1}\mathbin{\hat \otimes} \hat Q^{p_2}\mathbin{\hat \otimes} \cdots \mathbin{\hat \otimes} \hat Q^{p_t})(B_t),
		      \end{align*}
		      where the kernel $\hat Q^{p_t}:(E^0)^{t-1}\ni e_{1:t-1}^0\mapsto \hat Q^{p_t}((d\theta_t^{i},de_t^{i})_{i\in\mathbb{N}},de_t^0| e_{1:t-1}^0)\in {\cal P}((\Theta\times E)^{\mathbb{N}}\times E^0)$ is defined by
		      \begin{align*}%
			      \hspace{3.em}\hat Q^{p_t}\big((d\theta_t^{i},de_t^{i})_{i\in\mathbb{N}},de_t^0| e_{1:t-1}^0\big):=\mathop{\otimes}\limits_{i \in \mathbb{N}}\big\{p_\vartheta(d\theta^i_t) p_\varepsilon(de_t^i)
			      \big\}p_t(de_t^0|e_{1:t-1}^0).%
		      \end{align*}
	\end{itemize}
\end{definition}

\begin{remark}\label{rem:well_dfn_1}
	The introduction of $G$ and $\Theta$, together with $p_\gamma$ and $p_\vartheta$ in Definition~\ref{dfn:measures}, allows each agent $i\in \mathbb{N}$ to randomize the optimal policy determined by the social planner. This aligns with the existing literature on MFC problems with common noise; see, e.g.,~\cite{carmona2023model,lauriere2025robust,motte2022mean}.

	Moreover, the following statements hold:
	For any $\mathbb{P}\in {\cal Q}$ with corresponding $(p_t)_{t\geq 1}\in \mathcal{K}^0$ %
	\begin{itemize}[leftmargin=3.em]
		\item [(i)] $(\gamma^i)_{i\in \mathbb{N}}$ is independent and identically distributed (i.i.d.) with law $p_\gamma$. Moreover, $(\vartheta_t^i)_{t\geq 0,i\in \mathbb{N}}$ is i.i.d. with law $p_{\vartheta}$, and $(\varepsilon_t^i)_{t\geq 1,i\in \mathbb{N}}$ is i.i.d.\;with law $p_\varepsilon$.
		\item [(ii)] $\varepsilon_1^0$ is independent of $\bigvee_{i\in \mathbb{N}}{\cal G}_0^{i}$ with law $p_1$, whereas for every~$t\geq 2$ $\varepsilon_{t}^0$ is conditionally independent of $\bigvee_{i\in \mathbb{N}}{\cal G}_{t-1}^{i}$ given ${\cal F}_{t-1}^0$,
		      satisfying
		      \[
			      \mathscr{L}_{\mathbb{P}}(\varepsilon_{t}^0 | {\cal F}_{t-1}^0)= p_t(\,\cdot\,|\varepsilon_{1:t-1}^0) \quad \mbox{$\mathbb{P}$-a.s..}
		      \]
		      Therefore, $(\varepsilon_t^0)_{t\geq 1}$ is uncertain in the Wasserstein sense specified in \eqref{eq:uncertainty}.
	\end{itemize}
\end{remark}
\begin{remark} Recall the reference measure $p_\varepsilon\in \overline E$.
	Let $(\widehat p_t)_{t\geq 1}\in {\cal K}^0$ be defined by
	\[
		\widehat p_1:=\widehat{p}_{\varepsilon^0},\quad  \widehat{p}_t:(E^0)^{t-1}\ni e_{1:t-1}^0 \mapsto  \widehat{p}_t(de_t^0|e_{1:t-1}^0):=\widehat{p}_{\varepsilon^0}(de_t^0)\quad \mbox{for all $t\geq 2$}.
	\]
	Then we denote by $\widehat{\mathbb{P}}\in {\cal Q}$ the probability measure induced by $(\widehat p_t)_{t\geq 1}$, representing the reference probability measure on $(\Omega, {\cal F})$.
	When $m=0$ (so that ${\cal Q}=\{\widehat{\mathbb P}\}$, i.e., in the absence of uncertainty), the probability framework in Definition \ref{dfn:measures}
	coincides with the setting in \cite[Section~2.1.2]{carmona2023model} and is also similar to the one in \cite[Section~2]{motte2022mean}.
\end{remark}

We now define the social planner's robust MFC problem under open-loop controls setting, following  \cite[Definition~10]{carmona2023model} and \cite[Definition 2.5]{lauriere2025robust}. To this end, we recall from Section \ref{sec:intro} the state space $S$ and action space $A$, together with the basic components $\operatorname{F},r,\beta$ in \eqref{eq:Intro_condMckVla} and \eqref{eq:Intro_robustMFC}.
\begin{definition}\label{dfn:MFC}
	For each $i \in \mathbb{N}$, let $L^0_{\mathcal F_0^i}(S)$ denote the set of $\mathcal F_0^i$-measurable, $S$-valued random variables, representing the initial states of agent $i$.
	\begin{enumerate}[leftmargin=3.em]
		\item [(i)] Denote by $\Pi^{\operatorname{OL}}$ the set of all open-loop policies $\pi:=(\pi_t)_{t\geq 0}$ in the sense that $\pi_t:G\times \Theta^{t+1}\times E^t\times (E^0)^t\to A$ is a Borel measurable function for all $t\geq 0$. For any $\pi:=(\pi_t)_{t\geq 0}\in \Pi^{\operatorname{OL}}$, the corresponding action process of agent $i$ is given by the open-loop control
		      \[
			      \hspace{1.em} a_{t}^{i,\pi}:= \pi_t(\gamma^i,\vartheta_{0:t}^i,\varepsilon_{1:t}^{i},\varepsilon_{1:t}^0)\quad t\geq 1,\quad \mbox{with}\;\; a_0^{i,\pi}:=\pi_0(\gamma^i,\vartheta_0^i).
		      \]
		      In other words, $(a_t^{i,\pi})_{t\geq 0}$ is a $\mathbb{G}^i$-adapted process. %
		\item [(ii)] Let $\xi^i\in L^0_{\mathcal F_0^i}(S)$ be an initial state of agent $i$.  For any $\pi:=(\pi_t)_{t\geq0}\in \Pi^{\operatorname{OL}}$, the state process of agent $i$ under $\mathbb{P}\in {\cal Q}$ is governed by the conditional McKean--Vlasov dynamics:
		      \begin{align}\label{eq:MKV}
			      \hspace{1.em}
			      s_{t+1}^{i,\pi,\mathbb{P}} & :=\operatorname{F}(s^{i,\pi,\mathbb{P}}_t,a^{i,\pi}_t,\Lambda_t^{i,\pi,\mathbb{P}},\varepsilon_{t+1}^{i},\varepsilon_{t+1}^0)\quad t\geq 0,\quad \mbox{with } s_0^{i,\pi,\mathbb{P}}:=\xi^{i},
		      \end{align}
		      where $(a^{i,\pi}_t)_{t\geq 0}$ is the open-loop control of agent $i$ as defined in (i), and $\Lambda_t^{i,\pi,\mathbb{P}}$ is the conditional joint law of $(s^{i,\pi,\mathbb{P}}_t,a^{i,\pi}_t)$ under $\mathbb{P}$ given the common noise trajectory $\varepsilon_{1:t}^0$, i.e.,%
		      \[
			      \hspace{1.em}  \Lambda_t^{i,\pi,\mathbb{P}}:=\mathscr{L}_{\mathbb{P}}\big((s^{i,\pi,\mathbb{P}}_t,a^{i,\pi}_t) | \varepsilon_{1:t}^0\big)\quad t\geq 1,
		      \]
		      with the convention that $\Lambda_0^{i,\pi,\mathbb{P}}:=\mathscr{L}_{\mathbb{P}}((s^{i,\pi,\mathbb{P}}_0,a^{i,\pi}_0))$. %
		\item[(iii)] The contribution of agent $i$ to the social planner’s gain under $\pi:=(\pi_t)_{t\geq 0}\in \Pi^{\operatorname{OL}}$ and $\mathbb{P}\in {\cal Q}$ is defined by%
			\begin{align}\label{eq:reward_i}
				\quad R^{i,\pi,\mathbb{P}}(\xi^i):= \sum_{t=0}^\infty \beta^t r(s_{t}^{i,\pi,\mathbb{P}},a_{t}^{i,\pi},\Lambda_t^{i,\pi,\mathbb{P}}).
			\end{align}
			Then the social planner’s robust MFC problem is defined by
			\begin{align}\label{eq:worst_gain}
				\quad V^i(\xi^i):= \sup_{\pi \in \Pi^{\operatorname{OL}}} {\cal J}^{i,\pi}(\xi^i),\quad \mbox{where } {\cal J}^{i,\pi}(\xi^i):=\inf_{\mathbb{P}\in {\cal Q}}\mathbb{E}^{\mathbb{P}}[R^{i,\pi,\mathbb{P}}(\xi^i)].%
			\end{align}
	\end{enumerate}
\end{definition}

\begin{remark}\label{rem:MDP}
	Recall $\overline{\operatorname{F}},\overline{r}$ defined in Definition \ref{dfn:lifted_compo}. The following statements hold for any $i\in\mathbb{N}$.
	\begin{itemize}
		\item [(i)] By construction of the set ${\cal Q}$ in Definition \ref{dfn:measures}, the law of the initial state $\xi^i\in L_{{\cal F}_0^i}^0(S)$ is invariant w.r.t.~the choice of~supporting measure $\mathbb{P}\in {\cal Q}$ (see also \cite[Section 2.4, footnote~3]{lauriere2025robust}). Therefore, we can and do write $\mathscr{L}(\xi^i):=\mathscr{L}_{\mathbb{P}}(\xi^i)\in \overline S$ for any $\mathbb{P}\in{\cal Q}$.
		\item [(ii)] For $t\geq 1$, let $\mu^{i,\pi,\mathbb{P}}_t$ be the conditional law of $s^{i,\pi,\mathbb{P}}_t$ under $\mathbb{P}$ given~$\varepsilon_{1:t}^0$, with $\mu^{i,\pi,\mathbb{P}}_0:=\mathscr{L}(\xi^i)$. Then both
		      $(\mu^{{i,\pi,\mathbb{P}}}_t)_{t\geq 0}$ and $(\Lambda^{{i,\pi,\mathbb{P}}}_t)_{t\geq 0}$ in \eqref{eq:MKV} are $\mathbb{F}^0$-adapted, satisfying
		      \[
			      \mu^{{i,\pi,\mathbb{P}}}_{t+1}=\overline{\operatorname{F}}(\Lambda_t^{i,\pi,\mathbb{P}},\varepsilon_{t+1}^0)\quad \mbox{$\mathbb{P}$-a.s., for all $t\geq 0$.}
		      \]
		      Moreover, under the boundedness condition on $r$ in Assumption~\ref{as:MFC}\;(ii), for any $\pi:=(\pi_t)_{t\geq 0}\in \Pi^{\operatorname{OL}}$ and $\mathbb{P}\in {\cal Q}$, $R^{i,\pi,\mathbb{P}}(\xi^i)$ in \eqref{eq:reward_i} satisfies
		      \[
			      \mathbb{E}^{\mathbb{P}}[R^{i,\pi,\mathbb{P}}(\xi^i)]=\mathbb{E}^{\mathbb{P}}\bigg[\sum_{t=0}^\infty\beta^t\overline r (\Lambda_t^{i,\pi,\mathbb{P}})\bigg].%
		      \]
		      This implies that the robust MFC problem $V^i$ in \eqref{eq:worst_gain} can be formulated as a robust Markov decision process (MDP) on the probability space $\overline S$, as shown in \cite[Proposition~2.12 and Remark 2.13]{lauriere2025robust}.
	\end{itemize}
\end{remark}

The following theorem shows that the robust MFC problem $V^i(\xi^i)$ in \eqref{eq:worst_gain} is {\it law invariant} in the sense that $V^i(\xi^i)$ depends on its initial state $\xi^i$ only via its law~$\mathscr{L}(\xi^i)$; see~Remark \ref{rem:MDP}\;(i). Moreover, it is {\it indistinguishable} in the sense that for any $i\in\mathbb{N}$, $V^i(\xi^i)$ coincides with the unique fixed point~$v^*$ of the Bellman--Isaacs operator ${\cal T}$ defined in~\eqref{eq:bellmanisaacs} (see Proposition \ref{pro:LNP2025}). This result follows from \cite[Theorem~2.21]{lauriere2025robust}.
\begin{theorem}\label{thm:LNP2025}
	Under Assumption~\ref{as:MFC}, %
	it holds for any $i\in\mathbb{N}$ that $v^*(\mathscr{L}(\xi^i))=V^i(\xi^i)$.
\end{theorem}

This result is the verification theorem for the robust MFC problem in Definition~\ref{dfn:MFC}. In particular, it justifies that the optimal $Q$-function $Q^*$ defined in \eqref{eq:optimalQ}, constructed from the fixed point~$v^*$ in Proposition \ref{pro:LNP2025}, indeed serves as the learning target for the robust MFC problem. %

\subsection{Discretized $Q$-function}\label{sec:DiscretizedQ}
Since the optimal $Q$-function $Q^*$ in \eqref{eq:optimalQ} is defined on a infinite dimensional domain, a direct tabular implementation is not feasible. In this section, we will introduce an approximation of $Q^*$ living only on the finite subset $\check S\times \check \Pi \subset \overline S\times \Pi$, we refer to as the discretized $Q$-function, based on quantization and projection introduced in Definition~\ref{dfn:projectedmaps}. Then we will establish the corresponding approximation error bounds.

To this end, recalling the mapping $\overline{\operatorname{F}}$ in Definition \ref{dfn:lifted_compo}\;(i) and the projection $\operatorname{pj}_{\check S}$ in Definition~\ref{dfn:projectedmaps}\;(ii), we define the projected version of $\overline{\operatorname{F}}$ by %
\begin{align}\label{eq:project_S}
	\check{\operatorname{F}}:\check S\times \check \Pi\times E^0\ni (\check \mu,\check \pi,e^0)\mapsto \check{\operatorname{F}}(\check \mu,\check \pi,e^0):=\operatorname{pj}_{\check{S}}(\overline{\operatorname{F}}(\check \mu\hatotimes \check \pi,e^0))\in \check{S}.
\end{align}

Then we introduce a discretized version $\check {\cal T}$ of the Bellman--Isaacs operator ${\cal T}$ in \eqref{eq:bellmanisaacs}, %
defined on~the set
$\ell^\infty (\check S):=\{\check v:\check S\to\mathbb R$\} endowed with $\|\check v\|_{\check S}:=\max_{\check \mu\in \check S}|\check v(\check \mu)|$,  by
\begin{align}\label{eq:projectedbellmanisaac}
	\check{\cal T} \check v(\check \mu):= \max_{\check \pi \in \check \Pi}\bigg\{\overline{r}(\check \mu\hatotimes \check \pi )+\beta\inf_{p\in {\cal B}^0_{m,q}(\widehat{p}_{\varepsilon^0})} \int_{E^0} \check v\big(\check{\operatorname{F}}(\check \mu,\check \pi,e^0)\big) p(de^0)%
	\bigg\},\quad \check \mu\in \check S.
\end{align}

The following lemma shows that $\check {\cal T}$ admits a unique fixed point.
\begin{lemma}\label{lem:projectedDPP} Suppose that Assumption \ref{as:MFC} is satisfied. Then there exists a unique fixed point $\check v^*\in \ell^\infty (\check S)$ of $\check{\cal T}$ satisfying $\check v^*=\check {\cal T}\check v^*$. Moreover, for any $\check v\in \ell^\infty(\check S)$, $\lim_{n\to \infty} \check {\cal T}^n \check v=\check v^*$.
\end{lemma}
\begin{proof}
	We show that $\check{\mathcal T}$ is a contraction on $(\ell^\infty(\check S),\|\cdot\|_{\check S\times\check \Pi})$. Since $\check v\in \ell^\infty(\check S)$ and the reward function $r$ are bounded (see Remark \ref{rem:regular_lift}), we have %
	\[
		\|\check{\cal T} \check v\|_{\check S}\leq \sup_{(s,a,\Lambda)\in S\times A \times \overline{S\times A}} |r(s,a,\Lambda)|+ \beta \|\check v\|_{\check S}=C_{r,\infty}+ \beta \|\check v\|_{\check S}<\infty,
	\]
	hence, $\check{\cal T}\check v\in \ell^\infty(\check S)$.

	Moreover, for any $\check v,\check w\in \ell^\infty(\check{S})$ and any $\check \mu\in \check{S}$
	\begin{align*}
		 & |\check{\cal T}\check v(\check \mu)-\check{\cal T}\check w(\check \mu)|                                                                                                                                                                                                                                                                                                                  \\
		 & \quad \leq \beta \max_{\check \pi \in \check \Pi}\bigg\{\bigg|\inf_{p\in {\cal B}^0_{m,q}(\widehat{p}_{\varepsilon^0})}\int_{E^0} \check v\big(\check{\operatorname{F}}(\check \mu,\check \pi,e^0)\big) p(de^0)-\inf_{p\in {\cal B}^0_{m,q}(\widehat{p}_{\varepsilon^0})}\int_{E^0} \check w\big(\check{\operatorname{F}}(\check \mu,\check \pi,e^0)\big) p(de^0)\bigg|\bigg\} \nonumber \\
		 & \quad \leq \beta \max_{\check \pi \in \check \Pi} \sup_{p\in {\cal B}^0_{m,q}(\widehat{p}_{\varepsilon^0})}\int \Big|\check v\big(\check{\operatorname{F}}(\check \mu,\check \pi,e^0)\big) -\check w\big(\check{\operatorname{F}}(\check \mu,\check \pi,e^0)\big) \Big|p(de^0)                                                                                                           \\
		 & \quad \leq \beta \|\check v-\check w\|_{\check S}. \nonumber
	\end{align*}
	Furthermore, since $\beta\in[0,1)$, $\check {\cal T}$ is a contraction, as claimed.

	Thus the proof is concluded by an application of Banach's fixed point theorem (see, e.g., \cite[Theorem A 3.5]{bauerle2011markov}).
\end{proof}

Using the fixed point $\check v^*$ from Lemma \ref{lem:projectedDPP}, we define the discretized version of the optimal $Q$-function by setting for every $(\check \mu,\check \pi)\in\check S\times \check \Pi$
\begin{align}\label{eq:projectedoptimalQ}
	\check{Q}^*(\check \mu, \check \pi):= \overline r(\check \mu \hatotimes \check \pi)+ \beta \inf_{p\in {\cal B}^0_{m,q}(\widehat{p}_{\varepsilon^0})} \int_{E^0} \check v^*\big(\check{\operatorname{F}}(\check \mu,\check \pi,e^0)\big) p(de^0).
\end{align}
We have by Lemma \ref{lem:projectedDPP} that $\max_{\check\pi\in\check \Pi}\check{Q}^*(\check \mu, \check \pi)=\check v^*(\check \mu)$ for any $\check \mu\in \check S$.  We will refer to $\check{Q}^*$ as the discretized $Q$-function.

Next we aim to establish a bound for the discretized $Q$-function and quantify its deviation from the optimal $Q$-function in \eqref{eq:optimalQ} on the quantized spaces  $\check S$ and $\check\Pi$.%

To this end, we collect several preliminary results related to $\check S$ and $\check\Pi$.%
\begin{lemma}\label{lem:estimateQ_pre} The following hold:
	\begin{itemize}
		\item [(i)] \(\sup_{\pi \in \Pi}\min_{\check \pi \in \check \Pi}d_{\Pi}(\pi,\check \pi)\leq \varepsilon_{\check A}\), with the constant $\varepsilon_{\check A}>0$ given in Definition \ref{dfn:projectedmaps}.
		\item [(ii)] For any $(\mu,\pi)\in \overline{S}\times \Pi$ and $(\check \mu,\check \pi)\in \check S\times \check \Pi$, %
		      \[
			      {W}_1(\check \mu \hatotimes \check \pi ,\mu \hatotimes \pi)\leq (1+\Delta_A\underline{{\Delta}}_S^{-1}) {W}_1(\check \mu, \mu)+d_{\Pi}(\check \pi, \pi),
		      \]
		      with the constants $\Delta_A$ and $\underline{{\Delta}}_S$ given in \eqref{eq:sep_dis_diam}.
	\end{itemize}
\end{lemma}
\begin{proof}
	We first prove (i). %
	By Definition \ref{dfn:projectedmaps}, for any $\pi\in \Pi$ and $s\in S$, there exists $\nu_s^\pi\in \check A$ such that
	\(
	{W}_1(\pi(\cdot|s),\nu^\pi_s)\leq \varepsilon_{\check A}.
	\)
	Define $\check \pi\in \check \Pi$ by \(
	\check \pi(\cdot|s):= \nu_s^\pi
	\)
	for each $s\in S$.  Since $S$ is finite, we have
	\(
	d_{\Pi}(\pi,\check \pi)=\max_{s\in S} {W}_1(\pi(\cdot|s),\nu_s^\pi)\leq \varepsilon_{\check A}.
	\)
	Hence, $\min_{\check \pi\in \check \Pi}d_{\Pi}(\pi,\check \pi)\leq \varepsilon_{\check A}$, which proves~(i).

	We now prove (ii). Let $(\mu,\pi)\in \overline{S}\times \Pi$ and $(\check \mu,\check \pi)\in \check S\times \check \Pi$. By the triangle inequality,
	\begin{align}\label{eq:coupling1}
		{W}_1(\check \mu \hatotimes \check \pi ,\mu \hatotimes \pi)\leq {W}_1(\check \mu \hatotimes \check \pi ,\check\mu \hatotimes \pi)+{W}_1(\check \mu \hatotimes \pi ,\mu \hatotimes \pi)=:\operatorname{I}+\operatorname{II}.
	\end{align}

	For each $s\in S$, let $\kappa_s\in \operatorname{Cpl}(\check \pi(\cdot|s),\pi(\cdot|s))$ be an optimal\footnote{That is, $\kappa_s$ is a coupling such that %
		\(
		W_1(\check{\pi}(\cdot | s), \pi(\cdot | s))
		= \int_{A \times A} |a-a'| \kappa_s(da,da');
		\)
		see \cite[Theorem~4.1]{villani2008optimal}.
	} coupling. Define %
	\[
		\Gamma^1(ds,da,ds',da'):= \kappa_s(da,da')\delta_{s}(ds')\check\mu(ds)\in \operatorname{Cpl}(\check \mu \hatotimes \check \pi ,\check\mu \hatotimes \pi).
	\]
	Then, by the definition of $W_1$,
	\begin{align}\label{eq:coupling2}
		\operatorname{I}%
		\leq \int_{S}{ W}_1(\check \pi(\cdot|s),\pi(\cdot|s))\check \mu(ds)\leq d_{\Pi}(\check \pi, \pi).
	\end{align}

	Next, let $\gamma\in \operatorname{Cpl(\check\mu,\mu)}$ be an optimal coupling. For each $s,s'\in S$, let $\kappa_{s,s'}\in \operatorname{Cpl}(\pi(\cdot|s),\pi(\cdot|s'))$ be an optimal coupling.

	Define
	\(
	\Gamma^2(ds,da,ds',da'):=\kappa_{s,s'}(da,da') \gamma(ds,ds')\in \operatorname{Cpl}(\check \mu \hatotimes \pi ,\mu \hatotimes \pi).
	\)
	Then,
	\begin{align}\label{eq:coupling3}
		\begin{aligned}
			\operatorname{II} & \leq                                                                                                                                                %
			\int_{S\times S}\big(|s-s'|+{W}_1(\pi(\cdot|s),\pi(\cdot|s'))\big)\gamma(ds,ds')                                                                                        \\
			                  & \leq \int_{S\times S}(|s-s'|+ \Delta_A {\bf 1}_{\{s\neq s'\}})\gamma(ds,ds')                                                                        \\
			                  & \leq (1+\Delta_A\underline{{\Delta}}_S^{-1}) \int_{S\times S}|s-s'|\gamma(ds,ds')= (1+\Delta_A\underline{{\Delta}}_S^{-1})  {W}_1(\check \mu, \mu).
		\end{aligned}
	\end{align}
	Combining \eqref{eq:coupling1}, \eqref{eq:coupling2}, and \eqref{eq:coupling3} yields~(ii). %
\end{proof}

\begin{lemma}\label{lem:liftedoptimalQ}
	Under Assumption \ref{as:MFC}, %
	we have for every~$(\check \mu,\check \pi)\in\check S\times \check \Pi$ that
	\begin{align*}
		|\check Q^*(\check \mu,\check \pi)|                             & \leq \frac{C_{r,\infty}}{1-\beta},                                                                                             \\
		|\check Q^*(\check \mu,\check \pi)- Q^*(\check \mu,\check \pi)| & \leq \frac{\beta}{1-\beta} \big(L^* \varepsilon_{\check S} +2(C_r+\beta L^* C_{\operatorname{F}}) \varepsilon_{\check A}\big),
	\end{align*}
	with the constant $L^*>0$ given in Proposition \ref{pro:LNP2025}.
\end{lemma}
\begin{proof}
	We first show that the first estimate holds. %
	Indeed, since $\max_{\check\pi'\in\check \Pi}\check{Q}^*(\check \mu', \check \pi')=\check v^*(\check \mu')$ for all $\check \mu'\in \check S$ (see \eqref{eq:projectedoptimalQ}), by the boundedness of $r$ (see Remark \ref{rem:regular_lift}), for any $(\check \mu,\check \pi)\in \check S \times \check\Pi$
	\begin{align*}
		|\check Q^*(\check \mu,\check \pi)| & \leq |\overline r(\check \mu \hatotimes\check \pi)| + \beta \inf_{p\in {\cal B}^0_{m,q}(\widehat{p}_{\varepsilon^0})}\int_{E^0} \max_{\check \pi'\in \check \Pi}\Big|\check Q^*(\check{\operatorname{F}}(\check \mu,\check \pi,e^0),\check \pi')\Big|p(de^0) \\
		                                    & \leq {C}_{r,\infty}+\beta \max_{(\check \mu',\check \pi')\in \check S\times \check \Pi}|\check Q^*(\check \mu',\check \pi')|.
	\end{align*}
	This ensures the first estimate to hold. %

	Next we show that the other estimate holds. Let $(\check \mu,\check \pi)\in \check S\times \check \Pi$ be given. For every $e^0\in E^0$, define $\check\mu'(e^0):=\check{\operatorname{F}}(\check \mu,\check \pi,e^0)\in \check S$ and $\mu'(e^0)=\overline{\operatorname{F}}(\check \mu \hatotimes \check \pi,e^0)\in \overline{S}$. Then,%
	\begin{align}\label{eq:est0}
		\begin{aligned}
			 & |\check Q^*(\check \mu,\check \pi)- Q^*(\check \mu,\check \pi)|                                                                                                                                    \\
			 & \quad \leq \beta \sup_{p\in {\cal B}^0_{m,q}(\widehat{p}_{\varepsilon^0})}\int \big|\check v^*(\check\mu'(e^0)) - v^*(\mu'(e^0)) \big|p(de^0)                                                      \\
			 & \quad \leq  \beta  \sup_{p\in {\cal B}^0_{m,q}(\widehat{p}_{\varepsilon^0})} \int_{E^0}\Big(|\check v^*(\check\mu'(e^0))-v^*(\check\mu'(e^0))|+ |v^*(\check\mu'(e^0))-v^*(\mu'(e^0))|\Big)p(de^0). %
		\end{aligned}
	\end{align}

	By the definitions of $\operatorname{pj}_{\check S}$ and $\check{\operatorname{F}}$ (see Definition \ref{dfn:projectedmaps} and \eqref{eq:project_S}) and the $L^*$-Lipschitz continuity of $v^*$ (see Proposition \ref{pro:LNP2025}), it holds that for every $e^0\in E^0$
	\begin{align}\label{eq:est1}
		|v^*(\check\mu'(e^0))-v^*(\mu'(e^0))|\leq L^*{W}_1\big(\check\mu'(e^0),\mu'(e^0)\big)\leq L^*\varepsilon_{\check S}.
	\end{align}

	We now claim that %
	for every $e^0\in E^0$
	\begin{align}\label{eq:est2}
		|\check v^*(\check \mu'(e^0))-v^*(\check \mu'(e^0))|\leq \max_{(\check \mu',\check \pi')\in\check S\times \check \Pi}|\check Q^*(\check \mu',\check\pi')-Q^*(\check \mu',\check\pi')|+2(C_r+\beta L^* C_{\operatorname{F}})  \varepsilon_{\check {A}}.
	\end{align}

	To see this, note that $\max_{\check\pi'\in\check \Pi}\check{Q}^*(\check \mu', \check \pi')=\check v^*(\check \mu')$ for all $\check \mu'\in \check S$ and that $\check \Pi$ is finite. %
	Thus, for any $e^0\in E^0$, there exists an optimizer $\check \pi^{*}(e^0)\in \check\Pi$  such that
	\(
	\check v^*(\check \mu'(e^0))= \check Q^*(\check \mu'(e^0),\check \pi^{*}(e^0)).
	\)

	Then, since $v^*(\check \mu'(e^0))\geq Q^*(\check \mu'(e^0),\check\pi^{*}(e^0))$ for every $e^0\in E^0$ (see Proposition \ref{pro:LNP2025} and \eqref{eq:optimalQ}),
	\begin{align}\label{eq:est3}
		\begin{aligned}
			\check v^*(\check \mu'(e^0))-v^*(\check \mu'(e^0))
			 & \leq \check Q^*(\check \mu'(e^0),\check \pi^{*}(e^0)) -Q^*(\check \mu'(e^0),\check \pi^{*}(e^0)) \\
			 & \leq
			\max_{(\check \mu',\check \pi')\in\check S\times \check \Pi}|\check Q^*(\check \mu',\check\pi')-Q^*(\check \mu',\check\pi')|.
		\end{aligned}
	\end{align}

	Moreover, %
	for any $\delta>0$ and any $e^0\in E^0$, there exists a $\delta$-optimizer $\pi^\delta(e^0)\in \Pi$ such that \[
		v^*(\check \mu'(e^0))=\sup_{\pi\in \Pi}Q^*(\check \mu'(e^0),\pi)\leq Q^*(\check \mu'(e^0),\pi^\delta(e^0))+\delta.
	\]

	Using $\check \pi^*(e^0)\in \check \Pi$ (which maximizes $\check v^*(\check \mu'(e^0))$), $\pi^\delta(e^0)\in \Pi$, and the projection $\operatorname{pj}_{\check \Pi}$ in Definition \ref{dfn:projectedmaps}\;(ii), we have
	\begin{align}
		\check v^*(\check \mu'(e^0))-v^*(\check \mu'(e^0))+\delta
		 & \geq  \check Q^*(\check \mu'(e^0),\check \pi^{*}(e^0))- Q^*(\check \mu'(e^0),\pi^{\delta}(e^0))\nonumber                                                                       \\
		 & = \check Q^*(\check \mu'(e^0),\check \pi^{*}(e^0))-\check Q^*(\check \mu'(e^0),\operatorname{pj}_{\check \Pi}( \pi^{\delta}(e^0)))\nonumber                                    \\
		 & \quad +\check Q^*(\check \mu'(e^0),\operatorname{pj}_{\check \Pi}( \pi^{\delta}(e^0)))-Q^*(\check \mu'(e^0),\operatorname{pj}_{\check \Pi}( \pi^{\delta}(e^0)))\label{eq:est4} \\
		 & \quad +Q^*(\check \mu'(e^0),\operatorname{pj}_{\check \Pi}( \pi^{\delta}(e^0)))-Q^*(\check \mu'(e^0),\pi^{\delta}(e^0))\nonumber                                               \\
		 & =:\operatorname{I}(e^0)+\operatorname{II}(e^0)+\operatorname{III}(e^0).\nonumber
	\end{align}

	By the optimality of $\check \pi^{*}(e^0)$, we have \(
	\operatorname{I}(e^0)\geq 0.
	\)
	Moreover, %
	\[
		\operatorname{II}(e^0)\geq -\max_{(\check \mu',\check \pi')\in\check S\times \check \Pi}|\check Q^*(\check \mu',\check\pi')-Q^*(\check \mu',\check\pi')|.
	\]
	Lastly,
	\begin{align*}
		\begin{aligned}
			\operatorname{III}(e^0)%
			 & \geq -\Big|\overline r\big(\check \mu'(e^0) \hatotimes \operatorname{pj}_{\check \Pi}( \pi^{\delta}(e^0))\big)- \overline r\big(\check \mu'(e^0) \hatotimes \pi^{\delta}(e^0)\big)\Big|                                                                                                                                                                                \\
			 & \quad -  \beta \sup_{p\in {\cal B}^0_{m,q}(\widehat{p}_{\varepsilon^0})}\int_{E^0}\bigg|v^*\Big(\overline{\operatorname{F}}\big(\check \mu'(e^0) \hatotimes \operatorname{pj}_{\check \Pi}( \pi^{\delta}(e^0))\big),\tilde e^0\Big)-v^*\Big(\overline{\operatorname{F}}\big(\check \mu'(e^0) \hatotimes  \pi^{\delta}(e^0)\big),\tilde e^0\Big)  \bigg| p(d\tilde e^0) \\
			 & \geq -2(C_r+\beta L^* C_{\operatorname{F}}) {W}_1\big(\check \mu'(e^0) \hatotimes \operatorname{pj}_{\check \Pi}( \pi^{\delta}(e^0)),\check \mu'(e^0) \hatotimes  \pi^{\delta}(e^0)\big)                                                                                                                                                                               \\
			 & \geq -2(C_r+\beta L^* C_{\operatorname{F}}) d_{\operatorname{\Pi}} (\operatorname{pj}_{\check \Pi}( \pi^{\delta}(e^0)),  \pi^{\delta}(e^0))                                                                                                                                                                                                                            \\
			 & \geq -2(C_r+\beta L^* C_{\operatorname{F}})  \varepsilon_{\check {A}}.
		\end{aligned}
	\end{align*}
	The second inequality follows from the apriori estimates for $\overline r$ and $\overline{\operatorname{F}}$ given in Remark \ref{rem:regular_lift}, together with the $L^*$-Lipschitz continuity of $v^*$ in Proposition~\ref{pro:LNP2025}. The third inequality follows from Lemma~\ref{lem:estimateQ_pre}\;(ii), and the last inequality follows from Lemma~\ref{lem:estimateQ_pre}\;(i).

	Combining the estimates for $\operatorname{I}(e^0),$ $\operatorname{II}(e^0)$ and $\operatorname{III}(e^0)$ with \eqref{eq:est4}, we have
	\begin{align}\label{eq:est5}
		\check v^*(\check \mu'(e^0))-v^*(\check \mu'(e^0))+\delta
		\geq -\max_{(\check \mu',\check \pi')\in\check S\times \check \Pi}|\check Q^*(\check \mu',\check\pi')-Q^*(\check \mu',\check\pi')|-2(C_r+\beta L^* C_{\operatorname{F}})  \varepsilon_{\check {A}}.
	\end{align}
	By letting  $\delta\downarrow 0$ this %
	and then using \eqref{eq:est3}, we have indeed \eqref{eq:est2}, as claimed.

	Finally, combining \eqref{eq:est0}, \eqref{eq:est1} and \eqref{eq:est2} ensures %
	the second inequality to hold.

	This completes the proof.
\end{proof}

The estimate of $|\check Q-Q^*|$ established in Lemma~\ref{lem:liftedoptimalQ} is measured on the quantized space $\check{S} \times \check{\Pi}$, rather than on the original space $\overline{S} \times \Pi$. To obtain the desired discretization error for the optimal $Q$-function, one must account for both the error induced by the projection (from $\overline{S} \times {\Pi}$ onto $\check{S} \times \check{\Pi}$) and the estimate of $|\check Q-Q^*|$ on the quantized spaces.

The following lemma provides such an estimate, quantifying the gap between the optimal $Q$-function and the discretized $Q$-function when evaluated on the original domain via the projection in Definition \ref{dfn:projectedmaps}.

\begin{lemma}\label{lem:proj_errrors}
	Under Assumption \ref{as:MFC}, %
	we have for every~$(\mu,\pi)\in\overline S\times \Pi$ that
	\begin{align*}
		|\check Q^*(\operatorname{pj}_{\check S}(\mu),\operatorname{pj}_{\check \Pi}(\pi))- Q^*(\mu,\pi)|\leq C_1\varepsilon_{\check S}+{C_2}\varepsilon_{\check A},
	\end{align*}
	with the constants $C_1$ and $C_2$ given in Theorem \ref{thm:asynchro}.
\end{lemma}
\begin{proof}
	We first claim that  for any $(\mu,\pi)\in \overline{S}\times \Pi$
	\begin{align}\label{eq:pj_error_Q*}
		|Q^*(\operatorname{pj}_{\check S}(\mu),\operatorname{pj}_{\check \Pi}(\pi))- Q^*(\mu,\pi)|\leq 2(C_r+\beta L^* C_{\operatorname{F}})\big((1+\Delta_A\underline{{\Delta}}_S^{-1})\varepsilon_{\check S}+\varepsilon_{\check A}\big).
	\end{align}
	Indeed, let $(\mu,\pi)\in \overline S\times \Pi$ and set $\check \mu:=\operatorname{pj}_{\check S}(\mu)$ and $\check \pi:=\operatorname{pj}_{\check \Pi}(\pi)$. Then,
	\begin{align}
		\begin{aligned}
			 & |Q^*(\check\mu,\check \pi)- Q^*(\mu,\pi)|                                                                                                                                                                                                                                                                                             \\
			 & \quad \leq|\overline r(\check \mu \hatotimes \check \pi )-\overline r( \mu \hatotimes \pi )|+\beta \sup_{p\in {\cal B}^0_{m,q}(\widehat{p}_{\varepsilon^0})} \int_{E^0}\big|  v^*(\overline{\operatorname{F}}(\check \mu \hatotimes \check \pi,e^0))- v^*(\overline{\operatorname{F}}(\mu \hatotimes \pi,e^0)) \big| p(de^0)\nonumber \\
			 & \quad \leq 2(C_r+\beta L^* C_{\operatorname{F}}) {W}_1(\check \mu \hatotimes \check \pi ,\mu \hatotimes \pi)                                                                                                                                                                                                                          \\
			 & \quad \leq  2(C_r+\beta L^* C_{\operatorname{F}}) \big((1+\Delta_A\underline{{\Delta}}_S^{-1}) {W}_1(\check \mu, \mu)+d_{\Pi}(\check \pi, \pi)\big)                                                                                                                                                                                   \\
			 & \quad \leq   2(C_r+\beta L^* C_{\operatorname{F}})   \big((1+\Delta_A\underline{{\Delta}}_S^{-1}) \varepsilon_{\check S}+\varepsilon_{\check A}\big).
		\end{aligned}
	\end{align}
	The second inequality follows from the apriori estimates for $\overline r$ and $\overline{\operatorname{F}}$ given in Remark \ref{rem:regular_lift} %
	, together with the $L^*$-Lipschitz continuity of $v^*$ (see Proposition \ref{pro:LNP2025}).
	The third inequality follows from Lemma \ref{lem:estimateQ_pre}~(ii), and the last inequality follows from the definition of $(\check \mu,\check \pi)=(\operatorname{pj}_{\check S}(\mu),\operatorname{pj}_{\check \Pi}(\pi))$ (see Definition~\ref{dfn:projectedmaps}) and Lemma~\ref{lem:estimateQ_pre}~(i).

	Combining \eqref{eq:pj_error_Q*} with the second estimate in Lemma \ref{lem:liftedoptimalQ} ensures that for any $(\mu,\pi)\in \overline{S}\times \Pi$
	\begin{align*}
		|\check Q^*(\operatorname{pj}_{\check S}(\mu),\operatorname{pj}_{\check \Pi}(\pi))- Q^*(\mu,\pi)| & \leq|Q^*(\operatorname{pj}_{\check S}(\mu),\operatorname{pj}_{\check \Pi}(\pi))- Q^*(\mu,\pi)|                                                                         \\
		                                                                                                  & \quad +|\check Q^*(\operatorname{pj}_{\check S}(\mu),\operatorname{pj}_{\check \Pi}(\pi))- Q^*(\operatorname{pj}_{\check S}(\mu),\operatorname{pj}_{\check \Pi}(\pi))| \\
		                                                                                                  & \leq C_1\varepsilon_{\check S}+{C_2}\varepsilon_{\check A}.\qedhere
	\end{align*}
\end{proof}

Lemma~\ref{lem:liftedoptimalQ} shows that the discretized $Q$-function in \eqref{eq:projectedoptimalQ} serves as a suitable approximation target for learning the optimal $Q$-function in \eqref{eq:optimalQ}. In particular, the error bound in the lemma coincides with the desired accuracy level in the convergence result of Theorem \ref{thm:asynchro}.

\section{Proofs for Convergence of robust $Q$-Learning Algorithm}\label{sec:proof}
The goal of this section is to provide the proofs for our main results in Theorem \ref{thm:asynchro}, Theorem \ref{thm:asynchro_rate}, and Corollary \ref{cor:Qlearningrate}. To that end, we begin by establishing two key lemmas that are used to derive both the asymptotic convergence and finite-time iteration bound analyses in the following theorems. Throughout this section, we denote by
\[
	\|\check f\|_{\check S\times \check \Pi}:= \max_{(\check \mu,\check \pi)\in \check S\times \check \Pi}|\check f(\check \mu,\check \pi)|\quad \mbox{for any mapping $\check f:\check S\times \check \Pi\to\mathbb{R}$.}
\]

The first lemma establishes the existence of the maximizer $\lambda_{t,(\check \mu, \check \pi)}^*$ in~\eqref{eq:defconjugateJ} for any $(\check \mu, \check \pi)\in \check S\times \check \Pi$ and $t\geq 0$, and provides uniform-in-time bounds for the maximizers $(\lambda_{t,(\cdot, \cdot)}^*)_{t\geq 0}$ and the $Q$-functions $(\check Q_t)_{t\geq 0}$ introduced in Framework~\ref{frame:Qlearning}\;(iii).%

\begin{lemma}\label{lem:Q_est}
	Suppose that Assumptions \ref{as:MFC}\;(i),(ii) and \ref{as:Qlearning}\;(i) are satisfied. Then, for both the synchronous and asynchronous versions of the $Q$-learning algorithm in Framework \ref{frame:Qlearning}, the following statements hold: for every $t\geq0$ and $(\check \mu,\check \pi)\in \check S\times\check \Pi$,
	\begin{itemize}
		\item [(i)] There exists a maximizer $\lambda_{t,(\check \mu, \check \pi)}^*$ of $\Phi_{t,(\check\mu,\check \pi)}$ in \eqref{eq:defconjugateJ}.
		\item [(ii)] All maximizers $\lambda_{t,(\check \mu, \check \pi)}^*$ of $\Phi_{t,(\check\mu,\check \pi)}$ satisfy $\lambda_{t,(\check \mu, \check \pi)}^*\leq \frac{2 \check Q_{\operatorname{max}}}{m^q}$.
		\item [(iii)] \(
		      \|\check Q_t\|_{\check S\times \check \Pi}\leq \check Q_{\operatorname{max}}.\)
	\end{itemize}
\end{lemma}
\begin{proof}
	We first consider the synchronous case under Framework \ref{frame:Qlearning}\;(i),(iii).

	\noindent \emph{(Step 1)}
	Assume first that $\|\check Q_t\|_{\check S\times \check \Pi}\leq \check Q_{\operatorname{max}}$ for some $t\ge0$. We first show the existence of the maximizer $\lambda_{t,(\check \mu, \check \pi)}^*$ in \eqref{eq:defconjugateJ} for any $(\check \mu,\check \pi)\in \check S\times\check \Pi$.

	To see this, we recall the definition of $J_{t,(\check \mu, \check \pi)}$ in \eqref{eq:defJ}. For $\lambda\geq 0$ and  $e^0\in E^0$, it holds that%
	\begin{align*}
		-(-J_{t,(\check \mu, \check \pi)})^{\lambda }(e^0)%
		=\min_{\tilde e^0\in E^0} \bigg\{\max_{\check\pi'\in \check \Pi}\check {Q}_t(\check{\operatorname{F}}(\check \mu,\check \pi,\tilde e^0),\check\pi'\big) + \lambda |e^0-\tilde e^0|^q\bigg\}.%
	\end{align*}
	Since $E^0$ is finite, the map $\ni\lambda\mapsto-(-J_{t,(\check\mu,\check\pi)})^\lambda(e^0)$ is concave and continuous for any~$e^0\in E^0$. %

	Moreover, since for every  $\lambda\ge0$ and $e^0\in E^0$
	\begin{align}\label{eq:async_step1_bound}
		\big|(-J_{t,(\check\mu,\check\pi)})^\lambda(e^0)\big|
		\leq \|\check Q_t\|_{\check S\times \check \Pi}+\lambda \min_{\tilde e^0\in E^0} |e^0-\tilde e^0|^q \leq \check{Q}_{\operatorname{M}},
	\end{align}
	the map $\phi_{t,(\check\mu,\check \pi)}(\cdot)$ in \eqref{eq:defconjugateJ} is continuous and satisfies
	\begin{align}\label{eq:bound_lambda}
		\phi_{t,(\check\mu,\check \pi)}(\lambda)\leq \check Q_{\operatorname{max}} -m^q \lambda =: \overline{\phi}(\lambda) \quad \mbox{for all $\lambda \geq 0$}.
	\end{align}
	This implies $\limsup_{\lambda \to \infty}\phi_{t,(\check\mu,\check \pi)}(\lambda)=-\infty$. Therefore, $\lambda_{t,(\check \mu, \check \pi)}^*$ exists for any $(\check \mu,\check \pi)\in \check S\times\check \Pi$.

	Then we show the upper bound for $\lambda_{t,(\check \mu, \check \pi)}^*$. To see this, we note that for any $(\check \mu,\check \pi)\in \check S\times\check \Pi$
	\begin{align}\label{eq:lower_bound}
		\begin{aligned}
			\Phi_{t,(\check\mu,\check \pi)}\geq \phi_{t,(\check \mu,\check \pi)}(0) & = \min_{\tilde e^0\in E^0} J_{t,(\check \mu, \check \pi)}(\tilde e^0)\geq -\|\check Q_t\|_{\check S\times \check \Pi}\geq -\check Q_{\operatorname{max}}.
		\end{aligned}
	\end{align}
	Since $\overline\phi(\overline \lambda)=-\check Q_{\operatorname{max}}<0$ with $\overline{\lambda}:=\frac{2\check Q_{\operatorname{max}}}{m^q}$,  we have by \eqref{eq:bound_lambda}  that   for every $\lambda \geq \overline \lambda$,
	\[
		\phi_{t,(\check \mu,\check \pi)}(\lambda)\leq \overline \phi(\lambda) \leq -\check Q_{\operatorname{max}}.\]

	Combining this with \eqref{eq:lower_bound} yields that  $\lambda_{t,(\check \mu, \check \pi)}^*\leq \overline \lambda$ for any $(\check \mu,\check \pi)\in \check S\times\check \Pi$.

	\noindent \emph{(Step 2)} We now establish the bounds for $\lambda_{t,(\check \mu,\check \pi)}^*$ and $\check Q_t$ inductively over time $t\geq 0$: %
	The claim holds at $t=0$ by assumption, hence we have by Step 1 that for any $(\check \mu,\check \pi)\in \check S\times\check \Pi$,  $\lambda_{0,(\check \mu, \check \pi)}^*\leq \overline \lambda$.

	Fix $t\ge0$ and suppose $\|\check Q_t\|_{\check S\times \check \Pi}\leq \check Q_{\operatorname{max}}$. %
	Then for any $(\check \mu,\check \pi)\in \check S\times\check \Pi$, %
	we have %
	\begin{align}
		\check Q_{t+1}(\check \mu,\check \pi) & \leq (1-\alpha_{t,(\check \mu,\check \pi)})|\check Q_{t}(\check \mu,\check \pi)|+ \alpha_{t,(\check \mu,\check \pi)}\Big(|\overline r(\check \mu \hatotimes\check \pi)|-\beta (-J_{t,(\check \mu, \check \pi)})^{\lambda_{t,(\check \mu, \check \pi)}^*}(\varepsilon_{t+1,(\check \mu, \check \pi)}^0) \Big)\nonumber \\
		                                      & \leq (1-\alpha_{t,(\check \mu,\check \pi)})\check Q_{\operatorname{max}}+ \alpha_{t,(\check \mu,\check \pi)}\Big(
		C_{r,\infty}+\beta \big| (-J_{t,(\check \mu, \check \pi)})^{\lambda_{t,(\check \mu, \check \pi)}^*}(\varepsilon_{t+1,(\check \mu,\check \pi)}^0)\big| \Big) \label{eq:upper_Q1}                                                                                                                                                                               \\
		                                      & \leq (1-\alpha_{t,(\check \mu,\check \pi)})\check Q_{\operatorname{max}} +\alpha_{t,(\check \mu,\check \pi)}(C_{r,\infty}+\beta \check Q_{\operatorname{max}})=\check Q_{\operatorname{max}}(1-2\alpha_{t,(\check \mu,\check \pi)}\beta) < \check Q_{\operatorname{max}}, \nonumber
	\end{align}
	where the second inequality follows from the boundedness of $r$ (see Remark \ref{rem:regular_lift}) and the third inequality follows from \eqref{eq:async_step1_bound}.

	In a similar manner, by using \(
	\lambda_{t,(\check \mu, \check \pi)}^*\leq \overline \lambda,
	\)
	we have for any $(\check \mu,\check \pi)\in \check S\times\check \Pi$,
	\begin{align*}
		\check Q_{t+1}(\check \mu,\check \pi)\geq & -(1-\alpha_{t,(\check \mu,\check \pi)})|\check Q_t(\check \mu,\check \pi) |                                                                                                                                                                                                                                                            \\
		                                          & + \alpha_{t,(\check \mu,\check \pi)} \Big(- |\overline r(\check \mu \hatotimes\check \pi)| + \beta  \Big[-(-J_{t,(\check \mu, \check \pi)})^{\lambda_{t,(\check \mu, \check \pi)}^*}(\varepsilon_{t+1,(\check \mu, \check \pi)}^0)-m^q \lambda_{t,(\check \mu,\check \pi)}^* \Big] \Big) \nonumber                                     \\
		\geq                                      & -(1-\alpha_{t,(\check \mu,\check \pi)})\check Q_{\operatorname{max}} - \alpha_{t,(\check \mu,\check \pi)}\Big( C_{r,\infty} + \beta \big|(-J_{t,(\check \mu, \check \pi)})^{\lambda_{t,(\check \mu, \check \pi)}^*}(\varepsilon_{t+1,(\check \mu, \check \pi)}^0)\big|+\beta  m^q \lambda_{t,(\check \mu,\check \pi)}^*\Big) \nonumber \\
		\geq                                      & -(1-\alpha_{t,(\check \mu,\check \pi)}) \check Q_{\operatorname{max}}- \alpha_{t,(\check \mu,\check \pi)} ( C_{r,\infty}+3\beta \check Q_{\operatorname{max}})=-\check Q_{\operatorname{max}},\nonumber                                                                                                                                %
	\end{align*}
	where the last equality holds because $\check Q_{\operatorname{max}}=\frac{C_{r,\infty}}{1-3\beta}$ with $\beta<\frac{1}{3}$ (see Assumption \ref{as:Qlearning}\;(i)).

	Combining this with \eqref{eq:upper_Q1} yields $\|\check Q_{t+1}\|_{\check S\times \check \Pi}\leq \check Q_{\operatorname{max}}$.

	By induction, the estimate in (iii) for $\check Q_t$ holds for all $t$ and $(\check\mu,\check\pi)\in \check S\times \check \Pi$. Step~1 then ensures that the other estimate in (ii) %
	$\lambda^*_{t,(\check\mu,\check\pi)}$ hold for all $t$ and $(\check\mu,\check\pi)\in \check S\times \check \Pi$.

	For the case of the asynchronous version, the $Q$-learning update rule in \eqref{eq:asynchro_Q1} occurs only along the pre-sampled, projected dataset $(\check \mu_t,\check\pi_t)_{t\geq 0}\subseteq \check S\times \check \Pi$ in Framework \ref{frame:Qlearning}\;(ii). Therefore, the arguments of Steps 1 and 2 can directly be applied also for the asynchronous case, yielding the existence of $(\lambda^*_{t,(\cdot,\cdot)})_{t\geq 0}$ and the bounds for $(\lambda^*_{t,(\cdot,\cdot)})_{t\geq 0}$ and $(\check Q_t)_{t\geq 0}$ given in (ii) and (iii).

	This completes the proof.
\end{proof}

Using Remark \ref{rem:dual_Phi}\;(ii), we can rewrite the update rule~\eqref{eq:asynchro_Q1} in the standard form of stochastic iterative algorithms considered in \cite{bertsekas2025neuro}: for every $t\geq 0$ and $(\check \mu,\check\pi)\in \check S\times \check \Pi$,
\begin{align}\label{eq:asynchro_Q11}
	\begin{aligned}
		 & \check{Q}_{t+1}(\check \mu, \check \pi)=(1-\alpha_{t,(\check \mu, \check \pi)})\check{Q}_{t}(\check \mu, \check \pi)  +\alpha_{t,(\check \mu, \check \pi)}\big(                                                                                                              %
		\check {\cal H}\check Q_t( \check \mu,\check \pi) + \beta  Z_{t+1,(\check \mu,\check \pi)}\big),                                                                                                                                                                                \\
		 & \quad \mbox{where}\quad Z_{t+1,(\check \mu,\check \pi)}:= \Big[-(-J_{t,(\check \mu, \check \pi)})^{\lambda_{t,(\check \mu, \check \pi)}^*}(\varepsilon_{t+1,(\check \mu, \check \pi)}^0)-m^q \lambda_{t,(\check \mu, \check \pi)}^* \Big] - \Phi_{t,(\check\mu,\check \pi)}, %
	\end{aligned}
\end{align}
where the operator $\check{\mathcal H}$ on the set $\{\check Q:\check S\times \check \Pi\to \mathbb{R}\}$ is defined by setting
for every $(\check \mu,\check \pi)\in \check S\times \check \Pi$,
\begin{align}\label{eq:projectedH}
	\check{{\cal H}}\check{Q}(\check\mu,\check\pi):= \overline r(\check \mu \hatotimes \check \pi)+ \beta \inf_{p\in {\cal B}^0_{m,q}(\widehat{p}_{\varepsilon^0})}\int_{E^0}\max_{\check\pi'\in \check \Pi}\check {Q}\big(\check{\operatorname{F}}(\check \mu,\check \pi,e^0),\check\pi'\big)p(de^0).
\end{align}

Next we let $\check{\mathbb{F}}^0:=(\check{\cal F}_{t}^0)_{t\geq 0}$ be defined by
\begin{align}\label{eq:filtration_mupi}
	\check{\cal F}_{t}^0
	:=
	\sigma\big(
	\varepsilon^0_{s,(\check \mu,\check \pi)}
	:(\check \mu,\check \pi)\in \check S\times \check \Pi,1\le s \le t\big),
	\quad
	t\geq 1,\quad \mbox{with $\check{\cal F}_{0}^0
			:=
			\{\emptyset,\Omega^0\}.$}
\end{align}

\begin{remark}\label{rem:fixed_checkQ}
	By definition of $\check Q^*$ in \eqref{eq:projectedoptimalQ} and the operator $\check {\cal H}$ in \eqref{eq:projectedH}, we have $\check Q^*= \check {\cal H}\check Q^*$.
\end{remark}

\begin{remark}\label{rem:msr_indep} Recall the probability space $(\Omega^0,{\cal F}^0,\widehat{\mathbb P}^0)$ introduced in Framework~\ref{frame:Qlearning} and the filtration $\check{\mathbb{F}}^0$ defined in \eqref{eq:filtration_mupi}. The following hold for every $t\geq 0$ and every $(\check \mu,\check \pi)\in \check S\times \check \Pi$.
	\begin{itemize}
		\item[(i)] %
			The random variable $Z_{t+1,(\check \mu,\check \pi)}$ in \eqref{eq:asynchro_Q11} is $\check{\cal F}_{t+1}^0$-measurable,
		\item[(ii)] $\Phi_{t,(\check\mu,\check \pi)}$, $\lambda_{t,(\check \mu, \check \pi)}^*$, and \(
			(-J_{t,(\check \mu, \check \pi)})^{\lambda_{t,(\check \mu, \check \pi)}^*}
			:
			E^0 \to \mathbb R
			\)
			are $\check{\cal F}_{t}^0$-measurable.
	\end{itemize}
\end{remark}

Leveraging the convex duality result in Lemma \ref{lem:lctransform}, and
the uniform-in-time bound result for~$(\check Q_t)_{t\geq 0}$ in Lemma \ref{lem:Q_est}, it is possible to show that $Z_{t+1,(\cdot,\cdot)}$ defined in \eqref{eq:asynchro_Q11} is bounded and has conditional zero mean with respect to $\check{\cal F}^0_{t}$ in \eqref{eq:filtration_mupi}.
\begin{lemma}\label{lem:mtg_z}
	Suppose that Assumptions \ref{as:MFC}\;(i),(ii) and \ref{as:Qlearning}\;(i) are satisfied. Then, for both the synchronous and asynchronous versions of the  $Q$-learning algorithm in Framework \ref{frame:Qlearning}, the following properties hold: for every $t\geq0$ and $(\check \mu,\check \pi)\in \check S\times\check \Pi$,  %
	\[
		\mathbb{E}^{\widehat{\mathbb{P}}^0}[Z_{t+1,(\check \mu,\check \pi)}|\check{\cal F}_{t}^0]=0\quad \mbox{$\widehat{\mathbb{P}}^0$-a.s.},\quad \mbox{and}\quad |Z_{t+1,(\check \mu,\check \pi)}| \leq 4 \check Q_{\operatorname{max}}.
	\]
	As a consequence, we also obtain that %
	\(
	{\rm \mathbb{V}ar}^{\widehat{\mathbb P}^0}[Z_{t+1,(\check \mu,\check \pi)}|\check{\cal F}_{t}^0]
	\le 16\check Q_{\operatorname{max}}^2\) $\widehat{\mathbb{P}}^0$-a.s..
\end{lemma}
\begin{proof}
	The proof presented below applies to both the synchronous and asynchronous versions.

	Let $t\geq 0$ and $(\check \mu,\check \pi)\in \check S\times \check \Pi$. By the definition of $Z_{t+1,(\check \mu,\check \pi)}$ in \eqref{eq:asynchro_Q11} and the $\check{\cal F}_{t}^0$-measurability of $\Phi_{t,(\check \mu,\check \pi)}$ and $\lambda_{t,(\check \mu, \check \pi)}^*$ in \eqref{eq:defconjugateJ} (see Remark \ref{rem:msr_indep}\;(ii)), we have that $\widehat{\mathbb{P}}^0$-a.s.,
	\begin{align}\label{eq:worst1}
		\mathbb{E}^{\widehat{\mathbb{P}}^0}[Z_{t+1,(\check \mu,\check \pi)}|\check{\cal F}_{t}^0] + \Phi_{t,(\check \mu,\check \pi)}=\mathbb{E}^{\widehat{\mathbb{P}}^0}\Big[-(-J_{t,(\check \mu, \check \pi)})^{\lambda_{t,(\check \mu, \check \pi)}^*}(\varepsilon_{t+1,(\check \mu, \check \pi)}^0)\Big|\check{\cal F}_{t}^0\Big]-m^q \lambda_{t,(\check \mu, \check \pi)}^* =:\operatorname{I}.%
	\end{align}

	Moreover, since \(
	(-J_{t,(\check \mu, \check \pi)})^{\lambda_{t,(\check \mu, \check \pi)}^*}
	:
	E^0 \to \mathbb R
	\) is ${\cal F}_t^0$-measurable (see Remark \ref{rem:msr_indep}\;(ii)) and the family
	\((\varepsilon^0_{t,(\check\mu,\check\pi)})_{
			t\ge 1,
			(\check\mu,\check\pi)\in\check S\times\check\Pi
		}\subseteq E^0,
	\)
	is independent and identically distributed according to the law $\widehat{p}_{\varepsilon^0}$ (see Framework \ref{frame:Qlearning}), we have by Remark \ref{rem:dual_Phi}\;(ii) (see also Lemma \ref{lem:lctransform}) that $\widehat{\mathbb{P}}^0$-a.s.,
	\begin{align*}
		\operatorname{I}%
		 & =\int_{E^0}\Big(-(-J_{t,(\check \mu, \check \pi)})^{\lambda_{t,(\check \mu, \check \pi)}^*}(e^0)\Big)\widehat{p}_{\varepsilon^0}(de^0)-m^q \lambda_{t,(\check \mu, \check \pi)}^*=\phi_{t,(\check \mu,\check \pi)}(\lambda_{t,(\check \mu,\check \pi)}^*)=\Phi_{t,(\check\mu,\check \pi)}, %
	\end{align*}
	where the last two equalities follow from definition of $\lambda_{t,(\check \mu,\check \pi)}^*$ and \eqref{eq:defconjugateJ}.

	Combining this with \eqref{eq:worst1} yields that $\mathbb{E}^{\widehat{\mathbb{P}}^0}[Z_{t+1,(\check \mu,\check \pi)}|\check{\cal F}_{t}^0] =0$ $\widehat{\mathbb{P}}^0$-a.s..

	Next we show that $|Z_{t+1,(\check \mu,\check \pi)}| \leq 4 \check Q_{\operatorname{max}}$. Indeed, it follows from \eqref{eq:dual_Phi} in Remark \ref{rem:dual_Phi} and Lemma \ref{lem:Q_est} and \eqref{eq:async_step1_bound} (see the proof of Lemma \ref{lem:Q_est}) that
	\begin{align*}
		\begin{aligned}
			 & |Z_{t+1,(\check \mu,\check \pi)}|                                                                                                                                                                                                                                                                                                                                               \\
			 & \quad \leq \Big|(-J_{t,(\check \mu, \check \pi)})^{\lambda_{t,(\check \mu, \check \pi)}^*}(\varepsilon_{t+1,(\check \mu, \check \pi)}^0)\Big|+m ^q \lambda_{t,(\check \mu, \check \pi)}^*+\sup_{p\in \mathcal{B}_{m,q}^0(\widehat p^0)}\int_{E^0}\max_{\check \pi'\in \check\Pi}\Big|\check {Q}_t(\check{\operatorname{F}}(\check \mu,\check \pi, e^0),\check\pi')\Big|p(d e^0) \\
			 & \quad \leq 3\check Q_{\operatorname{max}} + \|\check Q_t\|_{\check S\times \check \Pi}\leq 4 \check Q_{\operatorname{max}}.
		\end{aligned}
	\end{align*}

	Last, we have ${\rm \mathbb{V}ar}^{\widehat{\mathbb P}^0}[Z_{t+1,(\check \mu,\check \pi)}|\check{\cal F}_{t}^0]= \mathbb{E}^{\widehat{\mathbb P}^0}[|Z_{t+1,(\check \mu,\check \pi)}|^2|\check{\cal F}_{t}^0]\leq 16 \check Q_{\operatorname{max}}^2$ $\widehat{\mathbb{P}}^0$-a.s..
\end{proof}

\subsection{Proof of Theorem \ref{thm:asynchro}}
The proof of Theorem \ref{thm:asynchro} is based on the convergence result for stochastic iterative algorithms in~\cite[Theorem~1]{6796861}, which we introduce in Lemma \ref{lem:jaakkola}. In particular, the uniform-in-time bound for $(\check Q_t)_{t\geq 0}$ established in Lemma~\ref{lem:Q_est}\,(iii), together with the properties of $(Z_{t+1,(\cdot,\cdot)})_{t\geq 0}$ established in Lemma \ref{lem:mtg_z}, are crucial for verifying that both the synchronous and asynchronous versions of our $Q$-learning algorithm in Framework \ref{frame:Qlearning} satisfy the hypotheses of~\cite[Theorem~1]{6796861}. Applying the convergence result, we obtain that for every $(\check \mu,\check \pi)\in \check S \times \check \Pi$
\[
	\lim_{t\to \infty}|\check Q_t(\check \mu,\check \pi)- \check Q^*(\check \mu,\check \pi)|=0\quad \mbox{$\widehat {\mathbb{P}}^0$-a.s.},
\]
where $\check Q^*$ is the discretized $Q$-function in \eqref{eq:projectedoptimalQ}.

The desired convergence result presented in Theorem \ref{thm:asynchro} then follows from the discretization error estimate for the optimal $Q$-function given in Lemma \ref{lem:proj_errrors}.
\begin{lemma}[Theorem~1 in Jaakkola et al.~\cite{6796861}]\label{lem:jaakkola}
	\label{lem:jaakkola1994_thm1}
	Let $X$ and $Y$ be some finite spaces. Consider a family of stochastic processes
	\(
	(\alpha_t(x,y), F_t(x,y), \Delta_t(x,y))_{t\geq 0,(x,y)\in X\times Y}
	\)
	on a probability measure space $(\widetilde \Omega,\widetilde{\mathcal{F}},\widetilde {\mathbb{P}})$,
	satisfying, for every $t\geq 0$ and $(x,y)\in X\times Y$
	\[
		\Delta_{t+1}(x,y)
		= \bigl(1-\alpha_t(x,y)\bigr)\Delta_t(x,y) + \alpha_t(x,y)F_t(x,y)
		\quad \widetilde{\mathbb{P}}\text{-a.s.}.
	\]
	Let $(\widetilde{\mathcal{F}}_t)_{t\geq 0}\subseteq\widetilde{\mathcal{F}}$ be a sequence of increasing
	$\sigma$-algebras such that $\Delta_0(x,y)$ and $\alpha_0(x,y)$ are $\widetilde{\mathcal{F}}_0$-measurable for all $(x,y)\in X\times Y$, whereas
	$\Delta_t(x,y)$, $\alpha_t(x,y)$, and $F_{t-1}(x,y)$ are $\widetilde{\mathcal{F}}_t$-measurable for all $t\geq 1$ and $(x,y)\in X\times Y$.
	Furthermore, assume that for every $t\ge 0$ and  $(x,y)\in X\times Y$,~$\widetilde{\mathbb{P}}$-a.s.,
	\begin{enumerate}
		\item[(i)] $0\leq  \alpha_t(x,y)\leq 1$,
			\(
			\sum_{t=0}^{\infty}\alpha_t(x,y)=\infty\),
			and
			\(
			\sum_{t=0}^{\infty}\alpha_t^2(x,y)<\infty\).
		\item[(ii)] There exists some constant $\delta\in(0,1)$ such that
			\(
			\|\mathbb{E}^{\widetilde{\mathbb{P}}}[F_t(\cdot,\cdot)|\widetilde{\mathcal{F}}_t]\|_{X\times Y}
			\le \delta\,\|\Delta_t\|_{X\times Y}\), where we denote by $\|f\|_{X\times Y}:= \max_{(x,y)\in X\times Y}|f(x,y)|$ for any mapping $f:X\times Y\to\mathbb{R}$.
		\item[(iii)] There exists some constant $C>0$ such that \(
			\|{\rm \mathbb{V}ar}^{\widetilde{\mathbb{P}}}[F_t(\cdot,\cdot)|\widetilde{\mathcal{F}}_t]\|_{X\times Y}
			\le C(1+\|\Delta_t\|_{X\times Y}^2),%
			\)
	\end{enumerate}
	Then $\lim_{t\to\infty}\Delta_t(x,y)=0$ $\widetilde{\mathbb{P}}$-a.s.~for any $(x,y)\in X\times Y$.
\end{lemma}

Recalling the iterative $Q$-functions $(\check Q_t)_{t\geq 0}$ given in Framework \ref{frame:Qlearning}\;(iii) and the discretized optimal $Q$-function $\check Q^*$ given in \eqref{eq:projectedoptimalQ}, we define, for every $t\geq 0$ and $(\check \mu,\check \pi)\in \check S\times \check \Pi$,
\begin{align}\label{eq:gap}
	\check \Delta_t(\check \mu,\check \pi):= \check Q_{t}(\check \mu,\check \pi)-\check Q^*(\check \mu,\check \pi).
\end{align}

By Remark \ref{rem:fixed_checkQ} and \eqref{eq:asynchro_Q11}, we can rewrite the update rule \eqref{eq:asynchro_Q1} as follows: for every $t\geq 0$ and~$(\check \mu,\check \pi)\in \check S\times \check \Pi$,
\begin{align}\label{eq:gap_update}
	\begin{aligned}
		 & \check{\Delta}_{t+1}(\check \mu, \check \pi)=(1-\alpha_{t,(\check \mu, \check \pi)})\check{\Delta}_{t}(\check \mu, \check \pi)  +\alpha_{t,(\check \mu, \check \pi)}\check F_t(\check \mu,\check \pi), \\
		 & \quad \mbox{where $\check F_t( \check \mu,\check \pi):=                                                                                                                                                %
				\check {\cal H}\check Q_t( \check \mu,\check \pi)-\check {\cal H}\check Q^*( \check \mu,\check \pi) + \beta Z_{t+1,(\check \mu,\check \pi)}$.}
	\end{aligned}
\end{align}

\begin{proof}[Proof of Theorem \ref{thm:asynchro}]
	For both the synchronous and asynchronous versions of the $Q$-learning algorithm, we have by Lemma~\ref{lem:proj_errrors} that for any $t\geq 0$ and $(\mu,\pi)\in \overline{S}\times \Pi$,
	\begin{align}
		 & |\check Q_t(\operatorname{pj}_{\check S}(\mu),\operatorname{pj}_{\check \Pi}(\pi))- Q^*(\mu,\pi)|\nonumber                                                                                                                                                                                                \\
		 & \quad \leq|\check Q_t(\operatorname{pj}_{\check S}(\mu),\operatorname{pj}_{\check \Pi}(\pi))- \check Q^*(\operatorname{pj}_{\check S}(\mu),\operatorname{pj}_{\check \Pi}(\pi))|+| \check{Q}^*(\operatorname{pj}_{\check S}(\mu),\operatorname{pj}_{\check \Pi}(\pi))- Q^*(\mu,\pi)| \label{eq:prj_error} \\
		 & \quad \leq  \|\check \Delta_t\|_{ \check S\times \check \Pi}+C_1\varepsilon_{\check S}+{C_2}\varepsilon_{\check A}.\nonumber
	\end{align}

	Thus, it suffices to show that %
	$(\check{\Delta}_{t}(\check \mu, \check \pi),\alpha_{t,(\check \mu,\check \pi)}, \check F_t(\check \mu,\check \pi))_{t\geq 0,(\check \mu,\check \pi)\in \check S\times \check \Pi}$ in \eqref{eq:gap} and \eqref{eq:gap_update} satisfy all the conditions of Lemma \ref{lem:jaakkola}, so that
	for any $(\check \mu,\check \pi)\in \check S\times \check \Pi$, $\check{\Delta}_{t}(\check \mu, \check \pi)\to 0$ $\widehat{\mathbb{P}}^0$-a.s. as $t\to \infty$.

	We first consider the synchronous version of the $Q$-learning algorithm.

	\noindent{\it (Measurability) } %
	The synchronous learning rates $(\alpha_{t,(\check \mu,\check \pi)})_{t\geq 0,(\check \mu,\check \pi)\in \check S\times \check \Pi}$ in Framework \ref{frame:Qlearning}\;(i) are deterministic.
	Moreover, from Remark \ref{rem:msr_indep}, we have for every $(\check \mu,\check \pi)\in \check S\times \check \Pi$ that $\check \Delta_t( \check \mu,\check \pi)$ and $\check F_{t-1}( \check \mu,\check \pi)$ are $\check{\cal F}_{t}^0$-measurable for any $t\geq 1$, while $\check \Delta_0( \check \mu,\check \pi)$ is $\check{\cal F}_{0}^0$-measurable.

	\noindent {\it (Learning rates) } %
	Since $\alpha_{t,(\check \mu,\check \pi)}=(t+1)^{-w}$ for $t\geq 0$ with $w\in(\frac{1}{2},1)$, for any $(\check\mu,\check\pi)\in \check S\times \check \Pi$
	\[
		\sum_{t=0}^\infty \alpha_{t,(\check\mu,\check\pi)}
		=\sum_{t=0}^\infty (t+1)^{-w}=\infty\quad \mbox{and}\quad \sum_{t=0}^\infty \alpha_{t,(\check\mu,\check\pi)}^2
		=\sum_{t=0}^\infty (t+1)^{-2w}<\infty.
	\]

	\noindent {\it (Mean estimate of $\check F_t$) }  %
	We note that for any $t\geq 0$ and $(\check \mu,\check \pi)\in \check S\times \check \Pi$ %
	\begin{align}\label{eq:contraction_H}
		\begin{aligned}
			 & |\check {\cal H}\check Q_t( \check \mu,\check \pi)-\check {\cal H}\check Q^*( \check \mu,\check \pi)|                                                                                                                                                                                                                           \\
			 & \quad \leq \beta \sup_{p\in {\cal B}^0_{m,q}(\widehat{p}_{\varepsilon^0})} \int_{E^0}\bigg| \max_{\check \pi^1\in\check \Pi}\check Q_t(\check{\operatorname{F}}(\check\mu,\check\pi,e^0),\check \pi^1)- \max_{\check \pi^2\in\check\Pi}\check Q^*(\check{\operatorname{F}}(\check\mu,\check\pi,e^0),\check \pi^2)\bigg| p(de^0) \\
			 & \quad \leq \beta \sup_{p\in {\cal B}^0_{m,q}(\widehat{p}_{\varepsilon^0})} \int_{E^0}\max_{\check \pi^1\in\check\Pi}\Big| \check Q_t(\check{\operatorname{F}}(\check\mu,\check\pi,e^0),\check \pi^1)- \check Q^*(\check{\operatorname{F}}(\check\mu,\check\pi,e^0),\check \pi^1)\Big| p(de^0)                                   \\
			 & \quad \leq \beta \|\check \Delta_t\|_{\check S\times \check\Pi}.
		\end{aligned}
	\end{align}
	Combining this with Lemma \ref{lem:mtg_z} yields that for any $t\geq 0$ and $(\check \mu,\check \pi)\in \check S\times \check \Pi$, $\widehat{\mathbb{P}}^0$-a.s.,
	\begin{align}\label{eq:mean_est_F}
		\begin{aligned}
			\big|\mathbb{E}^{\widehat {\mathbb{P}}^0}[\check F_t (\check \mu,\check \pi)|\check{\cal F}_t^0]\big| & =\big|\mathbb{E}^{\widehat {\mathbb{P}}^0}[\check {\cal H}\check Q_t( \check \mu,\check \pi)-\check {\cal H}\check Q^*( \check \mu,\check \pi)|\check{\cal F}_t^0]\big|                                                                   \\
			                                                                                                      & \leq \mathbb{E}^{\widehat {\mathbb{P}}^0}\big[|\check {\cal H}\check Q_t( \check \mu,\check \pi)-\check {\cal H}\check Q^*( \check \mu,\check \pi)|\big|\check{\cal F}_t^0\big]\leq \beta \|\check \Delta_t\|_{\check S\times \check\Pi}.
		\end{aligned}
	\end{align}

	\noindent {\it (Variance estimate of $\check F_t$) } %
	Using \eqref{eq:contraction_H} and Lemma \ref{lem:mtg_z}, we have for any $t\geq 0$ and $(\check \mu,\check \pi)\in \check S\times \check \Pi$
	\begin{align*}
		\begin{aligned}
			\mathbb{E}^{\widehat {\mathbb{P}}^0}\big[|\check F_t (\check \mu,\check \pi)|^2\big|\check{\cal F}_t^0\big] & \leq 2 \big(\mathbb{E}^{\widehat {\mathbb{P}}^0}\big[|\check {\cal H}\check Q_t( \check \mu,\check \pi)-\check {\cal H}\check Q^*( \check \mu,\check \pi)|^2\big|\check{\cal F}_t^0\big] +\beta^2  \mathbb{E}^{\widehat {\mathbb{P}}^0}\big[|Z_{t+1,(\check \mu,\check \pi)}|^2\big|\check{\cal F}_t^0\big]  \big) \\
			                                                                                                            & \leq 2\beta^2 \big( \|\check \Delta_t\|_{\check S\times \check \Pi}^2+ 2^4 \check Q_{\operatorname{max}}  \big).
		\end{aligned}
	\end{align*}

	Combining this with \eqref{eq:mean_est_F} %
	yields that for any $t\geq 0$,  $\widehat{\mathbb P}^0$-a.s.,
	\begin{align}\label{eq:var_est_F}
		\begin{aligned}
			\|{\rm \mathbb{V}ar}^{\widehat{\mathbb P}^0}[\check F_t (\cdot,\cdot)|\check{\cal F}_t^0] \|_{\check S\times \check \Pi} & \leq \|\mathbb{E}^{\widehat {\mathbb{P}}^0}[\check F_t (\cdot,\cdot)|\check{\cal F}_t^0]\|_{\check S\times \check \Pi}^2+ \|\mathbb{E}^{\widehat {\mathbb{P}}^0}[|\check F_t (\cdot,\cdot)|^2|\check{\cal F}_t^0]\|_{\check S\times \check \Pi} \\
			                                                                                                                         & \leq \beta^2(\|\check \Delta_t\|_{\check S\times \check \Pi}+2^5\check Q_{\operatorname{max}}^2).
		\end{aligned}
	\end{align}

	Therefore, all the conditions in Lemma \ref{lem:jaakkola} are satisfied.

	For the case of the asynchronous version, the learning rate in~\eqref{eq:learningrate} is also deterministic, since the trajectory $(\check \mu_t,\check \pi_t)_{t\geq 0}\subseteq \check S\times \check\Pi$ used in its definition is pre-sampled. Moreover, under Assumption \ref{as:Qlearning}\;(ii), for any $(\check \mu,\check \pi)\in \check S\times \check \Pi$, the set $T_{(\check\mu,\check\pi)}$ in \eqref{eq:visiting_time} satisfies $|T_{(\check\mu,\check\pi)}|=\infty$. Thus, we may enumerate the times in $T_{(\check\mu,\check\pi)}$ as a strictly increasing sequence $(t_k)_{k\ge 0}$ such that
	\(
	\sum_{t=0}^\infty \alpha_{t,(\check\mu,\check\pi)}
	=\sum_{k=0}^\infty \alpha_{t_k,(\check\mu,\check\pi)}
	=\sum_{k=0}^\infty (k+1)^{-w}=\infty
	\)
	and $\sum_{t=0}^\infty \alpha_{t,(\check\mu,\check\pi)}^2
		=\sum_{k=0}^\infty (k+1)^{-2w}<\infty$.

	Moreover, the arguments establishing \eqref{eq:contraction_H}, \eqref{eq:mean_est_F}, and \eqref{eq:var_est_F} for the synchronous case apply directly to the asynchronous case. Consequently, all the conditions in Lemma~\ref{lem:jaakkola} are satisfied, and the result now follows directly from applying the convergence result in Lemma~\ref{lem:jaakkola}.

	This completes the proof.
\end{proof}

\subsection{Proof of Theorem \ref{thm:asynchro_rate}}\label{sec:proof:thm:asynchro_rate1}
We start by outlining the main ideas of the proof of Theorem~\ref{thm:asynchro_rate} for the asynchronous version of the $Q$-learning algorithm in Framework~\ref{frame:Qlearning}\,(i),(iii).

Conceptually, the analysis follows the approach in the literature on a finite-time iteration bound analysis for $Q$-learning (see, e.g., \cite{even2003learning}). In context of our $Q$-learning algorithm, the process $(\check \Delta_t)_{t\geq 0}$  defined in~\eqref{eq:gap} is controlled epoch-wise. Each epoch accounts for the covering time $T_{\operatorname{cov}}$ in Assumption \ref{as:Qlearning}\;(ii). The epochs are constructed so that $(\check \Delta_t)_{t\geq 0}$ decays inductively over epochs in the following way.

The process $(\check \Delta_t)_{t\geq 0}$ is decomposed into deterministic and stochastic components. The deterministic component vanishes inductively across epochs, whereas the stochastic component is represented as a sum of $\widehat{\mathbb{P}}^0$-martingale difference terms characterized by $(Z_{t+1,(\cdot,\cdot)})_{t\geq 0}$ defined in~\eqref{eq:asynchro_Q11}. Since $(Z_{t+1,(\cdot,\cdot)})_{t\geq 0}$ is bounded and has conditional zero mean; see Lemma~\ref{lem:mtg_z}, Azuma's inequality \cite{azuma1967weighted} (also introduced in Lemma \ref{lem:azuma}) yields concentration bounds for the stochastic term. Consequently, by choosing the number of iterations $T^*$ of the order of a sufficiently large epoch, we obtain that for every $T\geq T^*$,
\[
	\widehat{\mathbb{P}}^0\big(\|\check \Delta_{T}\|_{\check S\times \check \Pi}=\|\check Q_{T}-\check Q^*\|_{\check S\times \check \Pi} \leq \widehat \varepsilon\big)\geq 1-\widehat \delta,
\]
where $\widehat \varepsilon>0$ and $\widehat \delta\in (0,1)$ denotes the prescribed error and confidence levels, respectively.

The final ingredient in the proof is the discretization error bound (i.e., $C_1\varepsilon_{\check S}+{C_2}\varepsilon_{\check A}$) for the optimal $Q$-function established in Lemma~\ref{lem:proj_errrors}.

The notions and lemmas introduced below are tailored to the asynchronous version, but extend to the synchronous version under Framework~\ref{frame:Qlearning}(i),(ii) with only minor modifications. The synchronous case is briefly elaborated at the end of the proof.
\begin{definition}\label{dfn:epoch}%
	Let Assumption \ref{as:Qlearning}\;(ii) hold.
	Recall that $w\in (\frac{1}{2},1)$ is the exponent of the asynchronous learning rates and $(\check \mu_t, \check \pi_t)_{t\geq 0}\subseteq  \check S\times \check \Pi$ is the sampled trajectory in Framework~\ref{frame:Qlearning}\;(ii). Fix $\check \kappa \in(0,1)$ and let $\check c\geq 1$.
	\begin{itemize}
		\item [(i)] Set $\tau_0:=0$. For any $n\geq 1$ and any $\check t  \ge 1$, define
		      \[
			      \tau_{n+1;{\check t}}:=\tau_{n;{\check t}}+\bigg\lceil T_{\operatorname{cov}}\frac{\check c}{\check \kappa} \tau_{n;{\check t}}^w\bigg\rceil,\quad \mbox{with $\tau_{1;{\check t}}:= \check t$},
		      \]
		      where $T_{\operatorname{cov}}>1$ is the covering time given in Assumption~\ref{as:Qlearning}\;(ii), and $\check t$ represents the initial epoch's length. %
		\item [(ii)] For any $0\leq t_1<t_2$ and  \((\check \mu,\check \pi)\in \check S\times \check \Pi,
		      \) define
		      \begin{align*}%
			      T^{t_1,t_2}_{(\check \mu, \check \pi)}:= \{t\in [t_1,t_2): (\check \mu,\check \pi)=(\check \mu_t, \check \pi_t)\}\subseteq T_{(\check \mu, \check \pi)},%
		      \end{align*}
		      with the set $T_{(\check \mu, \check \pi)}$ defined in \eqref{eq:visiting_time}.
		\item [(iii)] For any $n\geq 1$, define \(D_{n+1}:=(1-\check \beta)^n D_1\), where $D_1:=2\check Q_{\operatorname{max}}$ and $\check \beta:=\frac{1-\beta}{2}$.
	\end{itemize}
\end{definition}

We first collect some elementary properties of the notions introduced in Definition~\ref{dfn:epoch}\;(i),(ii) and the learning rates in Framework \ref{frame:Qlearning}\;(ii).
\begin{lemma}\label{lem:time_check0}
	Let Assumption \ref{as:Qlearning}\;(ii) hold.  Then for any $\check t \geq 1$, the sequence $(\tau_{n;\check t})_{n\geq 0}$ given in Definition \ref{dfn:epoch}\;(i)~satisfy, for every $n\geq 0$,
	\begin{itemize}
		\item [(i)] $\frac{\tau_{n;\check t}}{T_{\operatorname{cov}}}\leq \lceil \frac{\tau_{n;\check t}}{T_{\operatorname{cov}}} \rceil+ \lfloor \frac{\check c}{\check \kappa}\tau_{n;\check t}^w  \rfloor \leq \frac{\tau_{n+1;\check t}}{T_{\operatorname{cov}}}+1.$
		\item [(ii)] $\tau_{n+1;\check t}^{1-w}\leq \check t^{1-w}+n(1-w)(T_{\operatorname{cov}}\frac{\check c}{\check \kappa}+\check t^{-w})$.
	\end{itemize}
	Moreover, %
	for every $n\geq 1$, $t\in[\tau_{n+1;\check t},\tau_{n+2;\check t})$, and $(\check \mu,\check \pi) \in \check S\times \check \Pi$,
	\begin{itemize}
		\item [(iii)] $|T^{\tau_{n;\check t},t}_{(\check \mu,\check \pi)}|\geq \lfloor \frac{\check c}{\check \kappa} \tau_{n;\check t}^w \rfloor$.
		\item [(iv)]   Let $i\in T^{ \tau_{n;\check t},t}_{(\check \mu, \check \pi)}$ be the smallest element in the set. Then $\alpha_{i,(\check \mu ,\check \pi )}\leq ( \lfloor \frac{ \tau_{n;\check t}}{T_{\operatorname{cov}}} \rfloor+1)^{-w}$.
	\end{itemize}
\end{lemma}
\begin{proof}
	For (i), the inequality $\frac{\tau_{n;\check t}}{T_{\operatorname{cov}}}\leq \lceil \frac{\tau_{n;\check t}}{T_{\operatorname{cov}}} \rceil+ \lfloor \frac{\check c}{\check \kappa}\tau_{n;\check t}^w  \rfloor$ is obvious, while the other inequality follows from the definition of $\tau_{n+1;\check t}$ in Definition \ref{dfn:epoch}\;(i). Indeed, for any $n\geq 0$
	\[
		\bigg\lceil \frac{\tau_{n;\check t}}{T_{\operatorname{cov}}} \bigg\rceil+ \bigg\lfloor \frac{\check c}{\check \kappa}\tau_{n;\check t}^w  \bigg\rfloor \leq \frac{\tau_{n;\check t}+T_{\operatorname{cov}}\lfloor \frac{\check c}{\check \kappa}\tau_{n;\check t}^w  \rfloor}{T_{\operatorname{cov}}}+1\leq \frac{\tau_{n;\check t}+\lceil T_{\operatorname{cov}}\frac{\check c}{\check \kappa} \tau_{n;\check t}^w\rceil}{T_{\operatorname{cov}}}+1= \frac{\tau_{n+1;\check t}}{T_{\operatorname{cov}}}+1.
	\]

	For (ii), since by the concavity of the map $\mathbb{R}_+\ni x\to x^{1-w}\in \mathbb{R}_+$, for any $n\geq 0$,
	\begin{align*}
		\begin{aligned}
			\tau_{n+1;\check t}^{1-w}-\tau_{n;\check t}^{1-w} & \leq \bigg(\tau_{n;\check t}+T_{\operatorname{cov}}\frac{\check c}{\check \kappa}\tau_{n;\check t}^w+1\bigg)^{1-w}-\tau_{n;\check t}^{1-w}                                             \\
			                                                  & \leq (1-w)\tau_n^{-w}\bigg(T_{\operatorname{cov}}\frac{\check c}{\check \kappa}\tau_n^w+1\bigg)\leq (1-w)\bigg(T_{\operatorname{cov}}\frac{\check c}{\check \kappa}+\tau_1^{-w}\bigg).
		\end{aligned}
	\end{align*}
	Hence, the desired estimate follows by iterating this inequality over $n\geq 0$.

	For (iii), by Assumption \ref{as:Qlearning}\;(ii), for every $n\geq 1$, $t\in[\tau_{n+1},\tau_{n+2})$, and $(\check \mu,\check \pi) \in \check S\times \check \Pi$,
	\[
		|T^{\tau_n,t}_{(\check \mu,\check \pi)}|\geq |T^{\tau_n,\tau_{n+1}}_{(\check \mu,\check \pi)}| \geq \frac{\tau_{n+1}-\tau_n}{T_{\operatorname{cov}}} \geq \bigg\lfloor \frac{\check c}{\check \kappa} \tau_n^w \bigg\rfloor.
	\]

	Last, for (iv), let $n\geq 1$, $t\in[\tau_{n+1},\tau_{n+2})$, and $(\check \mu,\check \pi) \in \check S\times \check \Pi$. Since $i\in T^{\tau_n,t}_{(\check \mu, \check \pi)}$ satisfies $i\geq \tau_n$, by Assumption \ref{as:Qlearning}\;(ii), it holds that $N_{i,(\check \mu,\check \pi)}\geq \lfloor \frac{\tau_n}{T_{\operatorname{cov}}}\rfloor$. Hence, by definition of the asynchronous learning rate in \eqref{eq:learningrate} the desired inequality holds.
\end{proof}

The following lemma establishes a relation between the epoch sequences corresponding to the initial epoch length and its increase.
\begin{lemma}\label{lem:perturb_init}
	Let Assumption \ref{as:Qlearning}\;(ii) hold. Let $\check n\geq 1$ be given. If the initial epoch length $\check t $ satisfies
	\[
		\check t\ge
		\max\bigg\{
		\bigg(\frac{2T_{\rm cov}\frac{\check c}{\check\kappa}w(\check  n-1)}{\log 2}\bigg)^{\frac1{1-w}},
		\bigg(\frac{2\check n+1}{2T_{\rm cov}\frac{\check c}{\check\kappa}}\bigg)^{\frac1w},1
		\bigg\},
	\]
	then it holds that $\tau_{\check n+2;\check t}-1\geq \tau_{\check n;\check t +1}$.
\end{lemma}
\begin{proof}
	For notational simplicity, set \(\check C:=T_{\rm cov}\check c/\check\kappa\), and write, for every $n\ge 1$,
	\(
	d_n:=\tau_{n;\check t+1}-\tau_{n;\check t}
	\)
	Then, it holds that for every $n\geq 1$
	\begin{align}\label{eq:gap_n}
		\begin{aligned}
			d_{n+1}= \tau_{n;\check t+1}+\lceil\check C \tau_{n;\check t+1}^w \rceil - ( \tau_{n;\check t}+\lceil\check C \tau_{n;\check t}^w \rceil)
			 & \le
			d_n
			+
			\check C(
			\tau_{n;\check t+1}^{w}
			-
			\tau_{n;\check t}^{w}
			) +1    \\
			 & \leq
			d_n(1+\check C w \tau_{n;\check t}^{w-1})+1,
		\end{aligned}
	\end{align}
	where the last inequality follows from the concavity of the map $\mathbb{R}_+\ni x\mapsto x^w\in \mathbb{R}_+$.

	Since $d_1=1$, iterating the estimate in \eqref{eq:gap_n} gives, for every \(1\le n\le \check n\),
	\begin{align}\label{eq:gap_n1}
		d_n
		\le
		n\left(1+\check Cw\check t^{w-1}\right)^{n-1}
		\le
		n\exp\left(\check Cw(n-1)\check t^{w-1}\right)
		\le
		2n,
	\end{align}
	where the last inequality follows from the condition on \(\check t\), namely,
	\(
	\check t\ge (\frac{2 \check Cw(\check n-1)}{\log 2})^{\frac{1}{1-w}}.
	\)

	On the other hand,
	\begin{align}\label{eq:gap_n2}
		\begin{aligned}
			\tau_{\check n+2;\check t}-1-\tau_{\check n;\check t}
			 & =
			\lceil \check C\tau_{\check n;\check t}^w\rceil
			+
			\lceil \check C\tau_{\check n+1;\check t}^w\rceil
			-1     \\
			 & \ge
			2\check C\check t^w-1\geq 2\check n,
		\end{aligned}
	\end{align}
	where the last inequality follows from the condition on $\check t$, namely, $\check t\geq (\frac{2\check n+1}{2\check C})^{\frac{1}{w}}$.

	Combining \eqref{eq:gap_n1} and \eqref{eq:gap_n2} leads to
	\[
		\tau_{\check n+2;\check t}-1
		\ge
		\tau_{\check n;\check t}+2\check n \geq \tau_{\check n;\check t}+d_{\check n}
		=
		\tau_{\tilde n;\check t+1}.
	\]
	This proves the claim.
\end{proof}

In what follows, we often make use of the following elementary inequalities.
\begin{lemma}%
	\label{lem:elem_est}
	The following hold:
	\begin{itemize}
		\item [(i)] For any $\tilde w\in(0,1)$ and $t_1,t_2\in \mathbb{N}$ with $t_2>t_1$,
		      \(
		      \prod_{i=t_1}^{t_2}(1-{(i+1)^{-\tilde w}})\leq \exp(-\frac{t_2-t_1}{(t_2+1)^{\tilde w}}).
		      \)
		\item [(ii)] For any $a>0$ satisfying $2a\log a>1$, and any $A>2a\log a$,
		      \(
		      A \exp({-\frac{2A}{a}})\leq \exp({-\frac{A}{a}}).
		      \)
	\end{itemize}
\end{lemma}
\begin{proof} Since $1-x\leq e^x$ for all $x\in[0,1]$, we obtain (i), namely,
	\[
		\prod_{i=t_1}^{t_2}(1-{(i+1)^{-\tilde w}})\leq \exp\bigg(-\sum_{i=t_1}^{t_2}(i+1)^{-\tilde w}\bigg)\leq \exp\bigg(-\frac{t_2-t_1}{(t_2+1)^{\tilde w}}\bigg),
	\]
	where the last inequality holds because $(i+1)^{-\tilde w}\geq (t_2+1)^{\tilde w}$ for all $i=t_1,\dots,t_2$.

	For (ii),
	let $a>0$ satisfy $2a\log a>1$. %
	For each $A\in(2a\log a,+\infty)$, define \(\phi(A):=\frac{A}{a} - \log A\), and show that $\phi\geq 0$ on the interval.

	Since
	\(
	\phi'(A)
	=
	\frac{1}{a} - \frac{1}{A}
	=
	\frac{A-a}{aA},
	\)
	$\phi$ attains the minimum at \(
	A^* = \max\{a, 2a\log a\}.
	\)
	Thus it suffices to verify that $\phi(A^*)\geq 0$.

	Suppose that $A^* = a$. Then, since
	\(
	\phi(a)= 1 - \log a
	\)
	and $2a\log a \le a$, we have
	\(
	\phi(A^*) \ge \tfrac12 > 0.
	\)
	Otherwise (i.e., $A^*=2a \log a$), since
	\(
	\phi(A^*)
	=
	\log(\frac{a}{2\log a}),
	\)
	it suffices to show $a \ge 2\log a$.

	Define $\psi(x):=x-2\log x$ for $x>0$.  Then \(
	\psi'(x) = 1 - \frac{2}{x} = \frac{x-2}{x},
	\)
	so
	$\psi$ attains the minimum  at $x=2$.
	Since
	\(
	\psi(2)=2-2\log 2>0,
	\)
	we conclude $\psi(x)>0$ for all $x>0$, and thus
	$a>2\log a$.

	Therefore $\phi(A) \ge 0$ for all $A>2a\log a$,
	which implies
	\(
	A \exp({-\frac{2A}{a}})\leq \exp({-\frac{A}{a}}).
	\)
\end{proof}

\begin{lemma}\label{lem:separate_est}
	Suppose that Assumptions \ref{as:MFC}\;(i),(ii) and \ref{as:Qlearning}\;(i) are satisfied.
	Let $(Z_{t+1,(\cdot,\cdot)})_{t\geq 0}$ and $(\check \Delta_t)_{t\geq 0}$ be~given in \eqref{eq:asynchro_Q11} and \eqref{eq:gap}, respectively. Let $\check t\geq 1$ be given. Let $(\tau_{n;\check t})_{n\geq 0}$ and~$(D_n)_{n\geq 1}$ be given in Definition~\ref{dfn:epoch}. %
	Moreover, for any $n\geq 1$, let $(W_{t;n},Y_{t;n})_{t\geq \tau_{n;\check t}}$ be sequences of mappings defined for every $t\geq \tau_{n;\check t}$ and $(\check \mu,\check \pi)\in \check S\times \check \Pi$ by
	\begin{align}\label{eq:YW_seq}
		\begin{aligned}
			 & W_{t+1;n}(\check \mu,\check \pi):= (1-\alpha_{t,(\check \mu,\check \pi)})W_{t;n}(\check \mu,\check \pi) +\alpha_{t,(\check \mu,\check \pi)} \beta Z_{t+1,(\check \mu,\check \pi)}, \\
			 & Y_{t+1;n}(\check \mu,\check \pi):= (1-\alpha_{t,(\check \mu,\check \pi)}) Y_{t;n}(\check \mu,\check \pi) + \alpha_{t,(\check \mu,\check \pi)}\beta D_n,                            %
		\end{aligned}
	\end{align}
	with $W_{ \tau_{n;\check t} ;n}(\check \mu,\check \pi):=0$ and $Y_{ \tau_{n;\check t} ;n}(\check \mu,\check \pi):=D_n$. If $\|\check\Delta_t\|_{\check S\times \check \Pi} \leq D_n$ holds for all $t\in[\tau_{n;\check t},\tau_{n+2;\check t})$, then we have  for every $t\geq \tau_{n;\check t}$ and~$(\check \mu,\check\pi)\in \check S\times \check\Pi$ that
	\begin{align}\label{eq:separate_est}
		-Y_{t;n}(\check \mu,\check \pi)+W_{t;n}(\check \mu,\check \pi)\leq \check \Delta_t(\check \mu,\check \pi)\leq Y_{t;n}(\check \mu,\check \pi)+W_{t;n}(\check \mu,\check \pi).
	\end{align}
\end{lemma}
\begin{proof}
	The proof uses an induction over $t\geq \tau_{n;\check t}$. The estimate \eqref{eq:separate_est} holds at $t=\tau_{n;\check t}$ because $W_{ \tau_{n;\check t} ;n}(\check \mu,\check \pi)=0$, $Y_{ \tau_{n;\check t} ;n}(\check \mu,\check \pi)=D_n$ and $\|\check\Delta_{\tau_{n;\check t}}\|_{\check S\times \check \Pi} \leq D_n$.

	Assume that the estimate holds for some $t\geq \tau_{n;\check t}$.  Let $(\check \mu,\check\pi)\in \check S\times \check\Pi$ be given. %

	Since $\check {\cal H}\check Q^*=\check {\cal Q}^*$ (see Remark \ref{rem:fixed_checkQ}), by the assumption that $\|\check\Delta_t\|_{\check S\times \check \Pi} \leq D_n$, the upper bound in~\eqref{eq:separate_est}, and the estimate in \eqref{eq:contraction_H}, we have %
	\begin{align*}
		\check \Delta_{t+1}(\check \mu,\check \pi)%
		 & =(1-\alpha_{t,(\check\mu,\check\pi)})\check \Delta_t(\check \mu,\check \pi)+\alpha_{t,(\check\mu,\check\pi)}\Big(\check {\cal H}\check Q_t( \check \mu,\check \pi)-\check {\cal H}\check Q^*( \check \mu,\check \pi) + \beta Z_{t+1,(\check \mu,\check \pi)}\Big) \\
		 & \leq                                                                                                                                                                                                                                                              %
		(1-\alpha_{t,(\check\mu,\check\pi)})Y_{t;n}(\check \mu,\check \pi) + W_{t+1;n}(\check \mu,\check \pi) + \alpha_{t,(\check\mu,\check\pi)}(\check {\cal H}\check Q_t( \check \mu,\check \pi)-\check {\cal H}\check Q^*( \check \mu,\check \pi) )                       \\
		 & \leq (1-\alpha_{t,(\check\mu,\check\pi)})Y_{t;n}(\check \mu,\check \pi)  + W_{t+1;n}(\check \mu,\check \pi) + \alpha_{t,(\check\mu,\check\pi)} \beta \|\check\Delta_t\|_{\check S\times \check \Pi}                                                               \\
		 & \leq                                                                                                                                                                                                                                                              %
		Y_{t+1;n}(\check \mu,\check \pi)+W_{t+1;n}(\check \mu,\check \pi).
	\end{align*}

	Using similar arguments and the lower bound in \eqref{eq:separate_est}, we also have %
	\begin{align*}
		\begin{aligned}
			\check \Delta_{t+1}(\check \mu,\check \pi) & \geq                                                                                                                                                                                              %
			W_{t+1;n}(\check \mu,\check \pi)- (1-\alpha_{t,(\check\mu,\check\pi)}) Y_{t;n}(\check \mu,\check \pi) + \alpha_{t,(\check\mu,\check\pi)}(\check {\cal H}\check Q_t( \check \mu,\check \pi)-\check {\cal H}\check Q^*( \check \mu,\check \pi) ) \\
			                                           & \geq W_{t+1;n}(\check \mu,\check \pi)- (1-\alpha_{t,(\check\mu,\check\pi)}) Y_{t;n}(\check \mu,\check \pi)  -\alpha_{t,(\check\mu,\check\pi)}\beta\|\check\Delta_t\|_{\check S\times \check \Pi}  \\
			                                           & \geq                                                                                                                                                                                              %
			W_{t+1;n}(\check \mu,\check \pi)-Y_{t+1;n}(\check \mu,\check \pi).
		\end{aligned}
	\end{align*}
	Therefore, by the induction hypothesis, the desired bound holds at time $t+1$.
\end{proof}

\begin{lemma}\label{lem:Y_bound}
	Suppose that Assumptions \ref{as:MFC}\;(i),(ii) and \ref{as:Qlearning} are satisfied. Let $\check\delta \in (0,e-2)$ be~given. Moreover, assume~that the initial epoch's length $\check t\geq 1$ and $\check c\geq 1$ in Definition~\ref{dfn:epoch} satisfy
	\[
		\check t\ge \max\{\check  t_{1}^*,\lceil{T_{\operatorname{cov}}}({T_{\operatorname{cov}}-1})^{-1}\rceil\}\quad\mbox{and}\quad \check c\geq \max\{\check c_{1}^*,\,1\},
	\]
	where
	\begin{align}\label{eq:tau1_condi1}
		\check t_{1}^*:=T_{\operatorname{cov}}\big\lceil {2}^{\frac{1}{w}}(1-\log(2+\check \delta))^{-\frac{1}{w}}\big\rceil,\quad \check c_{1}^*:=\frac{\log(2+\check\delta)+2({T_{\operatorname{cov}}}/{\check t_1^*})^{w}}{1-(\log(2+\check\delta)+2({T_{\operatorname{cov}}}/{\check t_1^*})^{w})}.
	\end{align}
	Furthermore, for any $n\geq 1$, let $(Y_{t,n})_{t\geq \tau_{n;\check t}}$ be given  in \eqref{eq:YW_seq}.
	If $\|\check\Delta_t\|_{\check S\times \check \Pi} \leq D_n$ holds for all $t\in\{\tau_{n;\check t},\dots,\tau_{n+2;\check t}-1\}$, then for every $t\geq \tau_{n+1;\check t}$ %
	\begin{align}\label{eq:Y_bound}
		\|Y_{t;n}\|_{\check S\times \check \Pi}\leq (\beta+{2}\check \beta(2+\check \delta)^{-1})D_n.
	\end{align}
\end{lemma}
\begin{proof}
	Note that $(Y_{t;n})_{t\geq  \tau_{n;\check t}}$ can be rewritten %
	for every $t>\tau_{n;\check t}$ and $(\check\mu,\check\pi)\in \check S\times \check \Pi$ by
	\begin{align}\label{eq:Y_repre_1}
		Y_{t;n}(\check \mu,\check \pi)= D_n(\beta +2\check \beta \rho_{t;n}(\check\mu,\check\pi)),\quad\mbox{where } \rho_{t;n}(\check\mu,\check\pi):=\prod_{i\in T^{\tau_{n;\check t},t}_{(\check \mu,\check \pi)}}(1-\alpha_{i,(\check\mu,\check\pi)}),
	\end{align}
	with setting $\rho_{\tau_{n;\check t};n}(\check\mu,\check\pi):=(1-\beta)D_n$ and recalling that $Y_{\tau_{n;\check t};n}=D_n$. %

	Then it holds for every $t\geq \tau_{n+1;\check t}$ and $(\check\mu,\check\pi)\in \check S\times \check \Pi$ that
	\begin{align}\label{eq:Y_repre_11}
		\begin{aligned}
			 & \rho_{t;n}(\check\mu,\check\pi)\leq \prod_{i= \lceil \frac{\tau_{n;\check t}}{T_{\operatorname{cov}}} \rceil}^{\lceil \frac{\tau_{n;\check t}}{T_{\operatorname{cov}}} \rceil+ \lfloor \frac{\check c}{\check \kappa}\tau_{n;\check t}^w  \rfloor-1}(1-{(i+1)^{-w}})\leq \exp(\operatorname{I}),          \\
			 & \quad \mbox{where } \operatorname{I}:=-\bigg(\bigg\lfloor \frac{\check c}{\check \kappa}\tau_{n;\check t}^w  \bigg\rfloor-1\bigg)\bigg(\bigg\lceil \frac{\tau_{n;\check t}}{T_{\operatorname{cov}}} \bigg\rceil+ \bigg\lfloor \frac{\check c}{\check \kappa}\tau_{n;\check t}^w  \bigg\rfloor\bigg)^{-w}.
		\end{aligned}
	\end{align}
	Here the first inequality follows from Lemma~\ref{lem:time_check0}\;(iii),\;(iv) and the fact that $t\geq \tau_{n+1;\check t}$, whereas the second inequality follows from Lemma \ref{lem:elem_est}\;(i).

	Therefore, it suffices to prove that $\operatorname{I}\leq -\log(2+\check \delta)$. By Lemma~\ref{lem:time_check0}\;(i), %
	we have
	\begin{align}\label{eq:Y_repre_2}
		\begin{aligned}
			\operatorname{I} & \leq -\bigg(\frac{\check c}{\check \kappa} \tau_{n;\check t}^w -2\bigg)\bigg(\bigg\lceil \frac{\tau_{n;\check t}}{T_{\operatorname{cov}}} \bigg\rceil+ \bigg\lfloor \frac{\check c}{\check \kappa}\tau_{n;\check t}^w  \bigg\rfloor\bigg)^{-w}       \\
			                 & \leq  {-\frac{\check c}{\check \kappa}}\bigg({\frac{\tau_{n+1;\check t}}{T_{\operatorname{cov}}\tau_{n;\check t}} +\frac{1}{\tau_{n;\check t}}  }\bigg)^{-w}+  2\bigg(\frac{T_{\operatorname{cov}}}{\tau_{n;\check t}}\bigg)^{w}:=\operatorname{II}.
		\end{aligned}
	\end{align}

	Moreover, it holds that
	\[
		{\frac{\tau_{n+1;\check t}}{T_{\operatorname{cov}}\tau_{n;\check t}} +\frac{1}{\tau_{n;\check t}}  }\leq  \frac{\tau_{n;\check t}+ T_{\operatorname{cov}}\frac{\check c}{\check \kappa} \tau_{n;\check t}^w+1}{T_{\operatorname{cov}}\tau_{n;\check t}}+\frac{1}{\tau_{n;\check t}}\leq \frac{1}{T_{\operatorname{cov}}}+\frac{\check c}{\check \kappa}+\frac{1}{\check t} \leq  \frac{\check c}{\check \kappa}+1,
	\]
	where the last inequality follows from the assumption that $\check t\geq {T_{\operatorname{cov}}}({T_{\operatorname{cov}}-1})^{-1}$.

	From this and the facts that $\tau_{n;\check t}\geq \check t$ and $\check k,w <1$, we have
	\begin{align}\label{eq:Y_repre_3}
		\operatorname{II}\leq {-\frac{\check c}{\check \kappa}}{\bigg( \frac{\check c}{\check \kappa}+1\bigg)^{-w}}+2\bigg(\frac{T_{\operatorname{cov}}}{\check t}\bigg)^{w}  %
		\leq \frac{-{\check c}}{{\check c}+1}+2\bigg(\frac{T_{\operatorname{cov}}}{\check t}\bigg)^{w}\leq -\log(2+\check \delta),
	\end{align}
	where the last inequality holds because $\check t \geq \check t_{1}^*$ and $\check c\geq \check c_{1}^*$ (see \eqref{eq:tau1_condi1}), together with the fact that $\log(2+\check \delta)<1$ (because of $\check \delta \in(0,e-2)$).

	Combining \eqref{eq:Y_repre_2} and \eqref{eq:Y_repre_3} yields that $\operatorname{I}\leq -\log(2+\check \delta)$, as claimed. %
\end{proof}

Next, we define, for every $\check t\geq 1$, $n\geq 1$, $t\in\{\tau_{n+1;\check t},\dots,\tau_{n+3;\check t}-1\}$, $l\in\{-1,0,\cdots,t-1-\tau_{n;\check t}\}$, and $(\check \mu,\check \pi)\in \check S\times \check \Pi$,
\begin{align}\label{eq:X_def}
	X_{t;n}^l(\check \mu,\check \pi):= \beta \sum_{i=\tau_{n;\check t}}^{\tau_{n;\check t}+l}\alpha_{i,(\check \mu,\check \pi)} \bigg(\prod_{j=i+1}^{t-1} (1-\alpha_{j,(\check \mu,\check \pi)})\bigg)Z_{i+1;(\check \mu,\check \pi)},\quad \mbox{$l\geq 0$},
\end{align}
with $X_{t;n}^{-1}(\check \mu,\check \pi):= 0$. Then %
the sequence $(W_{t;n})_{t= \tau_{n+1;\check t}}^{\tau_{n+3;\check t}-1}$ given in \eqref{eq:YW_seq} can be rewritten for every $t\in\{\tau_{n+1;\check t},\dots,\tau_{n+3;\check t}-1\}$ and $(\check \mu,\check \pi)\in \check S\times \check \Pi$ by
\begin{align}\label{eq:W_seq2}
	W_{t;n}(\check \mu,\check \pi)= X_{t;n}^{t-1-\tau_{n;\check t}}(\check \mu,\check \pi)= \sum_{l=0}^{t-1-\tau_{n;\check t}}\big(X_{t;n}^{l}(\check \mu,\check \pi)-X_{t;n}^{l-1}(\check \mu,\check \pi)\big).
\end{align}

\begin{lemma}\label{lem:X_properties}
	Suppose that Assumptions \ref{as:MFC}\;(i),(ii) and \ref{as:Qlearning} are satisfied. Recall the probability space $(\Omega^0,{\cal F}^0,\widehat{\mathbb{P}}^0)$ in Framework \ref{frame:Qlearning} and the filtration $\check {\mathbb{F}}^0$ in \eqref{eq:filtration_mupi}. Then for every $\check t\geq 1$, $n\geq 1$, $t\in\{\tau_{n+1;\check t},\dots,\tau_{n+3;\check t}-1\}$,  $(\check \mu,\check \pi)\in \check S\times \check \Pi$, the following hold.
	\begin{itemize}
		\item [(i)] For every $k\in\{\tau_{n;\check t},\dots,t\}$, $X_{t;n}^{k-1-\tau_{n;\check t}}(\check \mu,\check \pi)$ in \eqref{eq:X_def} is $\check{\mathcal{F}}_{k}^0$-measurable.
		\item[(ii)] $(X_{t;n}^{k-1-\tau_{n;\check t}}(\check \mu,\check \pi))_{k=\tau_{n;\check t}}^t$ is a $\widehat {\mathbb{P}}^0$-martingale w.r.t.\;$(\check{\mathcal{F}}_{k}^0)_{k=\tau_{n;\check t}}^t$.
		\item [(iii)] For every $k\in\{\tau_{n;\check t},\dots,t-1\}$, $|X_{t;n}^{k-\tau_{n;\check t}}(\check \mu,\check \pi)-X_{t;n}^{k-1-\tau_{n;\check t}}(\check \mu,\check \pi)| \leq 4 \beta \check Q_{\operatorname{max}} (\frac{T_{\operatorname{cov}}}{\tau_{n;\check t}})^w$.
	\end{itemize}
\end{lemma}
\begin{proof}
	Let  $\check t\geq 1$, $n\geq 1$, $t\in[\tau_{n+1;\check t},\tau_{n+3;\check t})$ and $(\check \mu,\check \pi)\in \check S\times \check \Pi$.
	Indeed, since $(\alpha_{t,(\check\mu,\check\pi)})_{t\ge 0}$ are deterministic for any $(\check \mu,\check \pi)\in \check S\times \check \Pi$, we have by  Remark \ref{rem:msr_indep}\;(ii) that
	for any~$k= \tau_{n;\check t}+1,\cdots,t,$
	\[
		X_{t;n}^{k-1-\tau_{n;\check t}}(\check \mu,\check \pi)= \beta \sum_{i=\tau_{n;\check t}}^{k-1}\alpha_{i,(\check \mu,\check \pi)} \bigg(\prod_{j=i+1}^{t-1} (1-\alpha_{j,(\check \mu,\check \pi)})\bigg)Z_{i+1;(\check \mu,\check \pi)}
	\]
	is $\check{\cal F}_{k}^0$-measurable, while $X_{t;n}^{-1}(\check \mu,\check \pi)=0$ is $\check{\cal F}_{\tau_{n;\check t}}^0$-measurable. Thus, Part (i) holds.

	Parts (ii) and (iii) follow from Lemma \ref{lem:mtg_z}. Indeed, %
	for any $k\in\{\tau_{n;\check t},\dots,t-1\}$, $\widehat{\mathbb{P}}^0$-a.s.,
	\begin{align*}
		 & \mathbb{E}^{\widehat{\mathbb{P}}^0}[X_{t;n}^{k-\tau_{n;\check t}}(\check \mu,\check \pi)-X_{t;n}^{k-1-\tau_{n;\check t}}(\check \mu,\check \pi)|\check{\cal F}_{k}^0]                                  \\
		 & \quad = \beta \alpha_{k,(\check \mu,\check \pi)}\prod_{j=k+1}^{t-1}(1-\alpha_{j,(\check \mu,\check \pi)}) \mathbb{E}^{\widehat{\mathbb{P}}^0}[Z_{k+1;(\check \mu,\check \pi)}|\check{\cal F}^0_{k}]=0.
	\end{align*}

	Moreover, for any $k\in\{\tau_{n;\check t},\dots,t-1\}$,
	\begin{align*}
		\begin{aligned}
			\big|X_{t;n}^{k-\tau_{n;\check t}}(\check \mu,\check \pi)-X_{t;n}^{k-1-\tau_{n;\check t}}(\check \mu,\check \pi)\big| & =  \beta \alpha_{k,(\check \mu, \check \pi) }\prod_{j=k+1}^{t-1}(1-\alpha_{j,(\check \mu, \check \pi)}) \big|Z_{k+1;(\check \mu,\check \pi)} \big|                              \\
			                                                                                                                      & \leq  4 \beta \check Q_{\operatorname{max}}  \alpha_{k,(\check \mu, \check \pi) }\leq  4 \beta \check Q_{\operatorname{max}}  ({T_{\operatorname{cov}}}/{\tau_{n;\check t}})^w,
		\end{aligned}
	\end{align*}
	where the last inequality holds because $N_{(\check \mu,\check \pi)}^k\geq \lfloor \frac{\tau_{n;\check t}}{T_{\operatorname{cov}}} \rfloor\geq \frac{\tau_{n;\check t}}{T_{\operatorname{cov}}}-1$ (by the covering time condition in Assumption \ref{as:Qlearning}\;(ii); see also \eqref{eq:learningrate} for definition of the asynchronous learning rates).
\end{proof}

The following lemma states Azuma's inequality \cite{azuma1967weighted} (see, e.g., Exercise 9.2.4 in \cite{Klenke14}), which will be used in the subsequent lemma.
\begin{lemma}\label{lem:azuma}
	If $(M_n)_{n\geq 0}$ is a martingale with $M_0=0$ on a filtered probability space $(\widetilde \Omega,\widetilde{{\cal F}},\widetilde{\mathbb{F}},\widetilde{\mathbb{P}})$ and there exists a sequence $(c_n)_{n\geq 1}\subseteq [0,\infty)$ such that $|M_n-M_{n-1}|\leq c_n$ $\widetilde{\mathbb{P}}$-a.s.~for all $n\geq 1$, then Azuma's inequality holds:
	\begin{align*}
		\widetilde{\mathbb{P}}[|M_n|\geq \lambda]\leq 2 \exp\bigg(-\frac{\lambda^2}{2\sum_{k=1}^nc_k^2} \bigg)\quad \mbox{for all $\lambda \geq 0$}.
	\end{align*}
\end{lemma}

\begin{lemma}\label{lem:azuma_W}
	Suppose that Assumptions \ref{as:MFC}\;(i),(ii) and \ref{as:Qlearning} are satisfied. Recall the probability space $(\Omega^0,{\cal F}^0,\widehat{\mathbb{P}}^0)$ in Framework \ref{frame:Qlearning} and the filtration $\check {\mathbb{F}}^0$ in \eqref{eq:filtration_mupi}. Then, for any $\check t\geq 1$, $n\geq 1$, $t\in\{\tau_{n+1;\check t},\dots,\tau_{n+3;\check t}-1\}$, and  $(\check \mu,\check \pi)\in \check S\times \check \Pi$, the random variable $W_{t;n}(\check \mu,\check \pi)$ in~\eqref{eq:YW_seq} satisfies
	\begin{align*}
		\widehat{\mathbb P}^0(|W_{t;n}(\check \mu,\check \pi)|\geq \lambda ) & \leq 2\exp\bigg(-\frac{\lambda^2\tau_n^w}{2^5C_{3}(\beta \check Q_{\operatorname{max}}T_{\operatorname{cov}}^{w})^2}\bigg),\quad \mbox{for all $\lambda \geq 0$},
	\end{align*}
	where $C_3:=C_{(\check c,\check \kappa,w,T_{\operatorname{cov}})}:=T_{\operatorname{cov}}\frac{\check c}{\check \kappa}(7+5(T_{\operatorname{cov}}\frac{\check c}{\check \kappa})^w+(T_{\operatorname{cov}}\frac{\check c}{\check \kappa})^{2w})+3$. %
\end{lemma}
\begin{proof}
	Let  $\check t\geq 1$, $n\geq 1$, $t\in\{\tau_{n+1;\check t},\dots,\tau_{n+3;\check t}-1\}$, and $(\check \mu,\check \pi)\in \check S\times \check \Pi$ be given. Note that \(W_{t;n}(\check \mu,\check \pi)= X_{t;n}^{t-1-\tau_{n;\check t}}(\check \mu,\check \pi)\); see \eqref{eq:W_seq2}. From this and Lemma \ref{lem:X_properties}, we apply Azuma's inequality (see Lemma \ref{lem:azuma}) to $(X_{t;n}^{k-1-\tau_{n;\check t}}(\check \mu,\check \pi))_{k=\tau_{n;\check t}}^t$ to have for all $\lambda \geq 0$
	\begin{align*}
		\begin{aligned}
			\widehat{\mathbb P}^0(|W_{t;n}(\check \mu,\check \pi)|\geq \lambda )
			 & \leq 2\exp\bigg(-\frac{\lambda^2}{2(t-\tau_{n;\check t})(4 \beta \check Q_{\operatorname{max}}  (\frac{T_{\operatorname{cov}}}{\tau_{n;\check t}})^w)^2}\bigg)                   \\
			 & \leq 2\exp\bigg(-\frac{\lambda^2}{2(\tau_{n+3;\check t}-\tau_{n;\check t})(4\beta \check Q_{\operatorname{max}}  (\frac{T_{\operatorname{cov}}}{\tau_{n;\check t}})^w)^2}\bigg).
		\end{aligned}
	\end{align*}

	Therefore, it suffices to show that \(
	\tau_{n+3;\check t}-\tau_{n;\check t}\leq C_3\tau_{n;\check t}^w.\)
	From the inequality that $(a+b+c)^w\leq a^w+b^w+c^w$ for all $a,b,c\geq 0$ (as $w<1$), we have
	\begin{align}\label{eq:tau_grow}
		\tau_{n+1;\check t}^w\leq   \bigg(\tau_{n;\check t}+1+T_{\operatorname{cov}}\frac{\check c}{\check \kappa} \tau_{n;\check t}\bigg)^w\leq \tau_{n;\check t}^w+1+ \bigg(T_{\operatorname{cov}} \frac{\check c}{\check \kappa} \tau_{n;\check t}\bigg)^w\leq \widehat C_3\tau_{n;\check t}^w,%
	\end{align}
	where $\widehat C_3:=2+(T_{\operatorname{cov}}\frac{\check c}{\check \kappa})^w>0$. Similarly, we have $\tau_{n+2;\check t}^w\leq \widehat C_3 \tau_{n+1;\check t}^w\leq  \widehat C_3^2 \tau_{n;\check t}^w$.

	Thus, we have %
	\[
		\tau_{n+3;\check t}-\tau_{n;\check t}%
		\leq T_{\operatorname{cov}}\frac{\check c}{\check \kappa}(\tau_{n+2;\check t}^w+\tau_{n+1;\check t}^w+\tau_{n;\check t}^w)+3\leq T_{\operatorname{cov}}\frac{\check c}{\check \kappa}(\widehat C_3^2+\widehat C_3+1)\tau_{n;\check t}^w+3\leq C_{3} \tau_{n;\check t}^w,
	\]
	as claimed.  This completes the proof.
\end{proof}

\begin{lemma}\label{lem:est_W}
	Suppose that Assumptions \ref{as:MFC}\;(i),(ii) and \ref{as:Qlearning} are satisfied. Recall $D_1=2\check Q_{\operatorname{max}}$ (see Definition \ref{dfn:epoch}\;(iii)) and let $\underline{D}\in (0,D_1]$ be such~that %
	\begin{align}\label{eq:condi_a(D)}
		2a(\underline D)\log(a(\underline D))>1,\quad\mbox{where }\; a(\underline D):=a_{(\check c,\check \kappa,\check \delta,w,T_{\operatorname{cov}},\beta)}(\underline D):=\frac{2^6C_{3}(\beta \check Q_{\operatorname{max}}T_{\operatorname{cov}}^{w})^2}{({\check \delta} (2+\check \delta)^{-1} \check \beta \underline D)^2}>0,
	\end{align}
	with $\check \delta\in(0,e-2)$ and $C_{3}>0$ given in Lemma \ref{lem:Y_bound} and Lemma~\ref{lem:azuma_W}, respectively.  %
	Moreover, assume that the initial epoch's length $\check t$ in Definition \ref{dfn:epoch} satisfies \[
		\check t\geq \Big\lceil \big(2a(\underline D)\log( a(\underline D))\big)^{\frac{1}{w}}\Big\rceil.\]
	Then for any $\underline n\geq 1$ such that $D_{\underline n}\geq \underline{D}$, the following holds:
	\begin{align*}
		\begin{aligned}
			\widehat{\mathbb{P}}^0\bigg(\bigcap_{n=1}^{\underline n}\bigcap_{t=\tau_{n+1;\check t}}^{\tau_{n+3;\check t}-1}\bigcap_{(\check \mu,\check \pi)\in \check S\times \check \Pi}\big\{|W_{t;n}(\check \mu,\check \pi)| < {\check \delta} (2+\check \delta)^{-1} \check \beta D_n\big\} \bigg)%
			\geq  1- \frac{2|\check S||\check A|^{|\check S|}C_{3} \underline n}{\exp({\check t^w}/{a(\underline D)})}.
		\end{aligned}
	\end{align*}
\end{lemma}

\begin{proof} Let $\underline n\in \mathbb{N}$ be such that $D_{\underline n}\geq \underline{D}$. Then we have by Lemma \ref{lem:azuma_W} that
	\begin{align*}%
		\begin{aligned}
			 & \widehat{\mathbb{P}}^0\bigg(\bigcap_{n=1}^{\underline n}\bigcap_{t=\tau_{n+1;\check t}}^{\tau_{n+3;\check t}-1}\bigcap_{(\check \mu,\check \pi)\in \check S\times \check \Pi}\big\{|W_{t;n}(\check \mu,\check \pi)| < {\check \delta} (2+\check \delta)^{-1} \check \beta D_n\big\} \bigg) \\
			 & \quad \geq 1- \sum_{n=1}^{\underline n}  \sum_{t=\tau_{n+1;\check t}}^{\tau_{n+3;\check t}-1} \sum_{(\check \mu,\check \pi)\in \check S\times \check \Pi}\widehat{\mathbb{P}}^0(|W_{t;n}(\check \mu,\check \pi)| \geq {\check \delta} (2+\check \delta)^{-1} \check \beta D_n)             \\
			 & \quad \geq 1- 2|\check S||\check \Pi|\sum_{n=1}^{\underline n}  \bigg(T_{\operatorname{cov}}\frac{\check c}{\check \kappa } (\tau_{n+2;\check t}^w+\tau_{n+1;\check t}^w)+2\bigg)\exp(-{2 \tau_{n;\check t}^w}/{a(\underline D)})=:\operatorname{I},
		\end{aligned}
	\end{align*}
	where the last inequality holds because $D_n\geq \underline D$ for all $0\leq n\leq \underline n$.

	Moreover, using the estimate in \eqref{eq:tau_grow} and the constant $C_3>0$ in Lemma \ref{lem:azuma_W}, we have
	\begin{align*}%
		\operatorname{I}\geq
		1- 2|\check S||\check \Pi|C_{3} \sum_{n=1}^{\underline n}   \tau_{n;\check t}^w\exp(-{2\tau_{n;\check t}^w}/{a(\underline D)})=: \operatorname{II}.
	\end{align*}

	Moreover, since \(2a(\underline D)\log(a(\underline D))>1\) and \(
	\check t\geq \lceil (2a(\underline D)\log( a(\underline D)))^{\frac{1}{w}}\rceil\) by assumption, %
	we apply Lemma \ref{lem:elem_est}\;(ii) to have for every $n=1,\cdots,\underline n$,
	\begin{align*}
		\tau_{n;\check t}^w\exp(-{2\tau_{n;\check t}^w}/{a(\underline D)})\leq  \exp(-{\tau_{n;\check t}^w}/{a(\underline D)})\leq   \exp(-{\check t^w}/{a(\underline D)}).
	\end{align*}

	Thus, we conclude the proof by noting that
	$|\check \Pi|=|\check A|^{|\check S|}$.
\end{proof}

Finally, we present the proof of the finite-time iteration bound analysis for $Q$-learning algorithm. %

\begin{proof}[Proof of Theorem \ref{thm:asynchro_rate}]
	We first consider the asynchronous version of the $Q$-learning algorithm in Framework~\ref{frame:Qlearning}~(i),(iii). The proof consists of four steps, which rely on the following preliminary settings.

	Let $\widehat \varepsilon>0$ and $\widehat \delta \in(0,1)$ be given. We first set
	\begin{align}\label{eq:optimal_n*}
		n^*:= \max\{\lceil {\check \beta}^{-1}{\log(2\check Q_{\operatorname{max}}\widehat \varepsilon^{-1})}\rceil,1\}.
	\end{align}

	Recalling $\check t_{1}^*\geq 1$ and $\check c_1^*>0$ introduced in~\eqref{eq:tau1_condi1}, we then define
	\begin{align}\label{eq:optimal_checkc}
		\check c:= \max\{\check c_1^*,1\}.
	\end{align}

	Moreover, we define
	\begin{align}\label{eq:perturb_init}
		\check t^*_2:= \max\bigg\{\bigg\lceil
		\bigg(\frac{2T_{\rm cov}\frac{\check c}{\check\kappa}w(n^*-1)}{\log 2}\bigg)^{\frac 1{1-w}}\bigg\rceil,\bigg\lceil
		\bigg(\frac{2n^*+1}{2T_{\rm cov}\frac{\check c}{\check\kappa}}\bigg)^{\frac1w}\bigg\rceil,1
		\bigg\}.
	\end{align}
	Then by Lemma \ref{lem:perturb_init}, we have for any $t\geq \check t_2^*$
	\begin{align}\label{eq:overlap_n*}
		\tau_{n^*+2;\check t}-1\geq \tau_{n^*;\check t +1}.
	\end{align}

	In what follows, we construct the (least) initial epoch length $\check t^*\geq 1$ such that all conditions imposed in Lemmas \ref{lem:separate_est}, \ref{lem:Y_bound}, and \ref{lem:est_W} hold, while \eqref{eq:overlap_n*} holds for all $\check t\geq \check t^*$.

	To this end, we let $\sigma^*\in(0,1)$ be so that %
	\begin{align}\label{eq:condi_sigma}
		D_{n^*}< \widehat \varepsilon,\quad D_{n^*}\geq \sigma^* \widehat \varepsilon,\quad \mbox{and}\quad 2a(\sigma^* \widehat\varepsilon)\log(a(\sigma^* \widehat\varepsilon))>1,
	\end{align}
	where $a(\sigma^* \widehat\varepsilon)>0$ is defined as in \eqref{eq:condi_a(D)} (with replacing $\underline D$ by $\sigma^* \widehat\varepsilon$ therein), we define
	$\check t^*\geq 1$ by
	\begin{align}\label{eq:optimal_tau1}
		\check t^*:=\max\{\check t_{1}^*,\lceil{T_{\operatorname{cov}}}({T_{\operatorname{cov}}-1})^{-1}\rceil,\check t_{2}^*,\check t_{3}^*,\check t_{4}^*\},
	\end{align}
	where $\check t_{1}^*$ and $\check t_2^*$ are given in \eqref{eq:tau1_condi1} and \eqref{eq:perturb_init}, respectively, $\check t_3^*$ and $\check t_4^*$ are given by
	\begin{align}\label{eq:tau1_final}
		\begin{aligned}
			 & \check t_{3}^*:=\big\lceil \big(2a(\sigma^* \widehat\varepsilon)\log( a(\sigma^* \widehat\varepsilon))\big)^{\frac{1}{w}}\big\rceil,                                                     \\
			 & \check t_{4}^*:=\bigg\lceil \bigg(a(\sigma^*\widehat \varepsilon)\log\Big( \frac{2 |\check S||\check A|^{|\check S|}C_{3} n^* }{\widehat \delta} \Big) \bigg)^{\frac{1}{w}} \bigg\rceil,
		\end{aligned}
	\end{align}
	with the constant $C_3>0$ given in Lemma \ref{lem:azuma_W}, which depends on the constant $\check c\geq 1$ in~\eqref{eq:optimal_checkc}.

	Last, define, for every $\check t\geq \check t^*$ and $k\geq 1$,
	\begin{align*}
		\begin{aligned}
			\operatorname{I}_k(\check t) & :=\bigcap_{t=\tau_{k+1;\check t}}^{\tau_{k+3;\check t}-1}\big\{\|\check \Delta_t\|_{\check S\times \check \Pi} < D_{k+1}\big\},                                                                                                      \\
			\operatorname{J}_k(\check t) & :=\bigcap_{t=\tau_{k+1;\check t}}^{\tau_{k+3;\check t}-1}\bigcap_{(\check \mu,\check\pi)\in\check S\times\check \Pi }^{}\big\{|W_{t;k}(\check \mu,\check \pi)|\leq {\check \delta}({2+\check \delta})^{-1} \check \beta D_{k}\big\}.
		\end{aligned}
	\end{align*}

	\noindent {\it Step 1.} We claim that for any $\check t\geq \check t^*$ and  $n\geq 1$,
	\(
	\cap_{k=1}^{n} \operatorname{J}_k(\check t)\subseteq \cap_{k=1}^{n} \operatorname{I}_k(\check t).
	\)

	Let $\check t\geq \check t^*$ be given. We first show that \(
	\operatorname{J}_1(\check t)\subseteq \operatorname{I}_1(\check t),
	\)
	i.e., the case $n=1$. By Lemma \ref{lem:liftedoptimalQ} and Lemma~\ref{lem:Q_est}\;(iii), we have that for every $t\in \{\tau_{1;\check t},\dots,\tau_{3;\check t}-1\}$,
	\[
		\|\check \Delta_t\|_{ \check S\times \check \Pi}%
		\leq \|\check Q_t\|_{ \check S\times \check \Pi}+\|\check Q^*\|_{ \check S\times \check \Pi}\leq D_1=2\check Q_{\operatorname{max}}.
	\]

	Moreover, since $\check t\geq \{\check t_1^*,\lceil{T_{\operatorname{cov}}}({T_{\operatorname{cov}}-1})^{-1}\rceil\}$ and $\check c= \max\{\check c_1,1\}$ (see \eqref{eq:optimal_checkc}, \eqref{eq:optimal_tau1}), we can apply Lemma \ref{lem:separate_est} and Lemma \ref{lem:Y_bound} to have for every $t\geq \tau_{2;\check t}$ and $(\check \mu,\check \pi)\in \check S\times \check \Pi$
	\begin{align*}%
		|\check\Delta_t(\check \mu,\check \pi)|\leq Y_{t;1}(\check \mu,\check \pi)+|W_{t;1}(\check \mu,\check \pi)|
		\leq (\beta+{2}\check \beta(2+\check \delta)^{-1})D_1 + |W_{t;1}(\check \mu,\check \pi)|,
	\end{align*}
	where the first inequality holds because $(Y_{t;1})_{t\geq \tau_{1;\check t}}$ is nonnegative (see \eqref{eq:YW_seq}).

	Moreover, since $(\beta+{2}\check \beta(2+\check \delta)^{-1})D_1+{\check \delta}({2+\check \delta})^{-1} \check \beta D_{1}=(1-\check \beta) D_1=D_2$ (see Definition \ref{dfn:epoch}), we have \(
	\operatorname{J}_1(\check t)\subseteq \operatorname{I}_1(\check t),
	\)
	as claimed.

	Inductively, on ${\operatorname{I}}_1(\check t)$, we have  $\|\check \Delta_t\|_{ \check S\times \check \Pi}\leq D_2$ for every $t\in \{\tau_{2;\check t},\dots,\tau_{4;\check t}-1\}$. Repeating the arguments for the case $n=1$ with $(D_2;\{\tau_{2;\check t},\dots,\tau_{4;\check t}-1\})$ in place of $(D_1;\{\tau_{1;\check t},\dots,\tau_{3;\check t}-1\})$, we obtain
	\(
	\operatorname{J}_2(\check t)\subseteq   \operatorname{I}_2(\check t)\) on ${\operatorname{I}}_1(\check t)$,
	and hence $\operatorname{J}_2(\check t)\cap\, {\operatorname{I}}_1(\check t)\subseteq \operatorname{I}_1(\check t) \cap \,\operatorname{I}_2(\check t)$. %
	Combining this with \(\operatorname{J}_1(\check t)\subseteq \operatorname{I}_1(\check t)\), we have
	\(
	{\operatorname{J}}_1(\check t)\cap\,\operatorname{J}_2 (\check t)\subseteq \operatorname{I}_1(\check t)\cap\,\operatorname{I}_2(\check t).%
	\)

	By using the same arguments presented for the case $n=2$ inductively, we have the desired inclusion property for any $n\geq 1$. Moreover, since $\check t\geq \check t^*$ is arbitrary, the claim follows.

	\noindent {\it Step 2.} Recalling that  $D_{n^*}\leq \widehat \varepsilon$ (by definition of $n^*$ in \eqref{eq:optimal_n*}), we have by Step 1 that for any~$\check t\geq \check t^*$
	\begin{align}
		\begin{aligned}
			\widehat{\mathbb{P}}^0\bigg(\bigcap_{t=\tau_{n^*;\check t}}^{\tau_{n^*+2;\check t}-1}\big\{\|\check \Delta_t\|_{\check S\times \check \Pi} < \widehat \varepsilon\big\}\bigg) & \geq  \widehat{\mathbb{P}}^0\big( \operatorname{I}_{n^*-1}(\check t )\big)                                                                                                                                                    \\
			                                                                                                                                                                              & \geq \widehat{\mathbb{P}}^0\bigg(\bigcap_{k=1}^{n^*-1}\operatorname{I}_{k}(\check t) \bigg)\geq \widehat{\mathbb{P}}^0\bigg(\bigcap_{k=1}^{n^*-1}\operatorname{J}_{k}(\check t) \bigg)=:\widehat{\operatorname{P}}(\check t).
		\end{aligned}
	\end{align}

	Since $D_{n}\geq \sigma^* \widehat\varepsilon$ for any $n=1,\dots,n^*$, and $2a(\sigma^* \widehat\varepsilon)\log(a(\sigma^* \widehat\varepsilon))>1$ (see \eqref{eq:condi_sigma}) and $\check t\geq \check t_{3}^*$ (see \eqref{eq:optimal_tau1}, \eqref{eq:tau1_final}), we can apply Lemma \ref{lem:est_W} to have, for every~$\check t\geq \check t^*$,
	\begin{align*}
		\widehat{\operatorname{P}}(\check t)\geq   1- \frac{2|\check S||\check A|^{|\check S|}C_{3} n^*}{\exp({\check t^w}/{a(\sigma^*\widehat \varepsilon)})}
		\geq 1-\widehat \delta, %
	\end{align*}
	where the last inequality follows from $\check t\geq \check t_{4}^*$; see \eqref{eq:optimal_tau1}, \eqref{eq:tau1_final}.

	\noindent{\it Step 3.} Using $n^*$ and $\check t^*$ (see \eqref{eq:optimal_n*}, \eqref{eq:optimal_tau1}), we define $T^*:=\tau_{n^*;\check t^*}$.

	Since \eqref{eq:overlap_n*} holds for all $t\geq \check t^*$ (by definition of $\check t^*$ in \eqref{eq:optimal_tau1}), it holds that for all $m\geq 0$
	\[
		\{\tau_{n^*;\check t^*+m},\dots,\tau_{n^*+2;\check t^*+m}-1\}\cap \{\tau_{n^*;\check t^*+m+1},\dots,\tau_{n^*+2;\check t^*+m+1}-1\}\neq \emptyset.
	\]
	Hence, for any $T\geq T^*$, there exists $m_T\geq 0$ such that
	\(
	T\in \{\tau_{n^*;\check t^*+m_T},\dots,\tau_{n^*+2;\check t^*+m_T}-1\}.
	\)

	From this, we have by Step 2 and Step 3 that for any $T\geq T^*$,
	\begin{align}\label{eq:Step3_asynchro}
		\begin{aligned}
			\widehat{\mathbb{P}}^0(\|\check Q_{T}-\check Q^*  \|_{\check S\times \check \Pi} \leq \widehat \varepsilon )=\widehat{\mathbb{P}}^0(\|\check \Delta_{T} \|_{\check S\times \check \Pi}  \leq \widehat \varepsilon ) & \geq \widehat{\mathbb{P}}^0\bigg(\bigcap_{t=\tau_{n^*;\check t^*+m_T}}^{\tau_{n^*+2;\check t^*+m_T}-1}\big\{\|\check \Delta_t\|_{\check S\times \check \Pi} < \widehat \varepsilon\big\}\bigg) \\
			                                                                                                                                                                                                                    & \geq \widehat{\operatorname{P}}(\check t^*+m_T) \geq 1-\widehat \delta.
		\end{aligned}
	\end{align}

	Hence, the desired result in \eqref{eq:Qconverge2} follows from combining this with the discretization error estimates given in~\eqref{eq:prj_error}.

	\noindent {\it Step 4.} It remains to compute the order of $T^*=\tau_{n^*;\check t^*}$. To this end, we note that\footnote{\label{footnote:bigtheta}For any  nonnegative functions $f,g:\mathbb{N}\to \mathbb{R}$, we write $f(n)=\Theta(g(n))$ if and only if there exist some $c,C>0$ and $\overline n\in \mathbb{N}$ such that $cg(n)\leq f(n)\leq Cg(n)$ for all $n\geq \overline{n}$.} \(
	\check t_{1}^*=\Theta(T_{\operatorname{cov}})
	\)
	and $\check c_{1}^*= \Theta(1)$ (see \eqref{eq:tau1_condi1}), hence $\check{c}$ in \eqref{eq:optimal_checkc} satisfies $\check{c}=\Theta(1)$.

	Recall that $a(\sigma^*\widehat \varepsilon)$ is given in \eqref{eq:condi_a(D)} (with replacing $\underline D$ as $\sigma^*\widehat \varepsilon$ therein), $C_3>0$ is given in Lemma~\ref{lem:azuma_W}, and $\check Q_{\operatorname{max}}=\frac{C_r}{1-3\beta}$ and $\check \beta = \frac{1-\beta}{2}$. Since $\beta \in[0,\frac{1}{3}\wedge (2C_{\operatorname{F}})^{-1})$ (by Assumption~\ref{as:Qlearning}\;(i)), %
	we have that
	\(
	a(\sigma^*\widehat \varepsilon)=\Theta(\frac{\check N_1}{\widehat \varepsilon^2})\) and  \(C_3=\Theta(T_{\operatorname{cov}}^{1+3w}), %
	\)
	where $\check N_1$ is given in \eqref{eq:order_parameter}.

	Moreover, since $n^*$ in \eqref{eq:optimal_n*} satisfies $n^*=\Theta(\log(\frac{\check Q_{\operatorname{max}}}{\widehat \varepsilon}))$,  $\check t^*$ in \eqref{eq:optimal_tau1} can be expressed in terms of the algorithmic parameters as an order of
	\[
		\check t^*= O\bigg(\bigg(\frac{\check N_1}{\widehat \varepsilon^2} \max \bigg\{\log\bigg(\frac{\check N_2}{\widehat{\varepsilon}}\bigg),\log\bigg(\frac{\check N_3}{ \widehat \delta}\log \bigg(\frac{\check Q_{\operatorname{max}}}{\widehat \varepsilon}\bigg)\bigg)  \bigg\}\bigg)^{\frac{1}{w}}+\bigg(T_{\operatorname{cov}}\log\bigg(\frac{\check Q_{\operatorname{max}}}{\widehat \varepsilon}\bigg)\bigg)^{\frac{1}{1-w}}\bigg),
	\]
	where $\check N_2$ and $\check N_3$ are given in \eqref{eq:order_parameter}.%

	From this and Lemma \ref{lem:time_check0}\;(ii), we have the order of the iteration number $T^*$ given~by
	\begin{align*}
		T^*=\tau_{n^*;\check t^*}= O(\check t^*+(T_{\operatorname{cov}}n^*)^{\frac{1}{1-w}})=O(\check t^*),
	\end{align*}
	as desired.

	Therefore, the proof of the asynchronous version of the $Q$-learning algorithm is complete.

	Finally, for the synchronous version of the $Q$-learning algorithm under Framework~\ref{frame:Qlearning}\,(i),(ii), the construction of the iteration number $T^*\in\mathbb{N}$ is analogous to that for the asynchronous version but involves much simpler calculations, since the covering time condition in Assumption~\ref{as:Qlearning}\,(ii) is no longer required. In particular, %
	for any $\check t\geq 1$,
	the sequence $(\tau_{n;\check t})_{n\ge0}$ in Definition~\ref{dfn:epoch} can be redefined as
	\[
		\tau_{n+1;\check t}:=\tau_{n;\check t}+\bigg\lceil \frac{\check c}{\check\kappa}\tau_{n;\check t}^{w}\bigg\rceil,
		\quad n\ge0,\quad \mbox{with $\tau_{1;\check t}:=\check t$}.
	\]
	Moreover, in the resulting constants---such as $\check t_{1}^*$ and $\check c_1^*$ in \eqref{eq:tau1_condi1}, $\check t_2^*$ in \eqref{eq:perturb_init}, $C_3$ in \eqref{lem:azuma_W}, and $a(\underline D)$ in \eqref{eq:condi_a(D)}---the parameter $T_{\operatorname{cov}}$ can be taken as $1$ so that the results in Lemmas~\ref{lem:separate_est}, \ref{lem:Y_bound}, and \ref{lem:est_W} still hold, while \eqref{eq:overlap_n*} holds for all $\check t\geq \check t^*$. As a consequence, all the arguments presented in the main proof for the asynchronous case also apply to the synchronous version. This completes the proof.
\end{proof}

\subsection{Proof of Corollary \ref{cor:Qlearningrate}}\label{sec:proof:cor:Qlearningrate}
We only present the proof for the asynchronous case, since the synchronous case can be proven analogously following the same argument.
\begin{proof}
	Let $\widehat\delta\in(0,1)$ be given.
	By the discretization error estimate in \eqref{eq:prj_error}, it suffices to show that there exist some $ C_{\widehat \delta}>0$ and $T_{\widehat \delta}\geq 1$ such that
	\begin{align}\label{eq:est_claim}
		\widehat{\mathbb{P}}^0\bigg(\|\check Q_T-\check Q^*\|_{\check S\times \check \Pi}\leq C_{\widehat \delta} \sqrt{\frac{\log T}{T^w}} \bigg)>1-\widehat \delta\quad \mbox{for any $T\geq  T_{\widehat \delta}$}.
	\end{align}

	To this end, for every \(\widehat\varepsilon\in(0,1)\), we denote by
	\(T^*(\widehat\varepsilon,\widehat\delta)\) the iteration number \(T^*\) in Theorem~\ref{thm:asynchro_rate}\;(ii) under the prescribed error level \(\widehat\varepsilon\) and confidence level \(\widehat\delta\). We recall from the proof of Theorem~\ref{thm:asynchro_rate}\;(ii), particularly Step~3, that
	\(
	T^*(\widehat\varepsilon,\widehat\delta)
	=T^*=
	\tau_{n^*;\check t^*},
	\)
	where \(n^*\) and \(\check t^*\), given in
	\eqref{eq:optimal_n*} and \eqref{eq:optimal_tau1}, respectively,
	satisfy the following estimates: there exist constants
	\(K_{\widehat\delta,1}>0\) and
	\(\varepsilon_{\widehat\delta,1}\in(0,1)\), where
	\(K_{\widehat\delta,1}\) depends on the parameters appearing in
	\eqref{eq:order_parameter} and the fixed confidence level
	\(\widehat\delta\), but does not depend on \(\widehat\varepsilon\),
	such that for every
	\(\widehat\varepsilon\in(0,\varepsilon_{\widehat\delta,1})\),
	\[
		n^*
		\leq
		K_{\widehat \delta,1}\log(\widehat\varepsilon^{-1})\quad \mbox{and}\quad \check t^*
		\leq K_{\widehat \delta,1} \big[\widehat\varepsilon^{-2}
			\log(\widehat\varepsilon^{-1})\big]^{1/w}.
	\]

	Moreover, by Lemma \ref{lem:time_check0} there exists some
	\(K_{\widehat\delta}>0\) (depending on $K_{\widehat\delta,1}$ but not on \(\widehat\varepsilon\)) such that
	\begin{align}\label{eq:est_rate0}
		T^*(\widehat\varepsilon,\widehat\delta)
		\leq
		K_{\widehat\delta}
		\big[
			\widehat\varepsilon^{-2}
			\log(\widehat\varepsilon^{-1})
			\big]^{1/w}\quad \mbox{for all \(\widehat\varepsilon\in(0,\varepsilon_{\widehat\delta,1})\)}.
	\end{align}

	Now, we let $C_{\widehat \delta}:= K_{\widehat\delta}^{{w}/{2}}$ and define for every $T\geq 2$ by
	\begin{align}\label{eq:def_ET}
		E_{\widehat \delta}(T)
		:=
		C_{\widehat\delta}
		\sqrt{\frac{\log T}{T^w}}.
	\end{align}
	Since the map $[e^{1/w},\infty):T\mapsto \sqrt{\frac{\log T}{T^w}}$ is decreasing, we can choose $T_{\widehat \delta, 1}\geq \max\{2,e^{1/w}\}$ such that
	\begin{align}\label{eq:est_rate1}
		E_{\widehat \delta}(T)\leq \varepsilon_{\widehat \delta, 1}\quad \mbox{for all $T\geq T_{\widehat \delta, 1}$.}
	\end{align}

	Moreover, since
	\(
	\log(E_{\widehat \delta}(T)^{-1})
	=
	\frac{w}{2}\log T-\frac12\log\log T-\log C_{\widehat\delta}\) for all $T>2$,
	there exists some  $T_{\widehat\delta,2}\geq2$ such~that
	\begin{align}\label{eq:est_rate2}
		\log(E_{\widehat \delta}(T)^{-1})
		\leq
		\log T\quad \mbox{for all
			$T\geq T_{\widehat\delta,2}$.}
	\end{align}

	Last, we let $T_{\widehat\delta}:= \max\{T_{\widehat\delta,1},T_{\widehat\delta,2}\}$. Then, we have by \eqref{eq:est_rate0}, \eqref{eq:est_rate1}, \eqref{eq:est_rate2} that for every~$T\geq T_{\widehat\delta}$,
	\begin{align}\label{eq:bound_T}
		\begin{aligned}
			T^*(E_{\widehat \delta}(T),\widehat\delta) & \leq K_{\widehat\delta}
			\big[
				E_{\widehat \delta}(T)^{-2}
				\log\big(E_{\widehat \delta}(T)^{-1}\big)
			\big]^{1/w}                                                          \\
			                                           & = K_{\widehat\delta}
			\bigg[C_{\widehat\delta}^{-2} \frac{T^w}{\log T}
				\log\big(E_{\widehat \delta}(T)^{-1}\big)
				\bigg]^{1/w}\leq
			K_{\widehat\delta}
			\left[
				C_{\widehat\delta}^{-2}
				T^w
				\right]^{1/w}
			=T,
		\end{aligned}
	\end{align}
	where the first equality follows from definition of $E_{\widehat \delta}(T)$ and the last equality follows from the setting that $C_{\widehat \delta}= K_{\widehat\delta}^{{w}/{2}}$.

	Then, for each \(T \geq T_{\widehat{\delta}}\), the estimate \eqref{eq:Step3_asynchro} (which follows from the Step 2 and Step 3 in the proof of Theorem \ref{thm:asynchro_rate}) can be applied with the prescribed error level
	\(\widehat{\varepsilon}:=E_{\widehat{\delta}}(T)\) (see \eqref{eq:def_ET})
	and confidence level \(\widehat{\delta}\), which leads to
	\[
		\widehat{\mathbb{P}}^0\Big(
		\|\check Q_{\widetilde T}-\check Q^*\|_{\check S\times \check \Pi}
		\leq E_{\widehat \delta}(T)
		\Big)
		>
		1-\widehat \delta
		\quad
		\mbox{for any $\widetilde T\geq T^*(E_{\widehat \delta}(T),\widehat\delta).$}
	\]
	Since \(T\geq T^*(E_{\widehat \delta}(T),\widehat\delta)\) for every $T\geq T_{\widehat \delta}$ by \eqref{eq:bound_T}, the desired estimate \eqref{eq:est_claim} follows.

	This completes the proof.
\end{proof}

\bibliographystyle{abbrv}
\bibliography{references}

\end{document}